\documentclass[11pt]{amsart}
\usepackage {amsmath, amssymb, a4wide, epsfig, enumerate, psfrag}
\usepackage[latin1]{inputenc}
\usepackage[english]{babel}
\date{\today}

\author{Bertrand Deroin \and Romain Dujardin}
\thanks{B.D.'s research was partially supported by ANR-08-JCJC-0130-01,  ANR-09-BLAN-0116. R.D.'s research  was partially supported by ANR project BERKO and by ECOS project C07E01}
\address{CNRS \\ Département de Mathématique d'Orsay \\ B\^atiment 425,
Universit\'e de Paris Sud,
91405 Orsay cedex, France.}
\email{bertrand.deroin@math.u-psud.fr}
\address{CMLS \\ \'Ecole Polytechnique \\ 91128 Palaiseau\\
         France}
\email{dujardin@math.polytechnique.fr}

\title{Random walks, Kleinian groups, and bifurcation currents}

\newcommand{\cc}{\mathbb{C}}

\newcommand{\re}{\mathbb{R}}

\newcommand{\zz}{\mathbb{Z}}
\newcommand{\nn}{\mathbb{N}}
\newcommand{\pp}{\mathbb{P}}
\newcommand{\e}{\varepsilon}
\newcommand{\cv}{\rightarrow}

\newcommand{\om}{\Omega}
\newcommand{\set}[1]{\left\{#1\right\}}
\newcommand{\norm}[1]{\left\Vert#1\right\Vert}
\newcommand{\abs}[1]{\left\vert#1\right\vert}
\newcommand{\cd}{{\cc^2}}

\newcommand{\pu}{{\mathbb{P}^1}}
\newcommand{\rest}[1]{ \arrowvert_{#1}}
\newcommand{\m}{{\bf M}}
\newcommand{\unsur}[1]{\frac{1}{#1}}

\newcommand{\lrpar}[1]{\left(#1\right)}
\newcommand{\bra}[1]{\left\langle #1\right\rangle}

\newcommand{\la}{\lambda}
\newcommand{\La}{\Lambda}
\newcommand{\SL}{\mathrm{SL}(2,\mathbb C)}
\newcommand{\PSL}{\mathrm{PSL}(2,\mathbb C)}
\newcommand{\tbif}{T_{\mathrm{bif}}}

\newcommand{\note}[1]{\marginpar{\tiny #1}}
\newcommand{\itm}{\item[-]}

\DeclareMathOperator{\supp}{Supp}
\DeclareMathOperator{\vol}{vol}

\DeclareMathOperator{\tr}{tr}

\DeclareMathOperator{\fix}{Fix}
\DeclareMathOperator{\length}{length}
\DeclareMathOperator{\codim}{codim}

\newtheorem{prop} {Proposition} [section]
\newtheorem{thm}[prop] {Theorem}
\newtheorem{defi}[prop] {Definition}
\newtheorem{lem}[prop] {Lemma}
\newtheorem{cor}[prop]{Corollary}
\newtheorem{theo}{Theorem}

\theoremstyle{remark}

\newtheorem{rmk}[prop]{Remark}

\begin{document}

\begin{abstract} Let $(\rho_\lambda)_{\lambda\in \Lambda}$  be a holomorphic family of representations of a finitely generated group $G$ into $\mathrm{PSL(2, \mathbb{C})}$, parameterized by a  complex manifold $\Lambda$. We define a notion of  {\em bifurcation current} in this context, that is, a positive closed current  on $\Lambda$     describing  the bifurcations of this family of representations in a quantitative sense. It is the analogue of the  bifurcation  current introduced by DeMarco for   holomorphic families of rational mappings on $\mathbb{P}^1$. 
Our definition relies on the theory of random products of matrices, so it depends on the 
choice of  a probability measure $\mu$ on $G$. 

We show that under natural assumptions on $\mu$, the support of the bifurcation current coincides with the bifurcation locus of the family. We also prove that the bifurcation current describes the asymptotic distribution of several codimension 1 phenomena in parameter space, like accidental parabolics or new relations, or accidental collisions between fixed points.
\end{abstract}

\maketitle

\section{Introduction}

In recent years, the use of techniques from higher dimensional holomorphic dynamics, especially positive currents, has led to interesting new insights on the structure of parameter spaces of holomorphic dynamical systems on the Riemann sphere.
There is a deep and fruitful analogy --first brought to light by Sullivan \cite{sullivan1}-- between holomorphic dynamics  on  $\pu$ and the theory of Kleinian groups. Our purpose in this paper is to develop these ideas on the Kleinian group side, 
by initiating   the study of \textit{bifurcation currents} in this setting.

\bigskip

Let us first briefly discuss this notion in the context of rational mappings.  Let $\La$ be a complex manifold and $f = (f_\la)_{\la\in \La} :\Lambda\times\pu\cv\pu$ be a holomorphic family of rational maps of fixed degree.  The simplest way to define a positive closed current on $\La$ associated to this family is, following DeMarco \cite{demarco1, demarco2}, to observe that the Lyapunov exponent $\chi(f_\la)$ of $f_\la$ relative to its unique measure of maximal entropy defines a plurisubharmonic (psh for short) function on $\Lambda$. We then put $T_{\rm bif} = dd^c(\chi(f_\la))$. This is by definition the {\em bifurcation current} of the family.
In the most studied family  $(z^2+\la)_{\la\in \cc}$ of quadratic polynomials, $T_{\rm bif}$ is just the harmonic measure of the Mandelbrot set.

DeMarco proved that $\supp(T_{\rm bif})$ is exactly the bifurcation locus $\mathrm{Bif}$ (defined, e.g., as the locus of parameters where the Julia set does not move continuously in the Hausdorff topology). A  word about the proof: the inclusion $\supp(T_{\rm bif})\subset \mathrm{Bif}$ is obvious, while    the converse inclusion is based on a formula for  the Lyapunov exponent in terms of the dynamics of critical points. It follows in particular that if the Lyapunov exponent is pluriharmonic in some region of parameter space, then the critical points cannot bifurcate there. Standard arguments then imply that the dynamics is stable.

\medskip

A remarkable feature of the bifurcation current is that it describes the asymptotic distribution of natural sequences of dynamically defined subvarieties of parameter space. For instance it was shown by Favre and the second author in \cite{df} that  the hypersurfaces of parameters  possessing a  preperiodic critical point (of preperiod $n$ tending to infinity)  equidistribute towards $T_{\rm bif}$.
Another result, due to Bassanelli and Berteloot,  asserts that   parameters admitting a periodic point of period $n$ and a given multiplier typically equidistribute towards $\tbif$ \cite{bas-ber1, bas-ber2}.
% Another theme of study,  pursued in a number of papers \cite{bas-ber, bas-ber1, bas-ber2, df, cubic, buff-epstein, gauthier}  is the description of the exterior powers of $T_{\rm bif}$, with again some  emphasis   on  the  characterization of  their  supports and equidistribution theorems. The underlying ideology is that $\supp(T_\mathrm{bif}^k)$, for $1\leq k\leq \dim(\La)$ should define a dynamically meaningful filtration of the bifurcation locus.  This turns out to be rather successful for instance when $\La$ is the space of    polynomials (resp. rational maps) of degree $d$.
%  

\bigskip

Let us now turn to the subject of the paper.
Let $\La$ be a (connected) complex manifold, $G$ be a finitely generated group and $\rho = (\rho_\la)_{\la\in\La}:\La\times G\cv \PSL$ be a holomorphic family of non-elementary representations of $G$. To avoid trivialities, assume further   that the representations $\rho_\la$ are faithful at generic parameters, and  not all conjugate   to each other (in $\PSL$).
For such a family of M\"obius subgroups, there is a well-established  notion of bifurcation, mainly due to Sullivan
\cite{sullivan} (see also \cite{bers}), defined (by the negative) by saying that an open set $U$ is contained in the stability locus\footnote{One should be careful not to be confused with the notion of a stable representation in the sense of geometric invariant theory. In this paper, stability will always be understood in the sense of dynamical systems.} if for every $g\in G$, the holomorphic family of M\"obius transformations  $\rho_\la(g)$ is of constant type (loxodromic, parabolic or elliptic) throughout $U$. In particular the fixed points $\rho_\la(g)$ can be followed holomorphically over $U$.  Using the theory of  holomorphic motions, Sullivan proved that representations in a stable family are quasi-conformally conjugate on $\pu$.
He also proved that these representations must then be discrete with a non-empty set of discontinuity.
 
% To associate a bifurcation current to such a family we use the theory of random products of matrices (Bougerol-Lacroix \cite{bougerol-lacroix} and Furman \cite{furman} are good introductory texts on this topic). For this, we choose\footnote{Natural choices can be made in certain situations, see  below.}  a measure $\mu$ on $G$, such that $\supp(\mu)$ generates $G$ as a semi-group, and, for simplicity, suppose in this introduction that $\supp(\mu)$ is finite (the exact moment conditions required on $\mu$ will be made clear in the text).  Given a representation $\rho$ of $G$ into $\PSL$ we can now define a Lyapunov exponent by the formula $$\chi(\rho) = \lim_{n\cv\infty} \unsur{n}\int \log\norm{\rho(g_n\cdots g_1)} d\mu(g_1)\cdots d\mu(g_n) =
% \lim_{n\cv\infty} \unsur{n}\int \log\norm{\rho(g)} d\mu^n(g).$$ Here $\norm{\cdot}$ refers to any matrix norm on $\PSL$, and $\mu^n$ denotes the $n^{\rm th}$ convolution power of $\mu$. A fundamental result of Furstenberg's asserts that if $\rho$ is non-elementary, then $\chi(\rho)$ is a positive number.

To associate a bifurcation current to such a family we use the theory of random walks on groups. For this, we choose
a probability measure $\mu$ on $G$ (satisfying certain natural technical assumptions that will be made clear in the text), and consider the   random walk on $G$ whose transition probabilities are given by $\mu$. For the sake of simplicity we may suppose in this introduction that  $\mu$ is  equidistributed on a finite symmetric set of generators of $G$, in which case we are just considering the simple random walk on the associated Cayley graph. 

 Given a representation $\rho$ of $G$ into $\PSL$ we can now define a Lyapunov exponent by the formula $$\chi(\rho) := \lim_{n\cv\infty} \unsur{n}\int \log\norm{\rho(g)} d\mu^n(g).$$ 
Here $\norm{\cdot}$ refers to any matrix norm on $\PSL$, and $\mu^n$ denotes the $n^{\rm th}$ convolution power of $\mu$, that is, the image of the product measure $\mu^{\otimes n}$ on $G^n$ under the map $(g_1,\ldots, g_n)\mapsto g_1\cdots g_n$.

For a holomorphic family of representations as above, we obtain in this way a non-negative psh function $\lambda\mapsto \chi(\rho_\la)$, and define, in analogy with the polynomial case,  the bifurcation current by the formula $T_{\rm bif}  = dd^c(\chi(\rho_\la))$.
%We see that the bifurcation current %a priori
%depends on the (non-canonical) choice\footnote{Natural choices can be made in certain situations, see  below.}  of $\mu$.  Despite this, we shall try to demonstrate in this paper that this definition gives rise to a meaningful object.

\medskip

For readers not necessarily familiar with positive currents, it is worth noting  that our results are already  interesting when $\dim(\La)=1$, in which case one can simply replace ``psh" by ``subharmonic'' and ``positive current" by ``positive measure''. Neverthess, some arguments in the proof require to work with actual currents on the 2-dimensional space $\La\times \pu$.  

\medskip

The theory of random products of matrices will be used to study the properties of this Lyapunov exponent function, and  
 show that the bifurcation current is a  meaningful object, truly capturing the bifurcations of the family.  
 
A first fundamental result, due to Furstenberg,  asserts that under the above assumptions, 
   $\chi(\rho_\la)$ is   positive, and depends continuously on $\la$. Furthermore, $\chi(\rho_\la)$ admits an expression in terms of  a canonical probability measure $\nu_\la$ on $\pu$, invariant under the average action of $\rho_\la(G)$, which will play an important role in the paper. More generally, it is remarkable that the proofs in the paper will require non-trivial results like the exponential convergence of the transition operator, the Large Deviations Theorem, etc. (see Bougerol-Lacroix \cite{bougerol-lacroix} and Furman \cite{furman} for good introductory texts on these topics) 

\medskip

Our first main result is the characterization of the support of the bifurcation current.

\begin{theo}\label{theo:support}
Let $(G, \rho, \mu)$ be a holomorphic family of representations as above. Then the support of  $T_{\rm bif}$ is equal to the bifurcation locus.
\end{theo}

To say it differently, the stability of a holomorphic family of M\"obius subgroups is equivalent to the pluriharmonicity of  the Lyapunov exponent function (for any $\mu$).
A notable consequence of the theorem is that  the support of $T_{\rm bif}$ does not depend on $\mu$. Another corollary, which was a basic source of motivation in  \cite{demarco1}, is that if $\La$ is a Stein manifold, the components of the stability locus are also Stein. This holds in particular when $\La$ is the space of all representations of $G$ into $\PSL$ 
%or the space of all representations 
modulo conjugacy, %since these are affine algebraic varieties
which is an affine algebraic variety
\footnote{Notice that no trouble can arise from the possible singularities of these varieties, since the components of the stability locus are disjoint from them 
(see \cite[\S8.8]{kapovich}).}. %\note{il faut régler le problème des singularités. j'ai ecris a J. Porti, on attend la reponse, mais a priori c bon car ce sont des produits d'espaces de teichmuller BD p. 70 et p. 193}
As a corollary one recovers the result of  Bers and Ehrenpreis~\cite{bers-ehrenpreis} that 
Teichm\"uller spaces are Stein manifolds.

\medskip

As before, the delicate inclusion in Theorem \ref{theo:support} is to show that if $\chi$ is pluriharmonic on an open subset $U$, then $U$ must be contained in the stability locus. The approach used in the context of rational dynamics seems to have no analogue here.  Instead, we use a geometric interpretation of the bifurcation current, which we briefly describe now (the details can be found in \S \ref{subs:geom_interpt}).

Let us  look at the fibered action of $G$ on $\La\times \pu$: for this, we fix $z_0\in \pu$ and  consider the graphs over $\La$ defined by  $\la\mapsto g_\la(z_0)$ as $g$ ranges in $G$. Over the stability locus, these graphs form a normal family. On the other hand, they tend to oscillate wildly when approaching the bifurcation locus.  We show  that $T_{\rm bif}$ describes the asymptotic distribution of this oscillation phenomenon (see Theorem \ref{thm:geom_interpt} for a precise statement). In particular, $\lambda_0\in \supp(T_{\rm bif})$ if and only if for every neighborhood $U\ni\la_0$, the average volume, with respect to  $\mu^n$, of the graph of $\la\mapsto g_\la(z_0)$ over $U$,   grows linearly with $n$ (which is the fastest possible growth). %It is this description of $T_{\rm bif}$ which allows to make the connection  with the stability theory of Kleinian groups.

%In the same spirit, we offer in Theorem \ref{thm:motion of fixed points}
%  yet another interpretation of $T_{\rm bif}$
%in terms of the motion of  fixed points in $\La\times \pu$.

% Another key fact in the proof of Theorem~\ref{theo:support} is the idea  that  somehow the Poisson boundary of the random walk on $G$  cannot vary holomorphically when a parameter crosses the bifurcation locus. For a more precise result in this direction, see Theorem~\ref{thm:stationary current}.

On the other hand we show that if $U$ is disjoint from $\supp(\tbif)$, the average volume of these graphs is locally  bounded. Using Bishop's compactness  theorem for sequences of analytic sets, this allows us to construct for every $\la\in U$ an equivariant  map $\theta_\la$ from the Poisson boundary of $(G,\mu)$ to $\pu$, depending holomorphically on $\la$, which ultimately turns out to be a holomorphic motion.   Sullivan's theory then implies that  the family of representations is stable over $U$.

\medskip

It is an easy consequence of Sullivan's results that for every $t\in [0,4]$, the set of parameters $\lambda_0$ such that there exists $g\in G$ with non-constant trace and $\tr^2(g_{\la_0})=t$  is dense in the bifurcation locus (see Corollary \ref{cor:accidental} below). The same is true
 for the set of parameters at which a collision   between fixed points of different elements  occurs.
In the light of what is known about spaces of rational maps,
it is natural to wonder whether in these assertions,   density   can be turned into equidistribution.
This will be the second main theme developed in the paper.

%Consider the product space $G^\nn$, endowed with the product measure $\mu^\nn$. If $\mathbf{g}= (g_n)_{n\geq 1}\in G^\nn$, we put $l_n(\mathbf{g}) = g_n\cdot g_1$ ($l$ for left product).
If $V$ is an analytic subset of $\La$, recall that  the integration current on $V$  is denoted by
 $[V]$.   When $\dim(\La)=1$ (hence $\dim(V)=0$) this is just a sum of Dirac masses  at the points  of $V$ (counted with multiplicities, if any).  It is convenient to adopt the convention that $[\La] = 0$.

Our first equidistribution theorem concerns random sequences in $G$.

\begin{theo}\label{theo:equidist ae}
 Let $(G, \rho, \mu)$ be a holomorphic family of representations as above. Consider the product space $G^\nn$, endowed with the product measure $\mu^\nn$. Then the following conclusions hold.

\begin{enumerate}
\item For $g\in G$ and $t\in\cc$,   let $Z(g,t)=\set{\la,\  \tr^2(g_\la)=t}$. Then for $\mu^\nn$-a.e. sequence $(g_n)_{n\geq 1}$ we have that $$\unsur{2n}\left[Z(g_n\cdots g_1, t )\right] \underset{n\cv\infty}\longrightarrow
 T_{\rm bif}.$$
\item For $g,h\in G$, let $F(g,h)$ be the analytic   subset of  $\La$ defined by the condition that   $ g_\la $ and $ h_\la $  have a common fixed point. Then for
$\mu^\nn\otimes\mu^\nn$-a.e. pair  $((g_n), (h_n))$,
$$\unsur{4n}\left[F(g_n\cdots g_1, h_n\cdots h_1)\right] \underset{n\cv\infty} \longrightarrow
T_{\rm bif}.$$
\end{enumerate}
\end{theo}

It follows in particular from {\it (1)} that  if the bifurcation locus is non empty, then almost surely   $Z(g_n\cdots g_1, t)$  is a  non empty proper analytic subvariety for large   $n$ (and similarly for {\it (2)}).

As far as we know, this is the first equidistribution statement of this kind. The proof is based on  a general machinery which produces  equidistribution theorems in parameter space from limit theorems for random sequences at every  (fixed) parameter (see Theorem \ref{thm:abstrait}).

\medskip

Since the 1980's, several authors have produced pictures of stability loci in 1-dimensional families of representations, by plotting solutions of $\tr^2(g)=4$ in parameter space (see \cite[Chapter 10]{indra}  for a beautiful account on this). These pictures exhibit intriguing accumulation patterns as the length of $g$ increases. Our equidistribution results say that these patterns are governed by the bifurcation measure --at least when the words $g\in G$ are chosen according to a random walk on $G$. 

\medskip

Here is a consequence of Theorem \ref{theo:equidist ae}
 which  does not make  explicit reference to a measure on $G$, and
does not seem easy to prove without using probabilistic methods:
for any $\e>0$ and any relatively compact set $\La'\subset \La$,
there exists $g\in G$ such that the set of parameters $\la$ such that
$\rho_\la(g) =\mathrm{id}$ (resp. $\tr^2(\rho_\la(g))=4$) is $\e$-dense in the bifurcation locus restricted to $\La'$.
 For this, it suffices to take $t=0$ (resp. $t=4$) in  the first item  of the above theorem, and  take $g  = (g_n\cdots g_1)^4$ (resp. $g=g_n\cdots g_1$) for  a $\mu^\nn$-generic sequence  $(g_n)$ and large enough $n$.

 \medskip

We are also able to estimate the speed of convergence   in item (1) above after some averaging with respect to $g$. This requires some additional assumptions on $\La$.  

\begin{theo}\label{theo:equidist speed} 
 Let $(G, \rho, \mu)$ be a holomorphic family of representations as above, and fix $t\in\cc$. Suppose in addition that one of the following conditions holds.
   \begin{enumerate}
 \item[i.] either $\La$ is an algebraic family of representations, defined over $\overline{\mathbb{Q}}$.
 \item[ii.] or there is at least one geometrically finite representation in $\La$.
 \end{enumerate}
Then there exists a constant  $C$ such that  for every test form $\phi$
$$\bra{\unsur{2n} \int  \left[Z(g,t)\right] d\mu^n(g) -T_{\rm bif}, \phi}\leq
 C \frac{\log n }{n} \norm{\phi}_{C^2}.
$$
 \end{theo}

The proof is more involved than that of Theorem \ref{theo:equidist ae}, and based on several interesting ingredients. To prove the theorem,  we need to understand how the potential of $\unsur{2n} \int  \left[Z(g,t)\right]$, namely $\unsur{2n}\int \log\abs{\tr^2 (g_\la)-t} d\mu^n(g)$ is close (in $L^1_{\rm loc}(\La)$) to the Lyapunov exponent function $\chi$.
Two main ingredients for this are:
\begin{itemize}
\itm  precise (large deviations) estimates on the asymptotic distribution of  $\tr^2(g_\la)$ for fixed $\la$  (which are established in Appendix \ref{app:distance}),
\itm bounds on the volume  of the set of representations possessing  elements with trace  
 close to $t$. 
\end{itemize}
%For this, we  shall
%require   precise estimates on the asymptotic distribution of  $\tr^2(g_\la)$ (for fixed $\la$), %(which are developed in Appendix \ref{app:distance}),
%and also on bounds on the volume  of the set of representations possessing  elements with trace   close to $t$. 

The role of the additional assumptions {\it i.} and {\it ii.}  is, for the purpose of establishing these volume bounds,  to ensure that 
$\log\abs{\tr^2 (g_\la)-t}$ cannot be uniformly close to $-\infty$ along $\La$. For instance, under {\it ii.} we have a good control of the set of values of $\tr^2(g_\la)$ at the geometrically finite parameter.
% The desired volume bounds are then obtained by using 
The result then follows from classical estimates on the volume of sub-level sets of psh functions. We also see that we can weaken assumption {\it ii.} by only requiring that  the family of representations $(\rho_\la)$ can be analytically continued to a family  containing a  geometrically finite representation.

Under {\em i.}, the desired estimate on  $\log\abs{\tr^2 (g_\la)-t}$ follows from   number-theoretic considerations. 

As a byproduct of our methods, we obtain a new proof and a generalization of a result of Kaloshin and Rodnianski \cite{kr} (see Remark \ref{rmk:moderate}).  
       
It is unclear whether the speed $O\lrpar{\frac{\log n}{n}}$ that we obtain  is optimal or not.  On might guess that  the optimal speed is  bounded below by  $O\lrpar{\unsur{n}}$ (see Remark \ref{rmk:speed}).
%Notice also that a similar  estimate appears as  key argument in the proof of Theorem \ref{theo:support} (see Proposition \ref{prop:Ounsurn}).

For the analogous equidistribution theorems associated to families of rational maps, no such general estimate is known. The only quantitative equidistribution result towards $T_{\rm bif}$ that we know of in that context  is specific to the unicritical family $z^d+c$, and relies  on number-theoretic ideas \cite[Theorem 5]{frl}.  Notice also  that the proofs of most  of the  equidistribution theorems in  \cite{df, bas-ber2} also require some global assumptions on $\La$. In whatever case, it is unclear whether these assumptions are really necessary. 

\medskip

Many particular families of representations have been studied in the literature.
Let us focus on one classical situation (more details can be found in  \S \ref{subs:bers slice} which is itself a preview of a sequel \cite{dd2} to this paper).
Fix  a compact Riemann surface $S$ of genus $g\geq 2$, and introduce the complex affine space $\Lambda\simeq \cc^{3g-3}$ of complex projective structures on $S$, compatible with the complex structure of $S$. Any such  projective structure gives rise to a {\em monodromy representation} of the fundamental group of $S$ into $\PSL$, that varies holomorphically with $\la$. The well-known {\em Bers slices} of Teichm\"uller space are obtained from this construction.

Any random walk on the group $\pi_1(S)$ then gives rise to a bifurcation current on $\La$. In fact in this setting there is more: we claim that there exists   a {\em canonical bifurcation current} on $\Lambda$. For this, let  us shift  a little bit our point of view, and instead of a discrete random walk on $\pi_1(S)$, consider the Brownian motion on $S$ (which depends only on the Riemann surface structure). A projective structure being given, we can  consider the growth rate of   its holonomy   over generic Brownian paths, thereby obtaining a  Lyapunov exponent, in the spirit of~\cite{dk}. This induces a natural  psh function on $\La$, hence a natural bifurcation current. It turns out that this bifurcation current is induced by a measure on $\pi_1(S)$, therefore   it satisfies the above theorems.

\medskip

There is some similarity between Theorem~\ref{theo:support}  and a recent result of  Avila's \cite{avila}, appearing  as a crucial step in the proof of the stratified analyticity of the Lyapunov exponent of quasi-periodic Schr\"odinger operators.
To be precise,  to an irrational number  $\alpha\in\re\setminus\mathbb Q$ and      a real-analytic function  $A: \mathbb R / \mathbb Z\cv \PSL$, we associate a  Lyapunov exponent
by the formula
\[ L (A,\alpha) = \lim _{n\rightarrow \infty} \frac{1}{n} \int_{\re/\zz} \log \norm{A(x) A(x+\alpha)\cdots A(x + (n-1)\alpha)  }dx. \] The variation of the function $\varepsilon \mapsto L( A_{\varepsilon}, \alpha)$, where $A_{\varepsilon} (\cdot ) = A ( \cdot + i\varepsilon)$, is studied in detail, and Avila  proves \cite[Theorem 6]{avila} that %at a point where the Lyapunov exponent is positive
if $L(\alpha, A)>0$, this function is locally affine if and only if $(\alpha, A)$ is uniformly hyperbolic. This  is completely analogous to the above statement that a family of representations is stable if and only if the Lyapunov exponent is pluriharmonic.

\medskip

It is also worth  mentioning the recent work of Cantat~\cite{cantat} in which the author uses higher dimensional holomorphic dynamics to study the action of the mapping class group on the character variety of the once-punctured torus (resp. the four times punctured sphere). A given mapping class acts by holomorphic automorphisms on the character variety, so it usually admits invariant currents, supported on the bifurcation locus.
% His currents are presumably very different from ours. For instance, they  are  typically laminar,  a property which is not expected to hold in our situation (see \S\ref{subsub:wedge}). note: c'est ridicule ils sont portés par le bord du lieu de bif de toute façon.
%Arising from a different process, these currents are certainly very different from ours. It might  neverthess be interesting to relate the two constructions.\note{à reprendre}
It is unclear to us whether these currents are related to ours. It would be interesting nevertheless to explore the relationship between the two constructions. 

\bigskip

Here is the structure of the paper. In Section \ref{sec:preliminaries} we give some background on   holomorphic families of   subgroups of $\PSL$, as well as a number of basic results in the theory of %random walks on groups and
random products of matrices.  In Section \ref{sec:bifcurrent} we introduce the bifurcation current, give its geometric interpretation  and prove Theorem \ref{theo:support}. We also give in Theorem \ref{thm:stationary current} the classification of  all ``stationary currents'' for a holomorphic family of M\"obius groups (under a mild assumption).  Section \ref{sec:equidistribution} is mainly devoted to Theorems \ref{theo:equidist ae} and \ref{theo:equidist speed}. Two auxiliary results required in the proof of  the equidistribution theorems have been moved to appendices: in Appendix \ref{app:distance}, we study the distribution of fixed points of random products in $\PSL$ (more general related  results recently appeared in  \cite{aoun}). In appendix \ref{app:number} we prove a number-theoretic lemma  related to the assumption {\em i.} of Theorem \ref{theo:equidist speed}. Finally, in Section \ref{sec:further} we outline some further developments (in particular the construction of canonical bifurcation currents), as well as a number of open questions.

\medskip

It is a pleasure to thank our colleagues R. Aoun, E. Breuillard, S. Boucksom,  J.-Y. Briend, C. Favre, C. Lecuire, P. Philippon, A. Zeriahi as well as the anonymous referee for useful conversations and comments.

\section{Preliminaries}\label{sec:preliminaries}

\subsection{M\"obius subgroups}
Mainly for the purpose of fixing notation, we recall some basics on subgroups of $\mathrm{Aut}(\pu)$, where $\pu$ will refer to the Riemann sphere. The reader is referred to e.g. \cite{beardon, kapovich} for more details.

We identify $\mathrm{Aut}(\pu)= \set{z\mapsto \frac{az+b}{cz+d}, \ ad-bc\neq 0}$ with the matrix group $\PSL$.
If $\gamma\in \PSL$,  it is often convenient for calculations to lift $\gamma$ to one of its representatives in $\SL$.
 %$$\norm{\gamma} =\lrpar{\abs{a}^2+ \abs{b}^2+\abs{c}^2 + \abs{d}^2}^{\frac12} \text{ and }
%\tr ^2{\gamma} = (a+d)^2$$ 
We define the quantities $\norm{\gamma}$ and $\tr^2\gamma$ by lifting $\gamma$ to one of its representatives in $\SL$. Of course the result does not depend on the lift.  In this paper $\norm{\gamma}$ denotes the operator norm associated to the Hermitian norm on $\cd$. As it is well-known $\norm{\gamma}$ equals the square root of the spectral radius of $A^*A$, where    $A$   is any matrix representative of $\gamma$. It will also be sometimes convenient to work with $\norm{\gamma}_2:=  \big(\abs{a}^2+ \abs{b}^2+\abs{c}^2 + \abs{d}^2\big)^{1/2} $.

As usual we classify M\"obius transformations into three types:%(it will be convenient to us to us to consider the identity as being a parabolic element):
\begin{itemize}
\itm {\em parabolic} if $\tr^2(\gamma)=4$ and $\gamma \neq\mathrm{id}$; it is then conjugate to  $z\mapsto z+1$;
\itm {\em elliptic} if $\tr^2(\gamma)\in [0,4)$, it is then conjugate to $z\mapsto e^{i\theta}z$ for some real number $\theta$, and $\tr^2{\gamma} = 2+2\cos(\theta)$.
\itm {\em loxodromic} if $\tr^2(\gamma)\notin [0,4]$, it is then conjugate to $z\mapsto kz$, with $\abs{k}\neq 1$.
\end{itemize}

We equip $\mathbb P^1$ with the spherical metric $\frac{|dz|}{1+|z|^2}$, and the associated spherical volume form, simply denoted by $dz$. The subgroup of elements of $\PSL$ that preserve this metric is
 isomorphic to $\text{SO}(3,\mathbb R)$.
As usual, these elements are called rotations.

The following elementary lemma shows that when the fixed points of a loxodromic map $\gamma$ are separated enough, the quantities $\sqrt{|\tr^2{\gamma}|}$ and $\norm{\gamma} $ are essentially the same. 
We use the following notation: if $u$ and $v$ are two  real valued functions, we write
 $u \asymp v$ if there exists a constant $C>0$ such that $\unsur{C}u\leq v\leq Cu$.

\begin{lem} \label{l:norm, trace and distance between fixed points}
If $\gamma$ is not parabolic, then
\[  \norm{\gamma} \asymp  %\sqrt{ 1 + \frac{|\tr^2{\gamma} -4|}{\delta^2} }  
\max (1, \frac{\sqrt{|\tr^2 {\gamma} -4|}}{\delta} )  \]
where $\delta$ is the distance between the two fixed points of $\gamma$.    
\end{lem}

\begin{proof} Let $\gamma(z)  = \frac{az+b}{cz+d}$ as before. Pick a fixed point of $\gamma$ with multiplier of smallest modulus, and conjugate by a rotation so that this fixed point becomes $\infty$. This does not affect the trace nor the norm of $\gamma$. The expression of $\gamma$ is now  $\gamma (z ) = a^2z + ab$, with $\abs{a}\geq 1$. The other fixed point of $\gamma$ is $ab/(1-a^2)$ (by assumption, $a^2\neq1$). It will be convenient to assume that this point is separated from $0$ by a certain distance, say $1$. To achieve this, we conjugate $\gamma$ 
by  a translation $\tau$ of bounded length. Of course, $\norm{\tau^{-1}\gamma\tau}\asymp\norm{\gamma}$, so it is enough to estimate the norm of ${\tau^{-1}\gamma\tau}$, which we rename as $\gamma$. 

This being done, we have the following formulas 
\begin{itemize} 
\itm $\tr^2{\gamma} = (a + \frac{1}{a}) ^2$,
\itm $\norm{\gamma} \asymp \norm{\gamma}_2 = \sqrt{|a|^2 + \frac{1}{|a|^2} + |b|^2} $,
\itm $ \delta \asymp \frac{|1-a^2|}{|a b|}$.
\end{itemize} 
We split the argument into two cases.  First suppose that $|a|$ is large,  $|a | \geq 1000$, say. Then, we have  that
\[ \tr^2{\gamma} \asymp a^2,\ \ \delta \asymp \frac{|a|}{|b|},\ \ \norm{\gamma}\asymp \frac{|a|}{\delta},\]
and the lemma is proved is this case because $\frac{\tr^2{\gamma}}{\delta^2}$ is large. 

Now, suppose that $|a| \leq 1000$. As before, 
 $\delta \asymp \frac{|1/a -a |}{|b|} $, and since $(a - 1/a)^2 = \tr^2{\gamma} -4 $, we get that $\norm{\gamma} ^2 \asymp 1 + |b|^2 \asymp 1 + \frac{\tr^2{\gamma} -4}{\delta^2}$, and the lemma follows.\end{proof}

For $z_0\in \pu$ and $\gamma \in \PSL$,  let $\norm{\gamma z_0} = \frac{\norm{\gamma Z_0}}{\norm{Z_0}}$, where $Z_0\in \cd$ is any lift of $z_0$ and $\norm{\cdot}$ in $\cd$ is the Hermitian norm. We have the following estimate:

\begin{lem}\label{lem:norme integree}
 There exists a universal constant $C$ such that  
 if $\gamma\in \PSL$, $$\abs{\int_{\pu} \log\norm{\gamma z_0} dz_0 - \log \norm{\gamma}}\leq C.$$ 
\end{lem}
    
\begin{proof} 
Let $A\in \SL$ be a matrix representative of $\gamma$. By the KAK decomposition there exists $R,R'\in \mathrm{SU}(2)$ and $B=\lrpar{\begin{smallmatrix}\sigma &0\\ 0&\sigma^{-1}                                                                                                                                                                                                                                        \end{smallmatrix}}$ such that $A=RBR'$, where $\sigma=\norm{A}$ is the spectral radius of $\sqrt{A^*A}$. %Note that $\norm{A} \asymp \sigma  \asymp \norm{B}$. 
Changing variables, we see that $\int_{\pu} \log\norm{\gamma z_0} dz_0  = 
\int_{\pu} \log\norm{B w_0} dw_0$. Therefore it is enough to prove the lemma with $B$ in place of $\gamma$, which will be left as an exercise to the reader. 
  \end{proof}

The following  fact will also be useful (notice that if $\gamma$, $\gamma'\in \PSL$, the trace  $\tr[\gamma, \gamma']$  
 is well-defined).

\begin{lem}[{\cite[Thm. 4.3.5]{beardon}}]\label{lem:commutator}
Two M\"obius transformations $\gamma$ and $\gamma'$ have a common fixed point in $\pu$ if and only if
$\tr[\gamma, \gamma']=2$.
\end{lem}

 %As opposite to individual elements, it is not always possible to lift $\Gamma$ to a subgroup of $\SL$: the obstruction  for this lies in $H^2(\Gamma, \zz/2\zz)$.%\footnote{ $H^2(\Gamma, \zz/2\zz)$ vanishes for instance for free groups, but not for surface groups}

Recall that the action of a M\"obius transformation on $\pu$ naturally extends to the 3-dimensional hyperbolic space $\mathbb{H}^3$. Let now $\Gamma$ be a subgroup of $\PSL$. We say that $\Gamma$ is {\em elementary} if it admits a finite orbit in $\overline{\mathbb{H}^3}$. Then, either $\Gamma$ fixes a point in $\mathbb{H}^3$ and is conjugate to a subgroup of $\mathrm{SO}(3,\mathbb R)$ (in particular it contains only   elliptic elements) or it has a finite orbit (with one or two elements) on $\pu$ \cite[\S 5.1]{beardon}.  %\note{à vérifier, certainement inutile de toute façon. en fait c'est un cocycle à valeurs dans $H^2 (\Gamma, \{\pm 1\})$: pour le définir on prend des relevés $\tilde{\gamma}$ dans $SL$ des éléments $\gamma$ du groupe, et on constate que $\tilde{\gamma_1}\tilde{\gamma_2} = \epsilon (\gamma_1,\gamma_2 ) \widetilde{\gamma_1\gamma_2}$, pour un certain $\varepsilon = \pm$. Cette fonction est un $2$-cocycle à valeurs dans $\{\pm 1\}$ (elle vérifie $\varepsilon (\gamma_1, \gamma_2) \varepsilon (\gamma_1 \gamma_2, \gamma_3) = \varepsilon (\gamma_2, \gamma_3) \varepsilon (\gamma_1, \gamma_2 \gamma_3) $) et est bien définie modulo un cobord (comment on change de relevé). }.
By definition, a {\em Kleinian group} is a  discrete subgroup of $\PSL$.
%If $\Gamma$ is  Kleinian, we denote its limit set by $\mathrm{Lim}(\Gamma)$. It is the closure of fixed points of loxodromic elements of $\Gamma$, and it is a perfect set unless $\Gamma$ is elementary (in which case $\#\mathrm{Lim}(\Gamma)\leq 2$).

\subsection{Holomorphic families of finitely generated subgroups of $\PSL$}\label{subs:families}
Let $G$ be a finitely generated group and $\La$ be a connected complex manifold. A holomorphic family of representations of $G$ into $\PSL$ over   $\La$ is a mapping $$\rho: \Lambda\times G \longrightarrow \PSL,$$ such that $\la \mapsto \rho_\lambda(g)$ is holomorphic for fixed $g$, and $g\mapsto \rho_\lambda(g)$ is a group homomorphism for fixed $\lambda$. We denote such a family by $(\rho_\la)_{{\la\in \La}}$.
For $g\in G$, we usually denote $\rho_\lambda(g)$ by $g_\la$.

Throughout the paper, we make the standing assumption that $(\rho_\la)_{\la\in \La}$ is {\em generally faithful}, that is, that the set of parameters for which $\rho_\la$ is not injective is a countable union of proper subvarieties. For this, it is enough that  for some $\lambda_0\in \La$, $\rho_{\la_0}$ is injective.
We also assume that the family is non-trivial, that is, that  the $\rho_\lambda$ are not all conjugate in  $\PSL$.

\begin{lem}\label{lem:elementary} If there exists a non-elementary representation in $\La$, then
the set of parameters $\lambda\in \Lambda$ for which $\rho_\la(G)$ is elementary is contained in a proper real analytic subvariety of $\La$.
\end{lem}

\begin{proof}
As said before, there are two possibilities for  $\rho_\la(G)$ to be  elementary:
\begin{itemize}
 \itm[type I] either all elements have a common orbit of period $2$, hence for every pair
  $f, g\in G$, $\tr[f_\lambda^2, g_\lambda^2]\equiv 2$  (Lemma \ref{lem:commutator}),
  \itm[type II]  or all elements are elliptic, hence for every $g_\la$, $\tr^2(g_\lambda)\in[0,4]$, and in particular $\Im \tr ^2 (g_{\lambda}) =0$.
 \end{itemize}
We see that  the parameters for  elementary subgroups satisfy a family of real analytic equations.
\end{proof}

% We will also assume that at least one representation of the family is non-elementary. In this case the lemma shows that apart from a codimension $1$ analytic submanifold, the representations are non-elementary. Since the problems that we consider does not depend on what happens on a positive codimension submanifold in parameter space,  reducing it if necessary, it is not a restriction to
% assume that for $\la\in \La$, $G_\la$ is non-elementary.

We say that a family of representations is  {\em generally non-elementary} if it satisfies the assumption of the lemma. 
Since most of the problems that we consider are local, reducing the parameter space if necessary, it is not a restriction to assume that for all $\la\in \La$, $\rho_\la$ is non-elementary.  

\medskip

Two representations $\rho_{\lambda_0}$ and $\rho_{\lambda_1}$ are {\em quasi-conformally conjugate} if there exists a quasi-conformal homeomorphism $\phi:\pu\cv\pu$ such that for every $g\in G$, $\rho_{\lambda_0}(g)\circ\phi = \phi\circ \rho_{\lambda_1}(g)$.

\begin{defi}
We say that $\rho_{\lambda_0}$ is stable if for $\lambda$ close to $\lambda_0$, $\rho_\lambda$ is quasi-conformally conjugate to $\rho_{\lambda_0}$. The {\em stability locus} $\mathrm{Stab}\subset\La$ is the open set of stable representations. Its complement is the {\em bifurcation locus} $\mathrm{Bif}$.
\end{defi}

\begin{thm} [Sullivan \cite{sullivan}, Bers \cite{bers}]\label{thm:sullivan}
Let $(\rho_\lambda)_{{\lambda\in \La}}$ be a non-trivial, generally faithful, holomorphic family of non-elementary representations of $G$ into $\PSL$, and $\om\subset\Lambda$ be an open set. Then the following assertions are equivalent:

\begin{enumerate}[{ i.}]
\item  for every $\lambda\in \om$, $\rho_\la(G) $ is discrete;
\item for every $\lambda\in \om$, $\rho_\lambda$ is faithful;
\item for every $g$ in $G$, if for some $\lambda_0\in \om$,  $g_{\lambda_0}$ is loxodromic (resp. parabolic, elliptic), then $g_\lambda$ is loxodromic
(resp. parabolic, elliptic) throughout $\om$;
\item for any $\lambda_0$, $\lambda_1$ in $\om$
the representations $\rho_{\lambda_0}$ and $\rho_{\lambda_1}$ are  quasi-conformally conjugate on $\pu$.
\end{enumerate}
\end{thm}

In the situation of the theorem, due to {\em iii.}, the set of fixed points of $\rho_\la(g)$, $g\in G$ moves holomorphically. Furthermore, there exists a holomorphic motion of $\pu$, extending the motion of fixed points, and 
 commuting with the action of $G$. The well-known Zassenhaus-Margulis  lemma  implies that the set $\mathrm{DF}$ of discrete and faithful representations is closed in parameter space (this is also referred to as Chuckrow's or J\o rgensen's Theorem, 
see \cite[p. 170]{kapovich}). We infer that the stability locus is the interior  of $\mathrm{DF}$. Observe in particular that, in contrast with rational dynamics, when non-empty the bifurcation locus has non-empty interior.

The  previous theorem allows us to exhibit a dense (complex) codimension 1  phenomenon in the bifurcation locus. This is a basic  source of motivation for the  introduction of  bifurcation currents.

\begin{cor}\label{cor:accidental}
For every $t\in [0,4]$, the set of parameters $\lambda_0$ such that there exists $g\in G$ with $\tr^2{\rho_{\lambda_0}(g)}=t$ and $\lambda \mapsto \tr^2{\rho_{\lambda}(g)}$ is not constant at $\lambda_0$, is dense in the bifurcation locus.
\end{cor}

\begin{proof}
%The element $g$ has infinite order since for some $\lambda$ close to $\lambda_0$ the element $\rho_{\lambda} (g)$ is a loxodromic element. However, there are $\lambda$ arbitrarily close to $\lambda_0$ such that  $\tr^2{\rho_{\lambda}(g)}$ is of the form $2+2\cos(2\pi r)$, for some $r\in \mathbb Q \setminus \mathbb Z$. The corresponding element $\rho_{\lambda}(g)$ is conjugated to a finite order rotation. Thus $\rho_{\lambda}$ is not faithful.   
For $t=4$, the result is clear. For other values of $t$, consider an open set $\om$ with $\om\cap\mathrm{Bif}\neq \emptyset$. There exists $g\in G$ which changes type in $\om$. Thus, the values of $\tr^2(g_\lambda)$, for $\lambda\in \om$, cross $[0,4]$ along a non-empty  open interval. We infer that for large $k$, the set of values of $\tr^2(g^k_\lambda)$, $\lambda\in \La$,  contains $[0,4]$, hence the result.
\end{proof}

In Section \ref{sec:equidistribution} we will require the notion of an \textit{algebraic family of representations}. For this we need to introduce a few concepts;
see  e.g. \cite{kapovich} for a more detailed presentation.
Fix a finite set  $\set{g_1, \ldots , g_k}$ generating $G$.
The space $\mathrm{Hom}(G, \PSL)$ may be regarded as an algebraic subvariety $V_G$ of $\PSL^k$ by simply mapping a representation $\rho$ to $(\rho(g_1),\cdots, \rho(g_k))\in \PSL^k$, and observing that even if $G$ is not finitely presented, $V_G$ will be defined by finitely many equations.
The same holds of course for $\mathrm{Hom}(G, \SL)$. The algebraic structure of $V_G$ is actually independent of the presentation of $G$. Notice also that $V_G$ is defined over $\mathbb{Q}$. 

There is an obvious embedding of $\SL$ into $\cc^4$ making $\SL^k$   an affine subvariety of $\cc^{4k}$. To view
$\PSL$ as an affine variety,  observe that $\PSL$ acts faithfully by conjugation on the space of 2-by-2 complex matrices of trace zero. This embeds $\PSL$ into $\mathrm{GL}(3,\cc)\subset \cc^9$, and actually,  $\PSL$ is isomorphic to $\mathrm{SO}(3,\cc)$. We conclude that $V_G$ is   an affine algebraic variety (again, defined  over $\mathbb{Q}$).  

A holomorphic family of representations $(\rho_\la)_{\la\in \La}$ is now simply a holomorphic mapping $\rho:\La\cv V_G\subset \cc^{9k}$. We say that such a family is {\em algebraic} (resp. defined  over $K$, where $K$ is some subfield of $\cc$) if there exists an algebraic subset $V\subset V_G$ (resp. defined  over $K$) such that $\rho:\La\cv V$ is a dominant mapping. To say it differently, we require that the image $\rho(\La)$ contains an open %Zariski dense 
 subset of an algebraic subset of $V_G$. This notion does not depend on the presentation of $G$.

\subsection{Products of random matrices}\label{subs:random}
In this paragraph we recall some classical facts on random matrix products  that we specialize to our situation.  The reader is referred to \cite{bougerol-lacroix, furman} for more details and references.
Fix a non-elementary representation $\rho : G \rightarrow \PSL$. Let $\mu$ be a probability measure on the group $G$, whose support generates $G$ as a semi-group. Let us first  work under the following moment assumption:
\begin{equation}\label{eq:moment condition} \int _{G} \log \norm{\rho (g) } d\mu (g) < \infty .\end{equation}
If the group $G$ is finitely generated, and ${\rm length} ({g})$ denotes the minimal length of a representation of $g$ as a word in some fixed set of generators, then  the condition
\begin{equation} \label{eq:moment condition finitely generated} \int _G    \length({g})   d\mu (g) <\infty \end{equation}
clearly  implies (\ref{eq:moment condition}). An interesting case where these moment conditions (and also \eqref{eq:exponential moment} and \eqref{eq: exponential moment condition in the group} below) are satisfied is that 
of the normalized counting measure on a finite symmetric set of generators of $G$.

The \textit{Lyapunov exponent} of a representation $\rho : G \rightarrow \PSL$ is defined by the formula
\begin{equation}\label{def:lyapunov exponent} \chi (\rho ) : = \lim _{n\rightarrow \infty} \frac{1}{n}\int _G \log \norm{\rho(g)} d\mu^{n}(g), \end{equation}
where $\mu^{n}$ is the $n^{\rm th}$ convolution power of $\mu$, that is the measure on $G$ defined as the image of the product measure $\mu^{\otimes n}$ on $G^{n}$ under the map $G^n \ni (g_1,\ldots, g_n) \mapsto   g_n\ldots g_1\in G$.  Likewise, $\mu^n$ is the law of the ${n}^{\mathrm{th}}$ step of the left- or right- random walk on $G$ with transition probabilities given by $\mu$ and starting at the identity. %The limit in \eqref{def:lyapunov exponent} does not depend on a chosen norm on $\PSL$, and exists since the sequence $\int \log \norm{\rho(g)} d\mu^{n}(g)$ is easily seen to be sub-additive for an operator norm.

Throughout the paper we use the following notation: $\mathbf g=(g_n)_{n\geq 1}$ denotes a sequence in $G^\nn$, and $l_n(\mathbf g) = g_n\cdots g_1$ (resp. $r_n(\mathbf g) =g_1\cdots g_n$) is the product on the left (resp. right) of the first $n$ elements of $\mathbf g$. 
%We now look at the action of $G$ on the projective line $\mathbb P^1$ by M\"obius maps.

The study of the Lyapunov exponent is closely related to that of the {\em transition operator},
that is, the Markov operator $P$ acting on the space %$C^0(\mathbb P^1)$
of continuous complex valued functions on $\mathbb P^1$  by $f\mapsto Pf$, where $Pf$ is given by the formula
\begin{equation} \label{def:Markov operator}  P f (x) = \int _ G f(\rho(g) x) d\mu(g).\end{equation}
 A probability measure on $\mathbb P^1$  such that $\int Pf d\nu = \int fd\nu$ for all $f$
   is called {\em stationary}.

The following important result is due to Furstenberg \cite{furstenberg}.

\begin{thm}\label{t: furstenberg}
Let $\rho$ be a non-elementary representation of $G$ and $\mu$ be a probability measure on $G$, generating $G$ as a semi-group, and satisfying  (\ref{eq:moment condition}). Then the Lyapunov exponent defined in (\ref{def:lyapunov exponent}) is positive. Moreover, if $z_0\in \pu$ is fixed, then  for $\mu^\nn$ a.e. $\mathbf{g}$,
$$\lim_{n\cv\infty} \unsur{n} \log{\norm{\rho(l_n(\mathbf{g}))Z_0}}  =
\lim_{n\cv\infty} \unsur{n} \log\norm{\rho(l_n(\mathbf{g}))}  = \chi(\rho),$$
where $Z_0$ denotes any lift of $z_0\in \pu$ to $\mathbb C^2$.

Furthermore, there exists a unique stationary measure on $\mathbb P^1$, which
is diffuse and quasi-invariant under $\rho(G)$. Moreover, we have the formula
\begin{equation}\label{eq:Furstenberg's formula} \chi = \int _{\mathbb P^1} \int _ G \log \frac{\norm {\rho(g) (Z)} }{\norm{Z}} d\mu (g) d \nu (z)  \end{equation}
(again $Z$ is any lift of $z$).
\end{thm}

%A related result is that for every $z\in \pu$ the sequence of measures, $\int _G \delta_{g z} d\mu^n(g)$ converges weakly to the measure $\nu$. Let now $f$ the function on $\pu$ defined by $f(z)= \int\log\frac{\norm{\rho(g)(Z)}}{\norm{Z}} d\mu(g)$, where $Z\in \cd$ is any lift of $z$. We deduce from the above mentioned convergence the Furstenberg's formula  
%\begin{equation} \label{eq:Furstenberg's formula} \int f d\nu  = \chi , \end{equation} 
%where $\nu$ is the unique stationary measure. 

%The stationary measure is quasi-invariant by the group $\rho (G)$. This important fact is a consequence of the following trivial inequality $\rho(h)_* \nu \leq \frac{1}{\mu (h)} \int \rho(g)_* \nu d\mu (g) = \frac{1}{\mu(h)} \nu$ for every $h\in \text{supp}(\mu)$, and of the assumption that $\text{supp}(\mu)$ generates $G$ as a semi-group. In particular, the support of $\nu$ is $\rho(G)$-invariant.

The fact that  for $\mu^{\nn}$-a.e. ${\bf g}$, $\lim_{n\rightarrow \infty} \frac{1}{n} \log \norm{\rho (l_n({\bf g}))} =\chi (\rho)$ was originally proved by Furstenberg and Kesten \cite{furstenberg kesten}, and can nowadays easily be deduced from Kingman's Sub-additive Ergodic Theorem (see e.g. \cite{pollicott}). Let us explain the argument. The measure $\mu^{\mathbb N}$ is invariant and ergodic under the shift $\sigma: (g_n) \mapsto (g_{n+1})$. %Furthermore, the value of the limit does not depend on the choice of the norm, so (in this paragraph only\note{le ref n'est pas d'accord, on pourrait utiliser la notation $||| \cdot |||$. oui c'est une option, mais il ne faut pas oublier d'endroits où on utilise effectivement une norme d'opérateur. le mieux serait d'avoir une norme qui est à la fois psh et d'opérateur non, par exemple la norme sup, qu'en penses tu? il faudrait alors seulement changer la def et un tout petit peu la preuve du lemme 2.1}) we  replace $\norm{\cdot}$ by an operator norm.
%For every $\gamma_1, \gamma_2\in \PSL$,  $\norm{\gamma_1\gamma_2}\leq \norm{\gamma_1} \norm{\gamma_2}$.
Since, $\norm{\cdot}$ is an operator norm,  the family of functions $G^{\mathbb N} \ni {\bf g}   \mapsto \log \norm{\rho(l_n ({\bf g})) } \in \mathbb R ^+$ defines a sub-additive cocycle, that is  for every   $m,n \geq 0$,
\begin{equation}\label{eq:sub-additive cocycle}
 \log \norm{\rho(l_{n+m} ({\bf g})) }\leq  \log \norm{\rho(l_{ m} ({\bf g}) )} + \log \norm{\rho(l_{ n} (\sigma ^n {\bf g})) }.\end{equation} By the sub-additive ergodic theorem, $\unsur{n} \log \norm{\rho(l_n ({\bf g})) }$ converges  $\mu^{\mathbb N}$-a.e.  to a nonnegative number $\widetilde\chi(\rho)$. 
Furthermore,  by  subadditivity  we can write
$$0\leq \unsur{n} \log \norm{\rho(l_n ({\bf g})) } \leq \unsur{n}\sum_{k=1}^n \log\norm{\rho(l_1(\sigma^k(\mathbf{g})))},$$
where by the Birkhoff ergodic theorem and \eqref{eq:moment condition}, the right hand side converges in $L^1(\mu^\nn)$.   
This domination implies  that  $\unsur{n} \log \norm{\rho(l_n ({\bf g})) } $ also converges in $L^1$.
  By integrating against $\mu^\nn$,  we therefore conclude that 
 $\widetilde\chi(\rho)$ must equal $\chi(\rho)$, and the result follows.

\medskip

We note for future reference that when $\rho$ is non-elementary, the  support of the stationary measure coincides with the \textit{limit set} of the representation $\rho$,  defined as the minimal closed $\rho(G)$-invariant subset of the Riemann sphere. The proof   goes as follows: since $\nu$ is quasi-invariant under $\rho(G)$, $\supp(\nu)$ is closed and $\rho(G)$-invariant, thus it must  contain  the limit set. Conversely,  by uniqueness of the stationary measure, the limit set cannot be a proper subset of $\supp(\nu)$.

\bigskip

We now discuss the exponential convergence of the iterates of the transition operator to the stationary measure, due to Le Page. For this, denote by $N$ the operator of integration against the stationary measure $\nu$, $$N: f\mapsto \int fd\nu.$$ 
Let $C^\alpha$ be the space of H\"older continuous functions on $\mathbb P^1$ endowed with the norm \[ \norm{f}_{C^\alpha}  = \norm{f}_\infty +  m_{\alpha }(f)\text{, with }
m_{\alpha} (f) := \sup _{x\neq y\in \mathbb P^1} \left( \frac{|f(x) - f(y)|}{d_{\pu}(x,y)} \right) ^{\alpha} ,\] 
($d_{\pu}$ is the spherical distance). 

We also need  stronger moment assumptions on $\mu$: 
we assume that $(G, \mu,\rho)$ is non-elementary and satisfies the following:
\begin{equation}\label{eq:exponential moment}
\text{there exists }\tau>0 \text{ such that } \int_G\norm{\rho(g)}^\tau d\mu <\infty.
\end{equation}
As above, it is enough that $\mu$ satisfies an exponential moment condition in $G$:
\begin{equation}\label{eq: exponential moment condition in the group}
\text{there exists }\sigma>0 \text{ such that } \int_G \exp(\sigma  \length(g)) d\mu <\infty.
\end{equation}
 
The following important result is due to Le Page~\cite{lepage-thm limite}.

\begin{thm}\label{thm:lepage}   Let $(G, \mu, \rho)$ be a non-elementary representation satisfying \eqref{eq:exponential moment}. Then there exists $\alpha,\beta>0$, and a constant $C$ such that for every $n\geq 0$
\begin{equation}\label{eq:holder convergence}
\norm{P^n -N}_{C^{\alpha}} \leq C e^{-\beta n}.
\end{equation}
\end{thm} 

This in turn follows from an estimate on average contraction: there exists $0<\alpha<\tau$ and an integer 
$n_0$ such that 
\begin{equation}\label{eq:condition on distances}  \sup_{x\neq y \in \mathbb P^1} 
\mathbb \int \left(\frac{d_{\pu}( \rho (g) x, \rho (g) y) } {d_{\pu}(x,y)}\right)^{\alpha} d\mu^{n_0} (g) <1.
\end{equation}

\medskip

Important consequences of these estimates are versions for random matrix products of the classical limit theorems for i.i.d. random variables: Central Limit Theorem, Large Deviations Theorem, etc.

Another result, due to Guivarc'h, will be useful to us.

\begin{thm}[{\cite[Theorem 9]{guivarch}}]\label{thm:guivarch}
Let $(G, \mu, \rho)$ be a non-elementary representation satisfying \eqref{eq:exponential moment}. Then for $\mu^\nn$ a.e. $\mathbf{g}$ we have  that
\begin{equation}\label{eq:guivarch}
\unsur{n}\log \abs{\tr (\rho( l_n(\mathbf{g})))} = \unsur{n}\log\abs{ \tr(\rho(g_n\cdots g_1))} \underset{n\cv\infty}\longrightarrow
\chi(\rho).
\end{equation}
\end{thm}

We actually give a proof of a refined version of this theorem in Appendix \ref{app:distance} below.
The following corollary is immediate.

\begin{cor}\label{cor:guivarch}
Let $(G, \mu, \rho)$ be as in Theorem \ref{thm:guivarch}. Let
$h:\re\cv\re$ be  bounded from below and  equivalent to $\log x$ when  $x\cv +\infty$. Then
$$\unsur{n}\int  h\left( \abs{\tr (\rho( g))}\right) d\mu^n(g) \underset{n\cv\infty}\longrightarrow \chi(\rho).$$
For instance, if $\lambda_{\max}(\rho(g))$ denotes the spectral radius of $\rho(g)$,
then we have that
$$\unsur{n}\int  \log\abs{\lambda_{\max}(\rho( g))} d\mu^n(g) \cv \chi(\rho).$$
\end{cor}

On the other hand  one cannot expect in general  to have convergence in $L^1(G,\mu^{\mathbb N})$ in \eqref{eq:guivarch}. Indeed, if $\rho(G)$ contains the  rotation of angle $\pi$, whose trace is $0$, 
%$\left(\begin{smallmatrix} 1 & 0\\ 0 & -1 \end{smallmatrix}\right)$, 
then $\unsur{n}\int  \log \abs{\tr (\rho( g))} d\mu^n(g)$ takes the value
 $-\infty$ infinitely often. To get a less trivial example, if $\rho(G)$ contains a rotation with well chosen angle (i.e. with many  iterates very close to angle $\pi$) then the sequence $\unsur{n}\int  \log\abs{ \tr (\rho( g))} d\mu^n(g)$ may admit cluster values smaller than $\chi(\rho)$. The  same phenomenon of course occurs when considering  $\log \abs{ \tr^2 (\rho( g)) - t}$ for some $t\in [0,4]$. %This simple-minded observation explains why the proof of Theorem \ref{thm:equidist} is so delicate.

\begin{proof}[Proof of the corollary]
By Theorem \ref{thm:guivarch},  $\unsur{n} h(\abs{\tr (\rho( l_n(\mathbf{g})))})\cv \chi(\rho)$ a.s. Furthermore
% $\abs{\tr (\cdot)}\leq C \norm{\cdot}$ so
there exists a constant $C$ such that
$-C\leq h\left( \abs{\tr (\rho( g))}\right) \leq \max( C\log\norm{\rho(g)}, C),$ so the result follows from the Dominated Convergence Theorem.
\end{proof}

\begin{rmk} \label{rmk:uniform}
In the sequel, we often need some uniformity on the estimates \eqref{eq:holder convergence} with respect to $\rho$. To see why such a uniformity is true,
it is instructive to recall how (\ref{eq:condition on distances}) implies  (\ref{eq:holder convergence}). If we set $c_n := \sup_{x\neq y \in \mathbb P^1} 
\mathbb \int \left(\frac{d_{\pu}( \rho (g) x, \rho (g) y) } {d_{\pu}(x,y)}\right)^{\alpha} d\mu^{n} (g)$, then for any function $f$ in $C^{\alpha}$, 
\[ m_{\alpha} ( P^n f)  \leq c_n \cdot m_{\alpha} (f) . \]
Since $\nu$ is stationary, we infer that 
$  \norm{P^n - N}_{C^{\alpha}} \leq c_n$.
Now, it is straighforward to check that $c_{m+n}\leq c_n c_m$ for every pair of  integers $m,n$.  Furthermore, under  the condition (\ref{eq:condition on distances}), we get that  for every integer $n$,  $c_{n} \leq C \cdot e^{-\beta n} $ with $\beta = -\frac{1}{n_0} \log c_{n_0} >0$ and $C= \sup _{k<n_0} c_{k}$, and thus the estimate (\ref{eq:holder convergence}) holds. A useful consequence of this is that  the constants $C$, $\alpha$, and $\beta$ in (\ref{eq:holder convergence}) can be chosen uniformly in a neighborhood of $\rho$, since under our moment assumption $c_n$ depends continuously on $\rho$.
\end{rmk}

\section{The bifurcation current}\label{sec:bifcurrent}

Throughout this section we fix a holomorphic family of representations $(\rho_\la)_{\la\in \Lambda}$ of $G$ into $\PSL$, that we assume to be generally faithful, non-trivial, and for all $\la$, non-elementary.
We fix a probability measure $\mu$ on $G$, generating $G$ as a semi-group, and satisfying \eqref{eq:moment condition finitely generated}. In particular \eqref{eq:moment condition} holds, locally uniformly in $\la$.
From now on, such families of representations (endowed with a measure $\mu$ on $G$)
 will be called {\em admissible}.

For basics on plurisubharmonic (psh for short) functions and positive currents, the reader is referred to
\cite{hormander, demailly}. Recall   that a positive closed current $T$ of bidegree (1,1) locally admits a psh potential $u$,
i.e. $T=dd^cu$ 
and that it is possible to pull-back such currents under holomorphic maps by pulling back the potentials. Another frequently used result is the so-called Hartogs' lemma \cite[pp. 149-151]{hormander} which asserts that families of psh functions with uniform bounds from  above have good compactness properties.

\subsection{Definition}
Given an admissible family of representations $(G,\mu,\rho)$ as above, we let $\chi(\la):= \chi(\rho_\la)$ be the Lyapunov exponent of $\rho_\la$.

\begin{prop}\label{prop:defchi}
The Lyapunov exponent $\chi$ defines a continuous psh function on $\Lambda$, which is pluriharmonic on the stability locus.
\end{prop}

\begin{proof}
For $g\in G$, $\lambda \mapsto \log\norm{g_\la}$ is the supremum of 
the family of psh functions $\la\mapsto \log\norm{g_\la Z_0}$,  where $Z_0$ ranges  over the unit sphere in $\cd$, and it is continuous because the norm is. We thus infer from 
\cite[Thm 4.1.2]{hormander} that $\la\mapsto \log\norm{g_\la}$ is psh.
Hence $\chi$ is psh since by \eqref{def:lyapunov exponent}, it is the pointwise limit of a uniformly bounded sequence of psh functions.

Another proof goes by observing that we can replace $\norm{\cdot}$ by $\norm{\cdot}_2$ in the definition of $\chi$, in which case its plurisubharmonicity is obvious.

The continuity of $\chi$ is a consequence of Furstenberg's formula together with the fact that the stationary measure is unique, and therefore depends continuously on $\lambda$ in the weak topology (see \cite[\S 1.13]{furman}).

Finally, the second assertion of Corollary \ref{cor:guivarch} implies that
$\chi$ is pluriharmonic on the stability locus. Indeed, locally the $g_\lambda$ do not change type there, so the multipliers of fixed points vary holomorphically, without crossing the unit circle (by possibly staying constant of modulus 1), and we infer that $\chi$ is a limit of pluriharmonic functions, hence itself pluriharmonic.
\end{proof}

\begin{defi}
If $(G,\mu,\rho)$ is an admissible family of representations of $G$ into $\PSL$, the bifurcation current $T_{\rm bif}$ is defined by $T_{\rm bif} = dd^c\chi$.
\end{defi}

Proposition \ref{prop:defchi} implies that the support of $T_{\rm bif}$ is contained in the bifurcation locus.

%\begin{rmk} It is easy to construct specific examples of families for which  the bifurcation current is non-zero. For instance,  consider a $1$-parameter holomorphic family of representations $\{  \rho_{\lambda} \}$ such that for a parameter $\lambda_0$ the representation $\rho_{\lambda_0}$ has image in $\text{SO}(3,\mathbb R)$, but for $\lambda\neq \lambda_0$ close to $\lambda_0$ the representation is non-elementary. Of course  $\chi(\lambda_0)= 0$ because the elements of $\text{SO}(3,\mathbb R)$ have norm $\sqrt{2}$. Since for every $g\in G$, $\rho_{\lambda} (g)$ tends to $\sqrt{2}$ when $\lambda $ tends to $\lambda_0$, we infer that $\chi$ is continuous at $\lambda_0$, with a local minimum at $\lambda_0$. 
 
%We deduce that the bifurcation current (here, a measure) of the family does not vanish in any neighborhood of $\lambda_0$. Furthermore,  it cannot be concentrated at $\lambda_0$ for otherwise $\chi$ would have a logarithmic singularity at $\lambda_0$.

%Theorem~\ref{thm:support} below gives a much more systematic approach to this.
%\end{rmk} 

\begin{rmk}\label{rmk:elementary} 
The Lyapunov exponent of an elementary representation is well defined by the formula \eqref{def:lyapunov exponent}. Thus, if the subset of non-elementary representations is not empty, $\chi$ still defines a locally bounded psh function on $\La$, and 
  it makes perfect sense to talk about the bifurcation current also  in this case.
\end{rmk}

%\medskip

We close this subsection by  studying the regularity of the bifurcation current. The continuity of $\chi$ will be a technically useful fact  in the paper.  
For an admissible family satisfying an exponential moment condition, it was shown by Le Page   in  \cite{lepage-holder} that $\chi$ is actually H\"older continuous.
For the reader's convenience, we give a short proof of this result in the case where $\mu$ has finite support. Notice that the key argument is the exponential convergence of the transition operator (Theorem \ref{thm:lepage}).

\begin{thm} [Le Page] 
\label{t:holder dependence}
Let $(G, \mu, \rho)$ be an admissible family of representations, satisfying \eqref{eq:exponential moment} locally uniformly in $\la$ (e.g. satisfying \eqref{eq: exponential moment condition in the group}).
Then the Lyapunov exponent function is  H\"older continuous.
\end{thm}

\begin{proof}[Proof in the finite support case]
Fix a parameter $\lambda_0$, and some small neighborhood $U\ni\la_0$, that we view through a chart as an open set in $\cc^{\dim \La}$. We also consider a norm in $\cc^{\dim \La}$ that we simply denote by $\abs{\cdot}$. For  $\lambda, \lambda' \in U$, write $\rho_{\lambda'} (g) = \rho_{\lambda} (g) + \varepsilon (g)$.
Since $\supp(\mu)$ is finite, when $g\in \supp(\mu)$, $\e(g) = O(|\lambda' -\lambda|)$.  For notational simplicity let us  denote
$g'= \rho_{\lambda'}(g)$ and $g=\rho_{\lambda}(g)$ . For an integer $n$ (that will be chosen of the order of magnitude of $-\log|\lambda '-\lambda|$), we have
\begin{align*}
 (g_n'\ldots g_1') (g_n \ldots g_1)^{-1}
&=
(g_n + \varepsilon _n ) \ldots (g_1+ \varepsilon_1 ) (g_1^{-1} \ldots g_n^{-1}) \\
&=
(I+ \varepsilon_n g_n^{-1}) (I + Ad(g_n) (\varepsilon _{n-1} g_{n-1} ^{-1}) ) \ldots (I + Ad(g_n\ldots g_{2}) (\varepsilon _1 g_1^{-1} ) ). 
\end{align*}
We want to estimate the distance between the latter matrix and the identity. First observe that there exists a constant $M>1$ (the square of the maximum of the norms of the elements $\rho_{\lambda} (g) $ with $\lambda\in U$ and $g\in \supp(\mu)$) such that
\begin{equation} \label{eq:norm bound} \norm{Ad (g_n \ldots g_{k+1} ) (\varepsilon_k g_k^{-1} )} \leq    M^{n-k} | \lambda'-\lambda | .\end{equation}
For $1\leq k\leq n$, let  $u_k= \norm{Ad (g_n \ldots g_{k+1} ) (\varepsilon_k g_k^{-1} )}$ and $v = (I+u_1) ... (I+ u_n) - I$. Since $k\leq n$, and $M>1$, \eqref{eq:norm bound} shows that $||u_k||\leq M^n |\lambda - \lambda'|$. Expanding $v$ we obtain that
\[  \norm{v} =\norm{ \sum _{0< k\leq n}  \sum _{i_1< ...< i_k} u_{i_1}... u_{i_k} }\leq
\sum _{0< k\leq n} C_n^k M^{kn} |\lambda - \lambda '|^k  = (1+x_n)^n -1\leq nx_n(1+x_n)^{n-1},\] where
  $x_n= M^n |\lambda - \lambda'|$.
% the estimates
% \[ ||v||
% \leq
% cst \cdot x  \sum _{k=1}^n C_n^k x^{k-1}
% =
% cst \cdot x \sum _{k=0}^{n-1} C_n ^{k+1} x^k.\]
% We bound the right hand side by making use of the rough estimates $C_n^{k+1} \leq n C_{n-1} ^k$ and we deduce that
% \[ ||v|| \leq cst \cdot n x (1+x)^n .\]
Now choose $n$ so that  $nx_n = |\lambda - \lambda'|^{1/2}$,  that is,
$ n M^n = |\lambda - \lambda ' |^{-1/2}$.
The quantity $(1+x_n)^{n-1}$ is bounded since $nx_n$ is  constant. Thus we get that
\[ \norm{v}=\norm{(g_n'\ldots g_1')(g_n \ldots g_1)^{-1} -I} \leq C^{\rm st}  |\lambda - \lambda'|^{1/2}. \]

% and
% \[  n \sim \frac{- \log |\lambda - \lambda'| } { 2 \log M} .\]
%The norm of each term in the development of \eqref{eq:align}, but the identity, can be bounded using the last inequality. This gives
%\[ \norm{ (g_n'\ldots g_1')(g_n \ldots g_1)^{-1} -I } \leq \mathrm{cst}\cdot  \sum_{1\leq k \leq n} C_n^k M^{kn} |\lambda' -\lambda|^k \leq \mathrm{cst} \cdot (2M)^n |\lambda' -\lambda|, \]
%as soon as $|\lambda' -\lambda|$ is less than $1/M$. We choose $n$ so that up to multiplicative constants $(2M)^n \asymp |\lambda' - \lambda|^{-1/2}$, and  we infer that
%\[ \norm{(g_n'\ldots g_1')(g_n \ldots g_1)^{-1} -I} \leq \mathrm{cst} \cdot  |\lambda' -\lambda|^{1/2}. \]

Note that there is a constant $C^{\rm st}$ such that for every $g\in \mathrm{PSL}(2,\mathbb C)$ and every $y \in \mathbb P^1$, we have $d_\pu(g(y), y) \leq C^{\rm st} \cdot \norm{g-I}$.   Thus, for every $x\in \mathbb P^1$, we have that
\[ d_\pu ( (g' _n\ldots g' _1) x , (g_n \ldots g_1) x) = d_{\pu} ( ( g'_n\ldots g'_1 ) (g_n\ldots g_1)^{-1} y, y) \leq  C^{st} |\lambda ' -\lambda |^{1/2}, \]
by denoting $y= g_n \ldots g_1 (x)$. As a consequence, if $f$ is a H\"older continuous function of exponent $\alpha$,
\[ |P_{\lambda'}^n f (x) - P_{\lambda}^n f(x) | \leq C^{\rm st} |\lambda' -\lambda |^{\alpha/2} \norm{f}_{C^\alpha}. \]

We actually need to apply the latter estimate  for a function which    also depends on $\lambda$, but in a differentiable way, namely,
\[  f_{\lambda} (x) = \int \log \frac{\norm{\rho_{\lambda}(g) X}}{\norm{X}} d\mu(g)\  (X \text{ a lift of }x). \]
There exists some constant for which in the given neighborhood of $\lambda_0$, we have $\norm{f_{\lambda'} -f_{\lambda}}_{\infty}\leq C^{\rm st} |\lambda' -\lambda|$. We can thus write $\norm{ P_{\lambda'} ^n f_{\lambda'} - P_{\lambda}^n f_{\lambda}}_{\infty} \leq \norm{P_{\lambda '}^n  (f_{\lambda'} - f_{\lambda}) }_{\infty} + \norm{(P_{\lambda'}^n - P_{\lambda}^n)  f_{\lambda}}_{\infty}$ and consequently
 \begin{equation}\label{eq:holder1}
\norm{ P_{\lambda '}^n f_{\lambda'} - P_{\lambda}^n f_{\lambda}}_{\infty} \leq C^{\rm st} |\lambda ' - \lambda | ^{\alpha /2}.
\end{equation}

To finish the proof, notice that our choice of $n$ implies that
\[  n = \frac{- \log |\lambda - \lambda'| } { 2 \log M} + O\left(\log \abs{\log \abs{\lambda - \lambda'}}\right) \sim \frac{- \log |\lambda - \lambda'| } { 2 \log M} .\]
Therefore, by 
 the exponential convergence~(\ref{eq:holder convergence})
  of $P_{\lambda}^n$ towards $N_{\lambda}$, we obtain that for   $\la\in U$,
\begin{equation}\label{eq:holder2}
  \norm{(P_{\lambda} ^n -N_{\lambda}) f_{\lambda}  }_{\infty} \leq C^{\rm st} |\lambda ' -\lambda | ^{\gamma}
\end{equation}
for any $\gamma < \frac{\beta}{2\log (M)}$ (recall from Remark \ref{rmk:uniform} that $\beta$ is locally uniform). By Furstenberg's formula~(\ref{eq:Furstenberg's formula}),
 $N_{\lambda} f_{\lambda}=\chi_{\lambda}$, so we conclude by summing \eqref{eq:holder1} and \eqref{eq:holder2} that $\chi$ is H\"older continuous in $U$, of exponent $\gamma$, for any $\gamma < \min (\alpha /2 , \frac{\beta}{2\log (M)} )$.
\end{proof}

\subsection{Geometric interpretation}\label{subs:geom_interpt}
We now consider the fibered action of $G$ on $\La\times \pu$, that is, for $g\in G$,
we define $\widehat{g}$ by $\widehat{g}:(\lambda, z)\mapsto (\lambda, g_{\la}(z))$.
 If $p\in \pu$, we let  $\widehat{g}\cdot p := \set{(\la,g_\la (p)), \la\in \La}$.
 More generally,  objects living
in $\La\times\pu$ are marked with a hat.
Let also $\pi_1$ and $\pi_2$  be the respective coordinate projections from $\La\times \pu$ to $\La$ and $\pu$.

The following theorem gives a geometric characterization of the bifurcation current.

\begin{thm}\label{thm:geom_interpt}
Let $(G,\mu,\rho)$ be an  admissible family of representations of $G$ into $\PSL$. Fix $z_0\in \pu$, and
   for every $n$ define a  current of bidegree $(1,1)$ on $\La\times\pu$ by the formula
$$\widehat{T}_n = \unsur{ n} \int \left[\widehat{g}\cdot z_0\right] d\mu^n(g).$$
Then the sequence $(\widehat{T}_n)$ converges to $\pi_1^*(T_{\rm bif})$.
\end{thm}

This implies that $\lambda_0\in \supp (T_{\rm bif})$
 if and only if for every neighborhood $U\ni\lambda_0$, the average volume
of $\widehat{g\cdot p_0}\cap \pi^{-1}(U)$, relative to $\mu^n$, grows linearly in $n$.
If $U$ is contained in the stability locus, it is easy to show that $\bigcup_{g\in G}\widehat{g\cdot p_0}\cap \pi_1^{-1}(U)$ is a normal family of graphs over any relatively compact subset of $U$, hence $\widehat{T}_n\cv 0$ in $\pi_1^{-1}(U)$.
We thus obtain an alternate proof of the fact that $\supp(T_{\rm bif})\subset \mathrm{Bif}$.

\begin{proof}
This is a local result on $\La$, so we may assume that $\La$ is a ball in $\cc^k$, endowed with its standard K\"ahler form $\omega$. Let also $\omega_\pu$ be the Fubini-Study form on $\pu$ associated to our choice of Hermitian norm. On $\Lambda\times \pu$ we choose the K\"ahler form  $\widehat\omega:= \pi_1^*\omega+ \pi_2^*\omega_\pu$.

The first observation is that $\langle{\widehat{T}_n, \pi_1^*\omega^k}\rangle\cv 0$. 
Indeed, for $g\in G$, since $\widehat{g}\cdot z_0$ is a graph,
$\left[\widehat{g}\cdot z_0\right]\wedge \pi^*\omega^k  = \int_\La \omega^k$, so
$\langle{\widehat{T}_n, \pi_1^*\omega}\rangle =O\left( \unsur{n}\right)$. 
Thus, if we can show that $(\widehat{T}_n)$ has locally uniformly bounded mass,
every cluster value $\widehat{T}$ of this sequence satisfies $\widehat{T}\wedge \pi^*\omega^k  = 0$. In this case it is  classical that $\widehat{T}$ does not depend on the $\pu$ coordinate, in the sense that there exists a current $T$ on $\La$ such that  $\widehat{T} = \pi_1^*T$. For completeness we sketch a proof of this fact in Lemma \ref{lem:fibration} below. 
Now,  $\widehat{\omega}^k$ is equal to
$\pi_1^*\omega^k+ k\pi_1^*\omega^{k-1} \wedge \pi_2^*\omega_{\pu}$, therefore, since $\langle{\widehat{T}_n,  \pi_1^*\omega^k}\rangle \cv 0$ we are led to understand 
pairings of the form $\langle \widehat{T}_n,  \pi_2^*\omega_{\pu}\wedge \pi_1^*\phi\rangle$, where $\phi$ is a $(k-1, k-1)$ test form on $\La$, or equivalently, to understand
 $(\pi_1)_*( \widehat{T}_n \wedge \pi_2^*\omega_{\pu})$.

For this we compute 
\begin{align} \label{eq:chapeau}
\langle \widehat{T}_n,  \pi_2^*\omega_{\pu}\wedge \pi_1^*\phi\rangle
&=  \unsur{ n} 
\int  \bra{ \left[\widehat{g}\cdot z_0\right],\pi_2^*\omega_\pu\wedge \pi_1^*\phi}  d\mu^n(g)\\ \notag
&= \unsur{ n} 
\int \lrpar{ \int_\La \left( \pi_1\rest{\widehat{g}\cdot z_0}\right)_*\left(\pi_2^*\omega_\pu \rest{\widehat{g}\cdot z_0}\right) \wedge \phi }d\mu^n(g)\\ \notag
&= \unsur{ n}
\int \lrpar{\int_\La \left( \pi_2\circ \left( \pi_1\rest{\widehat{g}\cdot z_0}\right)^{-1}\right)^*\omega_\pu\wedge \phi }d\mu^n(g),
\end{align}
where in the second line we use the fact that for every $g$, $\pi_1\rest{\widehat{g}\cdot z_0}$ is a biholomorphism. Now observe that for
 $g\in G$, the map $\pi_2\circ \left( \pi_1\rest{\widehat{g}\cdot z_0}\right)^{-1}$ is just defined by the formula $\lambda\mapsto  g_\la(z_0)$. Denote it by $h$. We have that
 $h^*\omega_\pu =  dd^c \log\norm{H},$
where $H:\Lambda \cv\cd\setminus\set{0}$ is any lift of $h$. Denote  $g= \left(\begin{smallmatrix} a&b\\c&d \end{smallmatrix}\right)$ as before, %and let by definition $\norm{gz_0} = \norm{gZ_0}$, where $Z_0\in \cd$ is any lift  of $z_0$ of unit norm. Writing in coordinates
$z_0 = [x_0: y_0]$,   and $Z_0 = (x_0, y_0)$ be a lift of $z_0$ to $\cd$.
We infer that
$$h^*\omega_\pu = dd^c \log\left( \abs{a_\la x_0 + b_\la y_0}^2 +\abs{c_\la x_0 + d_\la y_0}^2 \right)^{\unsur{2}} =
dd^c \log \norm{g_\la(Z_0)}  = dd^c \log \frac{\norm{g_\la(Z_0)}}{\norm{Z_0}} ,$$ where the $dd^c$ takes  place  in the $\lambda$ variable.
 We conclude that
\begin{equation}\label{eq:pi etoile}
\langle \widehat{T}_n,  \pi_2^*\omega_{\pu}\wedge \pi_1^*\phi\rangle = \int_\La \phi\wedge
dd^c_\la \left( \unsur{n} \int\log \frac{\norm{g_\la(Z_0)}}{\norm{Z_0}} d\mu^n(g) \right).
\end{equation}
By Theorem \ref{t: furstenberg} 
 for every $z_0\in \pu$ and every $\lambda$,
$$ \unsur{n} \int\log\frac{\norm{g_\la(Z_0)}}{\norm{Z_0}} d\mu^n(g)
\underset{n\cv\infty}\longrightarrow \chi(\la).$$
Furthermore  by the  subadditivity of $\norm{g_\la}$  %\note{on se sert de la norme d'algèbre....}
and the uniform moment condition \eqref{eq:moment condition finitely generated}, 
this sequence is
 locally uniformly bounded above (with respect to  $\lambda$), hence by the Hartogs Lemma
the convergence holds in $L^1_{\rm loc}$ and we finally obtain that 
$\lim_n (\pi_1)_*(\widehat{T}_n \wedge \pi_2^*\omega_\pu) = dd^c\chi =T_{\rm bif}$. 

Applying this to $\phi= \varphi\omega^{k-1}$, where $\varphi$ is a cutoff function, 
we see that the sequence $(\widehat{T}_n)$
has locally uniformly bounded
mass. Let $\widehat{T}$ be any of its  cluster values. We know that it is of the form $\pi_1^*T$ and that
 $(\pi_1)_*(\widehat{T} \wedge \pi_2^*\omega_\pu)=T_{\rm bif}$.
The following classical computation then finishes the proof.\end{proof}

\begin{lem}
With notation as above, $(\pi_1)_*(\pi_1^*T \wedge \pi_2^*\omega_\pu) =T$
\end{lem}

  \begin{proof} Let  $\phi$ be a test $(k-1, k-1)$ form in $\La$. We have that
$$\bra{ (\pi_1)_*(\pi_1^*T \wedge \pi_2^*\omega_\pu) , \phi} = \int \pi_1^*(T\wedge \phi) \wedge \pi_2^*\omega_\pu.$$ Now, $T\wedge \phi$ is a current of bidegree $(k,k)$ with compact support in $\La$, so it is cohomologous (with compact supports) to $\left(\int T\wedge \phi\right)\Theta$, where $\Theta$ is any compactly supported positive smooth $(k,k)$ form of integral 1. So we deduce that
$$\int \pi_1^*(T\wedge \phi) \wedge \pi_2^*\omega_\pu = \left(\int T\wedge \phi\right)
\int  \pi_1^*\Theta \wedge \pi_2^*\omega_\pu=\int T\wedge \phi,$$ and we are done.
\end{proof}

As promised above, we sketch the proof of the following classical fact.

\begin{lem}\label{lem:fibration}
Let $B_1 \times B_2 \subset \cc^k\times \cc$ be a product of balls. Write $z = (z',z_{k+1})$ for the coordinate on  $\cc^{k+1}$ and denote by $\omega_1$ the standard K\"ahler form on $\cc^k$. Let $T$ be a positive closed current of bidegree (1,1) on 
$B_1 \times B_2$, such that $T\wedge \omega_1^k = 0$. Then there exists a positive closed current $T_1$ on $B_1$ such that $T = \pi_1^*T_1$, $\pi_1$ being the first projection. 
\end{lem}
 
\begin{proof}
Decompose $T$ in coordinates as $T = i\sum T_{i,j}  dz_i\wedge d\overline {z_j}$, where $(T_{i,j})$ is a Hermitian matrix of measures. Since $T\wedge \omega_1^k = 0$, $T_{k+1,k+1} =0$. Positivity implies  that $T_{k+1,j} = 0$ for all $j$ (see \cite[Prop. 1.14]{demailly}). Then closedness implies that the $T_{i,j}$, $i,j\leq k$ do not depend on $z_{k+1}$ (see the proof of Theorem 2.13 in \cite{demailly}). The result is proved. 
\end{proof}

In the next --presumably well-known-- proposition we give  an estimate for the speed of convergence of the potentials appearing in the proof of Theorem~\ref{thm:geom_interpt}. This will play a crucial role in Theorem \ref{thm:support}.  

\begin{prop}\label{prop:Ounsurn}
Let $(G,\mu,\rho)$ be an admissible family of representations of $G$ into $\PSL$,
satisfying the exponential moment condition \eqref{eq: exponential moment condition in the group}. Then 
for every $z_0\in \pu$ we have that
\begin{equation}\label{eq:Ounsurn}
 \unsur{n} \int\log\frac{\norm{g_\la(Z_0)}}{\norm{Z_0}} d\mu^n(g)
= \chi(\la) + O\left(\unsur{n}\right),
\end{equation}
where the $O(\cdot)$ is locally uniform in $\la$.  
\end{prop}

\begin{proof}
Fix $\la$ for the moment and let $f$ be the function on $\pu$ defined by
$f(z)= \int\log\frac{\norm{g(Z)}}{\norm{Z}} d\mu(g)$ (we drop the $\la$ from the formulas). Recall Furstenberg's formula  that
$\int f d\nu  = \chi$, where $\nu$ is the unique stationary measure.
We have that
\begin{align}
\notag\int\log\frac{\norm{g(Z_0)}}{\norm{Z_0}} d\mu^n(g) -n\chi
&= \sum_{k=1}^n\left(\int\log\frac{\norm{gg_{k-1}\cdots g_1(Z_0)}}{\norm{g_{k-1}\cdots g_1(Z_0)}}d\mu(g)d\mu(g_{k-1})\cdots
 d\mu(g_1) -\chi\right) \\
&=\label{eq:sum}\sum_{k=1}^n \left( P^{k-1} (f)(z_0) -\int f d\nu \right),
\end{align}
furthermore, under the moment condition \eqref{eq:exponential moment}, we know that there exist constants $C>0$, and $\beta<1$ such that $\norm{P^k (f)- \int f d\nu}_{L^\infty}<C\beta^k$.  We thus conclude that the  sum in \eqref{eq:sum} is bounded as $n\cv\infty$,    yielding the desired estimate for fixed $\la$.
For the uniformity statement, just recall from Remark \ref{rmk:uniform} that the values of  $C$ and $\beta$ are locally  uniform in  $\la$.
\end{proof}

\subsection{The support of $T_{\rm bif}$}\label{subs:support}

We keep  hypotheses as before, keeping in particular from the last proposition the exponential moment condition  
\eqref{eq: exponential moment condition in the group} (it would actually be enough to assume that 
\eqref{eq:exponential moment} holds locally uniformly).
 
Here is the precise statement of the characterization of the support of $T_{\rm bif}$.

\begin{thm}\label{thm:support}
Let $(G,\mu,\rho)$ be an admissible family of representations of $G$ into $\PSL$,
satisfying the exponential moment condition \eqref{eq: exponential moment condition in the group}.

Then the support of $T_{\rm bif}$ coincides with the bifurcation locus.
\end{thm}

We know from Proposition~\ref{prop:defchi} that $\supp(T_{\rm bif})\subset \mathrm{Bif}$ so only the reverse inclusion needs to be established.  For this, we make use of the geometric interpretation of $\tbif$ given in  \S\ref{subs:geom_interpt}.

\medskip

Since it will play a very important role in the proof,  let us start by reviewing the construction of the Poisson boundary of $(G,\mu)$ (see \cite{kaimanovich-gafa} for more details).  Consider the right random walk on the group $G$, defined as the Markov chain on $G$ with transition probabilities given by $p(x,y) = \mu (x^{-1} y )$. %Define $P ( G, \mu)$ to be the Poisson boundary of this Markov process. 
As a measurable space, the Poisson boundary is the set of paths $(r_n) \in G^{\nn}$, equipped with the tail algebra $\mathcal T$, that is the algebra of Borel sets in $G^{\nn}$ which are invariant by the shift $\sigma (r_n) = (r_{n+1})$. Hence, two paths $(r_n ) $ and $(r_n')$ have to be considered as equivalent in the Poisson boundary as soon as they have the same   tails. We denote the Poisson boundary by $P(G,\mu)$. 

It inherits a measure class induced by the $\mu$-random walk on $G$, as follows. Recall that the position $r_n$ ($n\geq 1$) of the random walk at time $n$ is deduced from its position at time $0$ by the following formula 
\begin{equation} \label{eq: random walk} r_n = r_0 h_1 \ldots h_n,\end{equation}
where the $h_i$ are mutually independant random variables with distribution $\mu$. Any initial distribution $\theta$ on $G$ determines a Markov measure $\mathsf P_{\theta}$ on the space of paths $(r_n)\in G^{\nn}$: the image of $\theta \otimes \mu ^{\nn}$ under the assignment $(r_0, (h_i))\mapsto (r_n)$ given in   \eqref{eq: random walk}. We shall denote by $\mathsf{P}_g$ the Markov measure corresponding to the Dirac mass at the point $g$. It is straighforward to verify that $\mathsf{P}_e$ is $\mu$-stationary, i.e. $\mathsf{P}_e = \int \mathsf{P}_g d\mu(g)$. %In particular 
Note also that the measures $\mathsf{P}_g$ are absolutely continuous with respect to each other.

The coordinate-wise left multiplications by an element of the group $G$ on $G^\nn$ commutes with the shift and induces an action of $G$ on $P(G,\mu)$. The measure $\mathsf{P}_e$ is pushed by an element $g$ of $G$ on the measure $\mathsf{P}_g$, so that the measure $\mathsf{P}_e$ is quasi-preserved by this action. 

\begin{proof}[Proof of Theorem \ref{thm:support}]
It is no loss of generality to assume that $\dim( \Lambda)=1$.
Let $V\subset \La$ be an open subset where $\tbif$ vanishes, and $U\Subset V$. We want to show that $U$ is contained in the stability locus. Fix $z_0\in \pu$ and define $\widehat{T}_n$ as in  Theorem \ref{thm:geom_interpt}. More generally we keep notation as in  the proof of that theorem.

\medskip

\noindent{\bf Step 1.} The vanishing of $\tbif$ on $U$ implies the  mass estimate $\displaystyle{ \int_{\pi_1^{-1}(U)} \widehat{T}_n  \wedge \widehat{\omega}= O\left(\unsur{n}\right) }$.
%$\displaystyle{\m\left(\widehat{T}_n\rest{\pi_1^{-1}(U)}\right)= O\left(\unsur{n}\right)}$.

Indeed, recall that $\widehat\omega =
\pi_1^*\omega+  \pi_2^*\omega_{\pu}$. We already observed that $\langle{\widehat{T}_n, \pi_1^*\omega}\rangle =O\left( \unsur{n}\right)$. For the second term, the computations in the proof of Theorem \ref{thm:geom_interpt} show that
\begin{align*}
\bra{ \widehat{T}_n\rest{\pi_1^{-1}(U)}, \pi_2^*\omega_{\pu}}
&=\int_U dd^c\chi_n  \text{ , where } \chi_n =  \unsur{n} \int\log\frac{\norm{g_\la(Z_0)}}{\norm{Z_0}} d\mu^n(g)\\
&=\int_U dd^c(\chi_n -\chi)  \text{ , because }\chi \text{ is harmonic on }U.
 \end{align*}
Since by Proposition \ref{prop:Ounsurn}, 
 $\norm{\chi_n-\chi}_{L^\infty} = O\left(\unsur{n}\right)$, introducing a cut-off function and integrating by parts shows that the last integral is  $O\left(\unsur{n}\right)$, which was the result to be proved. We note that    a similar argument  appears in   \cite[Thm 3.2]{df}.

\medskip

\noindent{\bf Step 2.} Construction of a holomorphic equivariant map from the Poisson boundary (a partially defined map is extended using Bishop's theorem and the mass estimate).

 %Let $G^\nn$ be the space of paths of the random walk on $G$, and $\mathsf P$ be the probability measure on $G^\nn$ corresponding to the walker being at the identity at step 0. To be specific, $(G^\nn,\mathsf{P})$ is simply the image of  $(G^\nn,\mu^\nn)$ under the map $(h_i)_{i\geq 1}\mapsto (g_i)_{i\geq 0}$ defined by $g_0 = id$ and $g_{n+1} = g_n h_{n+1}$. The group $G$ acts on the path space $G^\nn$ by coordinate-wise left multiplication. We usually denote paths by the letter $w$, and let $r_n= g_1\cdots g_n$ be  the $n^{\rm th}$ coordinate of $w$. Also as usual $r_{n,\lambda}$ stands for $\rho_\la(r_n)$.

% Consider the right random walk on the group $G$, defined as the Markov chain on $G$ with transition probabilities given by $p(x,y) = \mu (x^{-1} y )$. Let $P ( G, \mu)$  be the Poisson boundary of this Markov process. This is a measurable space equipped with a class of measures, on which the group $G$ acts by measurable transformations preserving the measure class (see below for more details).

\medskip
 
\begin{lem} \label{l: equivariant}
There exists a measurable map $\theta: P (G,\mu) \times U \cv \pu$, defined almost everywhere with respect to the first factor, which is holomorphic with respect to
 the second variable, and $G$-equivariant with respect to the first. 
\end{lem}

To prove the lemma, using notation as in the paragraph preceding the proof of the theorem, 
we will define the map $\theta$ on $\Omega\times U$ for a measurable subset $\Omega\subset \set{e}\times G^{\mathbb N}$ of full $\mathsf{P}_e$ measure, and then we will extend it on the union $\bigcup _{g\in G} g\Omega \times U$ by the formula $\theta (r, \lambda) = \rho(r_0)_{\lambda} \cdot \theta (r_0^{-1} r )$. To verify that the extension is $G$-equivariant and tail invariant --thus defining the desired equivariant map on the Poisson boundary-- it will be sufficient to check that $\theta$ depends only on the tail of the first variable, and satisfies %\note{y a t'il une subtilité entre   $r_0$ au dessus et $r_1$ ici? oui c un peu troublant, mais je crois que c bon! pour la première on ramène une suite $r$ à une suite qui commence en e en la multipliant par $r_0^{-1}$, et pour la seconde on décale ce qui ramène $r_1$ en position 0 et on multiplie ensuite par $r_1^{-1}$}
\begin{equation} \label{eq: equivariance} \theta ( r  ,\lambda) = \rho(r_1)_{\lambda}\cdot  \theta ( r_1^{-1} \sigma (r), \lambda) ,\end{equation}
for every $r\in \Omega$.

%for every $g\in G$, $r\in \Omega$ and $\lambda \in U$. To do so, we will only define $\theta$ for paths $r$ beginning at $r_0=e$, and verify that  $\theta (\sigma (r) ,\lambda) = \rho(r_1)_{\lambda} \theta ( r_1^{-1} \sigma (r), \lambda)$ for such paths; this clearly permits to extend the map over the whole paths space $G^{\nn}$ as a map satisfying~\eqref{eq: equivariance}.

\begin{proof}[Proof of  Lemma \ref{l: equivariant}]
Recall from~\cite[Corollary 7.1, p. 40]{bougerol-lacroix} that if $\lambda$ is fixed, a map as in~\eqref{eq: equivariance} exists and is unique. More precisely, for $\mathsf{P}_e$-a.e. $r \in G^\nn$, the sequence $r_n (z_0)$ converges to a point $z_{r}$, independent of $z_0$, and the law of $z_{r}$ is $\nu$. So if we fix a dense sequence $(\lambda_q)$ in $U$, we obtain a set $\om_1\subset G^\nn$ of full $\mathsf P_e$-measure such that for every $r\in \Omega_1$, and every $q$, $  {r}_{n,\lambda_q}  (z_0)$ converges to some $z_{r, \lambda_q}$. The proof of lemma~\ref{l: equivariant} consists in showing that we can interpolate this function of $\lambda_q$ by a holomorphic function of $\lambda$.

%We know from Step 1 that  there exists a constant $C$ such that
%$$\int \vol\left ( \widehat{g}\cdot z_0\right) d\mu^n(g)  %=\int \vol\left ( \widehat{g_1\cdots g_n} \cdot z_0 \right)  d\mu(g_1)\cdots d\mu(g_n) 
%\leq C$$

We claim that there exists $\om_2\subset G^{\nn}$ of full $\mathsf P_e$-measure such that
 for $r\in \Omega_2$ there exists a subsequence $n_j$ such that
 $\vol\left ( \widehat{r}_{n_j}\cdot z_0 \right) 
% \vol\left ( \widehat{g_1\cdots g_{n_j}} \cdot z_0 \right)
$ 
is bounded. Indeed, since   the random walk is conditioned to start at $r_0=e$, the distribution of $r_n$ is $\mu^n$, hence from Step 1 we deduce that there is a constant $C$ such that for every $n$,
$$ \int \vol \left( \widehat {r}_n \cdot z_0 \right ) d\mathsf P_e (r) \leq C.$$
Our claim now follows from the following elementary argument: for $r\in G^\nn$, let $\varphi_n(r) = \vol\left (  \widehat{r}_n \cdot z_0 \right)$ and $\psi_n(r) = \inf_{k\geq n} \varphi_n(r)$. The sequence $\psi_n$ is increasing, so
$$C\geq \limsup \int \varphi_n d\mathsf P_e \geq \lim \int \psi_n d\mathsf P_e =\int \lim\psi_n d\mathsf P_e,$$
and we conclude that $\lim \psi_n$ is a.s. finite, which was the  desired result. 

Recall that if $f_n:U\cv\pu$ is a sequence of holomorphic mappings such that the corresponding graphs
have uniformly bounded volume, then by Bishop's Theorem \cite[\S 15.5]{chirka} it admits a convergent subsequence, up to finitely many ``vertical  bubbles",  that is, there exists a subsequence $n_j$, $f:U\cv\pu$ and a finite subset $E$ in $U$ such that $f_{n_j}$ converges to $f$ uniformly on compact subsets of $U\setminus E$.

Now by putting $\om =\Omega_1\cap \Omega_2$, we are able to construct the mapping $\theta$. Indeed if $r \in \om$,
there exists a subsequence $n_j$ such that the sequence of graphs $\widehat{r}_{n_j} \cdot z_0 $ has  bounded volume. Extracting again if necessary we may assume that it converges, up to possibly finitely many bubbles. Let $f:U\cv\pu$ be the limit.
Then for all but possibly finitely many $\lambda_q$ (where bubbling occurs),  we have that $f(\lambda_q) = z_{r , \lambda_q}$. Thus the assignment $\lambda_q\mapsto z_{r, \lambda_q}$ admits a (necessary unique since $(\lambda_q)$ is dense) continuation as a holomorphic mapping $U\cv\pu$. Likewise,  $f$ is the only possible cluster value of $\widehat{r}_{n} \cdot z_0$ since it is determined by the $z_{r, \lambda_q}$. We can thus  define $\theta(r,\la)$ to be $f(\lambda)$. The same argument shows that this function depends only on the tail of $r$, and satisfies~\eqref{eq: equivariance}.\end{proof}

It is straightforward that the image  of $\mathsf{P}_e$ under an equivariant map is stationary. From this we get the following statement, which will be used at several places below. 

\begin{lem}
 Let $\theta$ be the mapping constructed in Lemma~\ref{l: equivariant}. Then $(\theta(.,\lambda))_ *\mathsf{P}_e$ is the stationary measure $\nu_{\lambda}$ on $\mathbb P^1$.
\end{lem}

\begin{rmk}\label{rmk:speed}
This argument shows that the estimate in Proposition \ref{prop:Ounsurn} cannot be substantially improved. Indeed, assume on the contrary that the $O\lrpar{\frac1n}$ in \eqref{eq:Ounsurn} can be replaced by $o\lrpar{\frac1n}$. Then we  infer that over the stability locus,  the average projected volume of $\widehat{r}_n\cdot z_0$ on the second factor (i.e. $\pu$)    tends to zero. Therefore, the limiting graphs are  horizontal lines, and the limit set does not depend on $\La$, i.e. the family of representations is constant. 
\end{rmk}

\medskip

\noindent {\bf Step 3.} Improving the equivariant map to a holomorphic motion of $\nu_\la$ (Double ergodicity, a reflected random walk, and the persistence of isolated intersections  are  used to rule out collisions between the holomorphic graphs). 

\medskip

More precisely here we show that there exists a discrete subset $F\subset U$,  such that, outside $F$, the support of the stationary measure moves holomorphically. Furthermore, this holomorphic motion is $G$-invariant.
With $\theta$ as in Step 2, we define  $\theta_r\subset U\times \pu$
to  be the graph of $\theta(r,\cdot)$.

The {\em reflected measure} $\check\mu$ is the push-forward of $\mu$ under $g\mapsto g^{-1}$. The associated Lyapunov exponent $\check\chi(\la)$ actually equals $\chi$ since for $g\in \PSL$, $\norm{g} = \norm{g^{-1}}$. In particular, $\check T_{\mathrm{bif}}= 0$ in $U$. Finally, we can define a map $\check \theta : P(G,\check{\mu}) \times U \rightarrow \mathbb P^1$ as in lemma~\ref{l: equivariant} and the equivariant family of graphs $\check\theta_{\check r}$ associated to it.
%the measure $\check{\mathsf{P}}$ on the space of paths $G^\nn$, and an equivariant family of graphs $\check\theta_{\check r}$ associated to it.

On the product $P(G,\mu)\times P(G,\check{\mu})$ we fix the measure class of the product of the Markov measures starting from $e$ on the corresponding Poisson boundaries. Let  $D\Subset U$ be a set, and $\iota_D: P(G,\mu) \times P(G,\check{\mu}) \cv \nn$ the measurable map defined almost everywhere by letting
$\iota_D(r,\check r)$ be the number of isolated intersection points (with multiplicity) of the graphs
$\theta_r$ and $\check\theta_{\check r}$ in $\pi_1^{-1}(D)$. This number is finite since
 $D\Subset U$, and is invariant under the diagonal action of $G$, i.e. $\iota_D(r,\check r) = \iota_D(gr,g\check r)$, by equivariance of the maps $\theta$ and $\check{\theta}$. By the double ergodicity theorem of Kaimanovich~\cite{kaimanovich-gafa}, $\iota_D$ is a.e. equal to a constant, which will simply be denoted by $\iota_D$.

We see that $D\mapsto \iota_D$ defines an integer-valued measure, which is finite on relatively compact subsets.  It is then  straightforward
to show that it must be a sum of Dirac masses with integer coefficients, supported on a  discrete set $F$.

\medskip

For the reader's convenience, let us recall the idea of the proof of double ergodicity. To a given   bi-infinite sequence $ h = (h_n)_{n\in \zz}$, we associate two sequences $h^+=(h_n)_{n\geq 1}$ and $h^- = (h_{-n}^{-1})_{n\geq 0}$, hence two points $r= r(h)$ and $\check r = r(\check{h})$ in the  respective Poisson boundaries $P(G,\mu)$ and $P(G,\check{\mu})$. The map $h\mapsto (r,\check{r})$ sends
 the measure $\mu^{\zz}$ on $G^{\zz}$ to a measure in the measure class of $P(G,\mu)\times P(G,\check{\mu})$. Now, if $\sigma (h_n)  =  (h_{n+1})$ is the bilateral shift acting on $G^{\zz}$, we have the immediate formulas $r(\sigma h) = h_1^{-1} r(h)$ and $\check{r}(\sigma h) = h_1^{-1} \check{r}(h)$. Hence we deduce that the function $\iota_D ( r,\check{r})$ on $G^{\zz}$ is invariant under the bilateral shift, hence constant $\mu^{\zz}$-a.e. by ergodicity. We conclude  that $\iota_D$ is almost everywhere constant on $P(G,\mu)\times P(G,\check{\mu})$.

%hence two paths $r = (h_1\cdots h_n)_{n\geq 1}$ and $\check r = (h_0^{-1}\cdots h_{-n}^{-1})_{n\geq 1}$. In this way we obtain a map $\Psi$ from  $G^\zz$ to $G^\nn \times G^\nn$ mapping the measure $\mu^\zz$ to a measure absolutely continuous with respect to $\mathsf P \times \check{\mathsf P}$. Now if $T$ denotes the left shift on $G^\zz$, we have that
%$$\Psi(Th) =
%\left( (h_2\cdots h_n)_{n\geq 2}, ( h_1^{-1} h_0^{-1}\cdots h_{-n}^{-1})_{n\geq 0}\right)
% = h_1^{-1}\cdot \Psi(h). $$ We conclude that $\iota_D$ is constant a.e. Indeed,
%  observe that the level sets $\set{\iota_D = cst}$ are
% invariant under the diagonal action of $G$, hence  they correspond
%  under $\Psi^{-1}$ to shift invariant subsets in $G^\zz$. Of course the shift on $G^\zz$ is ergodic, hence the result.

\medskip

 Fix an open subset $D$ disjoint from $F$. Reducing $D$ if necessary, we can find three disjoint graphs in the family $\set{\theta_r, \ r\in \om'}$. Indeed, if $\lambda_q\in D$ is a parameter from the dense sequence considered in the proof of Lemma~\ref{l: equivariant}, we know that
$(\theta(\cdot, \lambda_q))_*\mathsf{P}_e$ is the stationary measure $\nu_{\lambda_q}$. This measure is diffuse so this gives us three parameters $r_i$ for which the points $\theta(r_i, \lambda_q)$ are disjoint. If $D$ is small enough, the associated graphs will be disjoint as well.%\note{le ref n'aime pas ce paragraphe mais je ne vois pas pourquoi... moi non plus! laissons le comme ca et envoyons lui nos excuses.. BD}

If the $r_i$ are chosen generically, there exists a set $\check\om$ of full measure such that  for $\check r\in \check\om$, $\check\theta_{\check r}$ avoids these three disjoint graphs. We conclude that the $\check\theta_{\check r}$, for $\check r\in \check\om$, form a normal family. Reversing the argument, we obtain a set $\om$ of full measure such that the associated $\theta_r$ also form a normal family.

A consequence of this is that for each $\lambda\in D$,
 $(\theta(\cdot, \lambda))_*\mathsf{P}_e = \nu_{\lambda}$. Indeed, we know that  this equality is true on a dense subset. Furthermore,  the right hand side is continuous in $\la$ by uniqueness of the stationary measure,
and so is the left hand side by the normality of the family of graphs.
Likewise,  $(\check\theta(\cdot, \lambda))_*\check{\mathsf{P}}_e = \check\nu_{\lambda}$.

%Observe that the measures $\nu_\la$ and $\check\nu_{\lambda}$ have the same support. Indeed $\supp(\nu_\la)$ is a closed $G_\la$-invariant subset, so using for instance the fact that for $\check{\mathsf{P}}$-a.e. $\check r$ and every $z_0$, $\check g_{1, \la}\cdots \check g_{n, \la}(z_0)$ converges to a point $z_{\check r}$ of law $\check\nu_\la$, we infer that $\supp( \check\nu_{\lambda})\subset  \supp (\nu_\la)$, and the converse holds by symmetry.

\medskip

Let $\Theta$ denote the family of graphs $\set{\theta_r, \ r\in\Omega}$
(and similarly, $\check\Theta$ for  $\set{\check\theta_{\check r}}$). At this point we know that there exist full measure subsets $\om\subset G^\nn$ (resp. $\check \om\subset G^\nn$) such that for a.e. $(r, \check r)\in \om\times\check\om$, $\theta_r$ and $\check \theta_{\check r}$ do not intersect in $\pi_1^{-1}(D)$. Notice that if $\mu$ is symmetric (i.e. $\mu = \check\mu$) at this point we can simply take the closure to obtain the desired holomorphic motion. The general case requires a few more arguments.

We claim that if $\theta\in \overline \Theta$, the set of $r$'s such that   $\theta_r$ is different from $\theta_0$ and  contained in a given tubular  neighborhood of $\theta_0$ has positive measure. To see this, fix a large constant $C$, larger that the volume of $\theta_0$, and restrict the attention to the set $\Theta_C\subset \Theta$ of graphs whose volume is not greater than $C$. By Step 2, $\mu^n(\Theta_C)\geq 1-\e$ when $C$ is large. The space of graphs of volume $\leq C$, equipped with the convergence on compact subsets of $D$ is a
compact metrizable space. Pushing $\mathsf{P}_e$ under $\theta$ gives rise to a measure on this space, and our claim comes down to saying that the support of this measure has no isolated points. For this, observe that more generally $\check{\theta}_*\check {\mathsf P}_e$ has no atoms, for otherwise  since $(\check\theta(\cdot, \lambda))_*\check{\mathsf{P}}_e = \check\nu_{\lambda}$, such an atom would give rise to an atom of $\nu_\lambda$, which does not happen by  Furstenberg's Theorem \ref{t: furstenberg}.

Let $(\theta,\check\theta)\in\overline \Theta\times \overline{\check \Theta}$.
If $\theta$ and $\check\theta$ admit an isolated intersection, then by  the continuity of the intersection number of 
analytic subvarieties \cite[\S 12.3, Corollary 4]{chirka}, 
the same is true for any pair  of graphs $(\theta', \check\theta')$ close to $(\theta , \check\theta)$ in the Hausdorff topology. By the previous observation, we obtain a set of positive $\mathsf{P}_e\otimes\check{\mathsf{P}}_e $ measure of intersecting pairs, which is contradictory.
 We conclude that any two such $\theta$ and $\check\theta$ are either disjoint or equal.

Now fix $\lambda_0\in D$ and $z_0\in \supp(\nu_{\lambda_0})$. %Observe that the measures $\nu_{\la_0}$ and $\check\nu_{\lambda_0}$ have the same support.
Since $\Theta$ is a  normal family, there exists $\theta\in
\overline\Theta$% (resp. $\check \theta\in\overline{\check\Theta}$)
, passing through $(\lambda_0, z_0)$.
But since the measures $\nu_{\la_0}$ and $\check\nu_{\lambda_0}$ have the same support there also exists $\check \theta\in\overline{\check\Theta}$   through $(\lambda_0, z_0)$. Thus, by the previous paragraph,  $\theta = \check\theta$. We conclude that $\supp(\nu_\lambda)$ moves holomorphically over $D$.
The invariance of this holomorphic motion follows from the equivariance of $\theta$.
% The proof of the invariance of the holomorphic motion is similar. Let $\theta$ be a graph (leaf) in the lamination $\overline\theta$, and $g\in G$. We want to show that $\widehat g \cdot \theta$ is also a leaf, using the (weak) equivariance of $\theta$ obtained in Step 2. It is enough to prove it for $g\in \supp(\mu)$. Recall that $\theta$ was defined on a full measure subset $\om\subset G^\nn$. Since $\mu(\set{g})>0$ the set $\om_g = g^{-1}\om\cap\om$ has full $\mathsf{P}$ measure so it is dense in $\om$. Hence there  exists a sequence $l_n\in \om_g^\nn$  such that $\theta_{l_n}$ converges to $\theta$. Now, $\widehat g \cdot \theta_{l_n} = \theta_{gl_n}$, and $\theta_{gl_n}\in \Theta$, thus $\widehat g \cdot \theta= \lim \widehat g \cdot \theta_{l_n}$ belongs to $\Theta$ and we are done.

 \medskip

\noindent{\bf Step 4.} Concluding stability from the motion of $\supp(\nu_\la)$.

\medskip

Let as above $D\subset \La$ be a domain disjoint from the discrete exceptional set $F$. 
Being a closed invariant set,  $\supp(\nu_\la)$ contains all fixed points of loxodromic and parabolic elements. 
For $\la_0\in D$, let $q(\la_0)$ be a fixed point, say  attracting, of a loxodromic element $\rho_{\la_0}(g)$. It admits a natural   holomorphic continuation  $q(\la)$  as a fixed point in   a neighborhood  of $\la_0$ in $D$. Let also $\gamma$ be the graph of the holomorphic motion of $\supp(\nu_\la)$ through $q(\la_0)$, constructed in step 3. With notation as before, by invariance of the holomorphic motion,  we have that $\widehat{g}(\gamma) = \gamma$. On the other hand, near $\la_0$, since $q(\la)$ stays attracting, $\widehat{g}^n(\gamma)$ converges to  $q$. Hence $\gamma \equiv q$ near $\la_0$. By analytic continuation we thus infer that  $\gamma(\la)$ is a fixed point of $\rho_{\la}(g)$ throughout $D$. 

Reversing the argument shows that for all $\la\in D$, $\rho_{\la}(g)$ stays loxodromic. Indeed the above reasoning first  implies that the two fixed points of $\rho_{\la}(g)$ remain distinct throughout $D$. Furthermore, if $p(\la_0)\in \pu$ 
is any point
 of $\supp(\nu_{\la_0})$ different from the other fixed point of $\rho_{\la_0}(g)$, and $\la\mapsto p(\la)$ denotes its continuation along the holomorphic motion, then by normality the sequence of graphs $\widehat{g}^n(p) $ converges to $q(\la)$ locally uniformly on $D$. This shows that for $\la\in D$, $\rho_{\la}(g)$ is never elliptic. 

Since the same reasoning is valid for   parabolic  transformations, we see that the Möbius transformations $\rho_{\la}(g)$ stay of constant type as $\la$ ranges along $D$, and we conclude that $D$ is contained in the stability locus. In particular, for $\la\in D$, $\rho_\la$ is discrete and faithful. 

\medskip
 
It remains to show that the exceptional set $F$ is empty. Let 
  $\lambda_0\in F$. By the J\o rgensen (Margulis-Zassenhaus) theorem~\cite[p. 170]{kapovich}, $\rho_{\la_0}$ is discrete and faithful, so $\rho_{\la}$  is discrete and faithful in the neighborhood of $\lambda_0$, and finally  $\la_0\in\mathrm{Stab}$. Thus we have shown that $U\subset \mathrm{Stab}$, thereby concluding the proof of the theorem.
\end{proof}

We now show that Theorem \ref{thm:support} remains true for generally non-elementary families, that is,  when 
a proper subset of $\La$ is made of elementary representations. This is the case for instance for the universal family of representations of $G$ into $\PSL$ (possibly after desingularization). 

\begin{thm}
 Let $(G,\mu,\rho)$ be a holomorphic family of representations, which is non-trivial, generally faithful and generally non-elementary, endowed with a probability measure $\mu$, generating $G$ as a semi-group, and satisfying 
\eqref{eq: exponential moment condition in the group}.

Then the support of $\tbif$ coincides with the bifurcation locus.
\end{thm}

\begin{proof}
Assume   that  the subset $E\subset \La$ of elementary representations is non-empty, and different from $\La$. In particular 
$E$ is contained in the bifurcation locus. We need to show that $E\subset\supp(\tbif)$. 
 
  Using the terminology introduced in the proof of Lemma \ref{lem:elementary},  $E$ decomposes as $E=E_{\rm I}\cup E_{\rm II}$ where $E_{\rm I}$ (resp. $E_{\rm II}$) is a proper
 analytic (resp. real analytic) subset of type I (resp. type II) points. 
It is easy to see that $\chi$ extends continuously by 0 on  $E_{\rm II}$. In    particular that $\chi$ cannot be harmonic near $E_{\rm II}$, so     $E_{\rm II}\subset \supp(\tbif)$. 
%We claim that $\chi$ extends continuously by 0 on  $E_{\rm II}$. \note{j'ai un doute: est ce que ça garantit bien que l'extension est psh?} Indeed, let $\la\in E_{\rm II}$ and $(\la_j)$ a sequence of parameters in $\La\setminus E$ converging to $\la_0$. Extracting if necessary, we may assume that the sequence of stationary measures $\nu_{\la_j}$ converges to some measure $\nu$. Then by Furstenberg's formula \eqref{eq:Furstenberg's formula}, 
%$$ \chi(\la_j) = \int _{\mathbb P^1} \int _ G \log \frac{\norm {\rho_{\la_j}(g) (Z)} }{\norm{Z}} d\mu (g) d \nu_{\la_j} (z) \longrightarrow \int _{\mathbb P^1} \int _ G \log \frac{\norm {\rho_{\la_0}(g) (Z)} }{\norm{Z}} d\mu (g) d \nu  (z) = 0$$ since  the $\rho_{\la_0}(g) $ are isometries. We see in  particular that $\chi$ cannot be harmonic near $\la_0$, so   $E_{\rm II}\subset \supp(\tbif)$. 

Assume now that $\la_0\in E_{\rm I}\setminus E_{\rm II}$. Fix a neighborhood $N\ni\la_0$ such that $N\cap E_{\rm II}=\emptyset$. %Then in $N$, $E$ is an analytic subset, and $\chi$, being a bounded psh function on $\La\setminus E$, admits a psh extension accross $E$\note{verifier ds klimek}, and $\tbif$ gives no mass to $E$. 
We claim that in $N$, $E\subset \overline{\rm Bif\rest{\La\setminus E}}$. This clearly implies that $E\cap N\subset \supp(\tbif)$. To prove the claim, notice that  since $E\subsetneq\La$, there exists $g,h\in G$ such that $\la\mapsto \tr^2[g_\la, h_\la]$ is not constant  so there are parameters $\la$ close to $E$ where $\tr^2[g_\la, h_\la]= 4\cos^2\theta$, with $\theta/\pi\notin\mathbb{Q}$. By assumption these parameters do not belong to  $E_{\rm I}$ so they correspond to non-elementary representations, which are  not discrete  because they contain an elliptic element of infinite order, 
  and we are done.  
\end{proof}

\subsection{Classification of stationary currents on $\La\times\pu$}
In view of the previous results, it is natural to wonder whether it is possible for a holomorphic family of representations to admit a {\em stationary current}, that is a positive closed (1,1) current $\widehat {T}$  on $\Lambda \times \pu$ such that (with notation as before)
$\int \widehat g_* \widehat T \ d\mu(g) =\widehat T$.

Let us keep hypotheses as in \S\ref{subs:support}.
We say that a current on $\La\times \pu$ is vertical if it is an integral of currents of integration over vertical fibers, that is, a current of the form $\pi_1^*T$, with $T$ a closed   positive current on $\La$. Equivalently (see Lemma \ref{lem:fibration}), $\widehat T$ is vertical if   
$\widehat T\wedge \pi_1^*\omega^k =0$ ($k=\dim(\La)$). Every vertical current is stationary.

Another possibility for the existence of a stationary current is when the family of representations is stable over $\La$.
Then it is clear that the family of stationary measures $\nu_\la$ is invariant under the holomorphic motion conjugating the representations. Fix a parameter $\la_0\in \La$, and for $z\in \pu$, let  $\Gamma_z \subset \La\times\pu$ be the graph of the holomorphic motion  passing through $(\lambda_0, z)$. Consequently, we can define a stationary current by setting $\widehat{T}  = \int [\Gamma_z] d\nu_{\lambda_0}(z)$.

The following result says that essentially all stationary currents are of this form.

\begin{thm}\label{thm:stationary current}
Let $(G,\mu,\rho)$ be an admissible  family of representations, satisfying the exponential moment assumption \eqref{eq: exponential moment condition in the group}. Assume further that the stability locus is  not empty.

Assume that there exists a   stationary current $\widehat{T}$ in $\La\times \pu$. Then either $\widehat{T}$ is vertical or  the family is stable over $\La$ and $\widehat{T}$ is the current made of the family of holomorphically varying  stationary measures as above.
\end{thm}

We believe that the additional assumption that $\mathrm{Stab}$ is non-empty is unnecessary. 

Another interpretation of this result is the following. We say that a family of measures $\set{m_\la}_{\la\in \La}$ on $\set{\lambda}\times \pu$ \textit{varies holomorphically} if the $m_\lambda$ are vertical slices of a positive closed current in $\La\times\pu$ (see the discussion on {structural varieties} in the space of 
positive measures on $\pu$ in \cite[\S A.4]{ds-survey}). What the theorem says is that the natural holomorphic family of stationary measures over the stability locus can never be holomorphically continued accross the boundary of the stability locus.

\begin{proof}
 Let us first  recall some classical facts on currents on $\La\times \pu$.
If $m$ is any probability measure with compact support in $\La$, then the mass of
$\widehat T\wedge \pi_1^*m$ is a constant independent of $m$, called the {\em slice mass}
of $\widehat T$. Indeed if $m_1$ and $m_2$ are smooth  probability measures (viewed as $2k$-forms on $\La$), then $m_1-m_2 = d\theta$, where $\theta$ is a compactly supported $(2k-1)$-form, and $\bra{\widehat T, d\theta}=0$. For Lebesgue a.e. $\lambda\in \La$, the slice measure $\widehat T\wedge [\set{\la}\times \pu]$ is well-defined, and by the above discussion, its mass does not depend on $\la$. In particular if $\widehat{T}$ is not vertical the slice mass is non-zero and we may assume that  it equals 1.

\medskip

 We assume that the family admits a non-vertical stationary current, and will show that it is stable.
Assume first that $\dim(\La)=1$. 
Since $\widehat T$ is stationary,  for a.e. $\lambda$,
$\widehat T\wedge [\set{\la}\times \pu]$ must be the unique stationary  probability measure  with respect to the action of  $\rho_\la(G)$.
For every $n\geq 1$ we have that
$\int \left(\widehat{g_1\cdots g_n}\right)_*\widehat{T}d\mu(g_1)\cdots d\mu(g_n) = \widehat{T}$. Therefore, arguing as in Step 2 of the   proof of Theorem \ref{thm:support}, there exists a set $\om\subset G^\nn$ of full measure such that if $\mathbf{g}\in\om$, there exists a subsequence $n_j$ such that the sequence of currents
$\left(\widehat{r}_{n_j}\right)_*\widehat{T} $ has bounded mass (recall that $r_{n} = r_{n}(\mathbf{g}) = g_1\cdots g_n$). Let $U$ be an open set contained in the stability locus, and  $\la_0\in U$. As before, if $z\in \pu$
let $\Gamma_z$ be the graph over $U$, passing through $(\la_0,z)$, subordinate to the holomorphic motion conjugating the representations.

Working in $\set{\la_0}\times \pu$, we know that for a.e. $\mathbf{g} \in G^\nn$,
$(g_{\lambda_0, 1})_*\cdots (g_{\lambda_0, n})_* \nu_{\lambda_0}$ converges to a Dirac mass
$\delta_{z(\mathbf{g}, \lambda_0)}$ of law $ \nu_{\lambda_0}$. Since the representations are conjugate over $U$, we conclude that there exists a set $\om'\subset G^\nn$ of full measure such that if $\mathbf{g}\in\om'$
$\lrpar{\widehat{r}_{n}}_*\widehat{T}$ converges to $[\Gamma_{z(\mathbf{g}, \lambda_0)}]=:[\Gamma_\mathbf{g}]$ in $U\times \pu$.

Putting the two previous paragraphs together (and extracting again if necessary), we see that if $\mathbf{g}\in\om\cap \om'$, there exists a subsequence $n_j$ such that $\lrpar{\widehat{r}_{n_j}}_*\widehat{T}$ converges to some $\widehat{S}$, with
$\widehat{S} = [\Gamma_{\mathbf{g}}]$ in $U\times \pu$.

We claim that $[\Gamma_{\mathbf{g}}]$ admits a continuation as a graph over $\La$. Indeed, by Siu's Decomposition Theorem  \cite{siu}, $\widehat{S} =  S_1 + S_2$, where $S_1$ is a current of integration over an at most countable family of analytic subsets, and $S_2$ gives no mass to curves. Thus, there exists an irreducible  analytic subset $V$ of $\La\times \pu$, continuing $\Gamma_{\mathbf{g}}$. Notice that $V$ is a branched cover over $\La$ relative to $\pi_1$. Since  $\widehat{S}\geq [V]$, we see that $V = \Gamma_{\mathbf{g}}$ in $U\times \pu$, hence $V$ must be graph over $\La$, and we are done.

In this way we construct a family of graphs, parameterized by a full measure subset of $G^\nn$, which is equivariant since it is equivariant over $U$. So we are exactly in the same situation as in Step 3 of the proof of Theorem \ref{thm:support}, and we conclude that the family of representations is stable over $\La$. This settles the case where $\dim(\La)=1$.

\medskip

To handle the general case we use a slicing argument (see  \cite[\S A.3]{ds-survey} for basics on slicing closed positive currents). It is no loss of generality to  assume that $\La$ is an open ball in $\cc^k$. Assume as before that the family admits a non-vertical stationary current, and suppose by contradiction that the bifurcation locus in non-empty. Then by the Margulis Zassenhaus lemma, there exists an open subset $V$ in $\La$ that is disjoint from the set of discrete and faithful representations. Now consider a linear projection $p:\La\cv\cc^{k-1}$, having the property that an open set of fibers intersects both $V$ and the stability locus, and define 
$\widehat p: \La\times \pu \cv\cc^{k-1}$ by $\widehat p = p\circ\pi$. 

For (Lebesgue) a.e. $x\in \cc^{k-1}$ the slice $  \widehat T\rest{\widehat p^{-1}(x)}$ of $\widehat T$ along the fiber $\widehat p^{-1}(x) = p^{-1}(x)\times \pu$ is a well defined  closed positive current, which is a.s. stationary since the group action preserves the fibers. The proof will be finished if we can show that  for a.e. $x$,   $\widehat T\rest{\widehat 
p^{-1}(x)}$ is not vertical. Indeed, we would then have a set of positive measure of fibers $p^{-1}(x)$ intersecting both 
$V$ and the stability  locus, and for which there exists a non-vertical stationary 
current on $p^{-1}(x)\times \pu$, thereby  contradicting the previously treated case $\dim(\La)=1$. 

To show that  $\widehat T\rest{\widehat 
p^{-1}(x)}$ is not vertical, we show that it has positive slice mass (relative to the projection $\pi: 
\widehat p^{-1}(x) \cv p^{-1}(x)$). Recall that $\widehat{T}$ is supposed to have slice mass 1. 
The so-called {\em slicing formula} asserts that if $\om$ is a  positive test  form of maximal degree on $\cc^{k-1}$ of total mass 1 (which can be identified to a probability  measure), and  $\phi$ is any test form of bidegree (1,1) on 
$\La\times \pu$, we have  
$$\int \bigg(\int_{\widehat p^{-1}(x)} \widehat T\rest{\widehat p^{-1}(x)}\wedge \phi\bigg) \om(x)
  = \int \widehat T \wedge \phi\wedge (\widehat p) ^*\om.$$

 Now if locally we view $\La$ as a product $\cc^{k-1}\times \cc$ with respective first and second projections $p$ and $q$ (thus identifying under $q$  the fibers $p^{-1}(x)$ with the second factor), and if we specialize the above formula to forms $\phi$ of the form $\pi^*q^*\varphi$ with $\varphi$ a positive test $(1,1)$ form of total mass 1 on $\cc$, we get that for a.e. $x$, and every such $\varphi$, $\int_{\widehat p^{-1}(x)}  \widehat T\rest{\widehat p^{-1}(x)}\wedge \pi^*q^*\varphi=1$, which was the desired result. 
\end{proof}

% \begin{rmk}\label{rmk:stationary current}
% The additional assumption that $\mathrm{Stab}$ is non-empty is likely to be superfluous. The point is indeed to construct an equivariant family of graphs over some open subset. If for instance there exists a measured family of graphs over some open subset such that $\widehat T\geq \int[\Gamma_\alpha] d\mu(\alpha)$, then we see that for a.e. $\alpha$ and a.e. $\mathbf{g}$, there exists a subsequence s.t. $\lrpar{\widehat{g_1\cdots g_{n_j}}}_*[\Gamma_\alpha]$  has bounded volume, and from the contraction properties of $g_1\cdots g_{n_j}$
% we obtain a limiting graph.
% 
% We see that if $\mathrm{Stab} = \emptyset$
% a stationary current must be a current with full support (since its slices have full support) such that no uniformly woven current is dominated by $T$. No such example seems to be known....
% \end{rmk}

\section{Equidistribution theorems}\label{sec:equidistribution}

In this section we prove several equidistribution results in parameter space, including Theorems \ref{theo:equidist ae} and \ref{theo:equidist speed}.   We also  give another geometric description of $T_{\rm bif}$, in the spirit of Theorem \ref{thm:geom_interpt}, where the approximating varieties are now fixed points of fibered M\"obius transformations.
%We make repeated  use of the Large Deviations Theorem for the traces of random matrix products, proven in Appendix \ref{app:large deviations}.

\subsection{A general equidistribution scheme} 
The following theorem may be interpreted as  a general method for proving equidistribution results associated to random sequences  in parameter space. 
  Specializing it to well chosen functions $F$ will lead to various equidistribution statements, including Theorem \ref{theo:equidist ae}. 

\begin{thm}\label{thm:abstrait}
Let $(G,\mu, \rho)$ be an admissible families of representations of $G$. 
Let $F$ be a psh function on $\PSL^k$. Assume that:
\begin{enumerate}[{i.}]
\item There exist non negative real numbers 
 $a_1,\ldots, a_k$, with $\sum a_i=1$, and a  constant $C$ such that for $(\gamma_i)_{i=1}^k\in \PSL^k$, 
  $$ F(\gamma_1, \ldots , \gamma_k)\leq a_1\log\norm{\gamma_1} + \cdots + a_k\log\norm{\gamma_k} + C.$$
\item If $\la\in\La$ is fixed, then for $(\mu^\nn)^{\otimes k}$-a.e. $(\mathbf{g}_1,\ldots, \mathbf{g}_k)$,  $$\unsur{n}F(\rho_\la(l_n(\mathbf{g}_1)), \ldots, \rho_\la(l_n(\mathbf{g}_k)))\underset{n\cv\infty}\longrightarrow \chi(\la).$$
\end{enumerate}
Then  for $(\mu^\nn)^{\otimes k}$-a.e. $(\mathbf{g}_1,\ldots, \mathbf{g}_k)$, 
the sequence of psh functions defined by  $$\la\longmapsto \unsur{n} F\big(\rho_\la(l_n(\mathbf{g}_1)), \ldots, \rho_\la(l_n(\mathbf{g}_k))\big)$$ converges to $ \chi(\la)$ in $L^1_{\rm loc}(\La)$.
\end{thm}
 
One might also specify different measures  $\mu_i$ on each factor. In this case the $\chi(\la)$
in {\em ii.} must be replaced by $\sum a_i\chi(\la,\mu_i)$ and the same function will appear in the conclusion. 

\medskip

 The starting point    is the following proposition. 
Recall from Theorem \ref{t: furstenberg} 
    that for a fixed representation, for $\mu^\nn$-a.e. $\mathbf g$, $\unsur{n}\log\norm{\rho(l_n(\mathbf{g}))}$ converges to $\chi(\rho)$.  We now give a parameterized version of this result.

\begin{prop}\label{prop:kingman}
  Let $(G,\mu,\rho)$ be an  admissible family of representations of $G$ into $\PSL$. Then  for $\mu^\nn$-a.e. $\mathbf{g}\in G^\nn$, the sequence of functions $\la\mapsto \unsur{n} \log\norm{\rho_\la(l_n(\mathbf{g}))}$ converges to $\chi(\la)$ in $L^1_{\rm loc}(\La)$.
\end{prop}

\begin{proof} Of course, the point  is to make a choice of generic $\mathbf g$ not depending     on $\la$.
 It is no loss of generality to assume that $\La$ is a ball in $\cc^{\dim(\La)}$. %Pick first a dense sequence $(\la_p)_p$ in $\La$ and $\om_0\subset G^\nn$ such that if $\mathbf{g}\in\om_0$, then for every $p\in\nn$ $\unsur{n} \log\norm{\rho_{\la_p}(l_n(g))}$ converges to $\chi(\la_p)$.
Let  $U\subset \La$ be any open subset, and  for $\mathbf{g}\in G^\nn$ consider the sequence $\Theta_n(U)$ defined by $\Theta_n(\mathbf{g}, U) = \int_U \log\norm{\rho_{\la }(l_n(\mathbf{g}))} d\la$.
By \eqref{eq:sub-additive cocycle}, this assignment
defines a real valued sub-additive cocycle,  which, by the moment condition \eqref{eq:moment condition finitely generated}, satisfies $\int  \Theta_1(\mathbf{g}, U) d\mu^\nn(\mathbf{g})<\infty$ .  Therefore,
 by Kingman's sub-additive ergodic theorem and the ergodicity of the shift acting on $(G^\nn, \mu^\nn)$,  we deduce that   $\unsur{n} \Theta_n(\mathbf{g}, U)$
 converges $\mu^\nn$-a.e.  to a non-negative number $\Theta(U)$ independent of $\mathbf{g}$.
% (we are  actually cheating  a little bit, since $\norm{\cdot}$ is not an operator norm; nevertheless changing the norm does not affect the asymptotic behaviour of $\unsur{n}\Theta_n$). 

Also,
$\unsur{n}\int \Theta_n(\mathbf{g}, U) d\mu^\nn(\mathbf{g})$    converges to $\Theta(U)$. Indeed 
$0\leq \Theta_n(\mathbf{g}, U)\leq C \length(l_n\mathbf{g})$ for some constant $C$, whereas by Kingman's theorem and \eqref{eq:moment condition finitely generated} the sequence $\length(l_n(\mathbf{g}))$ converges in $L^1(\mu^\nn)$ (see the domination argument after Theorem \ref{t: furstenberg}). Therefore, the convergence of $\unsur{n} \Theta_n(\mathbf{g}, U)$ to $\Theta(U)$ takes place in $L^1(\mu^\nn)$.

Take now a countable neighborhood basis $(U_q)_q$ of $\La$. There exists a full measure subset $\om\subset G^\nn$ such that if $\mathbf{g}\in\om$, then for every $q$,
$\unsur{n} \int_{U_q} \log\norm{\rho_{\la }(l_n(\mathbf{g}))} d\la$ converges to some $\Theta (U_q)$.

By   \eqref{eq:moment condition finitely generated} again, for $\mu^\nn$-a.e. $\mathbf{g}$
the length of the word $l_n(\mathbf{g})$ in $G$ grows at linear speed. Hence
the sequence of psh functions on $\La$ defined by  $\left(\lambda\mapsto\unsur{n}\log\norm{\rho_{\la}(l_n(\mathbf{g}))}\right)_n$ is   locally uniformly bounded, so it admits convergent subsequences in $L^1_{loc}$. If $\theta(\mathbf{g}, \la)$ denotes  any of its cluster values, we see that
$\int_{U_q}\theta(\mathbf{g}, \la)d\la$ must be equal to $\Theta(U_q)$. Hence the sequence actually converges  to a limit independent of $\mathbf g$, which we denote by   $\theta(  \la)$.

\medskip

The last step is of course to prove that $\theta(\cdot)=\chi(\cdot)$. For this, it is enough to integrate with respect to $\mathbf{g}$. Indeed, for any $\la\in \La$,
$$ \frac{1}{n}\int _G \log \norm{\rho_\la(g)} d\mu^{n}(g)
=\frac{1}{n}\int\log\norm{\rho_{\la}(l_n(\mathbf{g}))} d\mu^\nn(\mathbf{g}) \underset{n\cv\infty} \longrightarrow \chi(\la).$$ Since the left hand side is locally uniformly bounded in $n$, by dominated convergence we  infer that for any open set $U$,
$$\int_U\lrpar{\frac{1}{n}\int\log\norm{\rho_{\la}(l_n(\mathbf{g}))} d\mu^\nn(\mathbf{g})} d\la \underset{n\cv\infty} \longrightarrow  \int_U\chi.$$
Now we let $U= U_q$ and switch the integrals  to see that
$$\int_U\lrpar{\frac{1}{n}\int\log\norm{\rho_{\la}(l_n(\mathbf{g}))} d\mu^\nn(\mathbf{g})} d\la
=\int \unsur{n} \Theta_n(\mathbf{g}, U_q) d\mu^\nn(\mathbf{g}),$$ which converges to $\Theta(U_q)  = \int_{U_q}\theta$. We conclude that for any $q$, $\int_{U_q}\theta =\int_{U_q}\chi$, and the result follows.
\end{proof}

We also need the following variation on the Hartogs Lemma (see  \cite[pp. 149-151]{hormander}). 

\begin{lem}\label{lem:hartogs}
Let $\om\subset\cc^n$ be a connected  open set and $(v_n)$ be a sequence of psh functions in $\om$ converging in $L^1_{\rm loc}$ to a continuous psh function $v$. Assume now that $(u_n)$ is another sequence of psh functions in $\om$ such that:
\begin{itemize}
\itm for every $x\in \om$, $u_n(x)\leq v_n(x)$;
\itm there exists a dense subset $D\subset \om$ such that for every $x\in D$, $u_n(x)\cv v(x)$ as $n\cv\infty$.
\end{itemize}
Then $(u_n)$ converges to $v$ in $L^1_{\rm loc}$.
\end{lem}

\begin{proof}
 Observe first that $(v_n)$ is locally uniformly bounded above, hence  so is $(u_n)$. Since $u_n(x)\cv v(x)$ on $D$,  $(u_n)$ cannot  diverge to $-\infty$ so there exists a  subsequence $(u_{n_j})$ converging  in $L^1_{\rm loc}$ to a psh function $u$. The point is to prove that $u=v$.

For a.e. $x$, $\limsup u_{n_j} = u(x)$, from which we infer that $u\leq v$ a.e. Now suppose that there exists $x_0$ such that $u(x_0)< v(x_0)$. By upper semi-continuity (we use the fact that $v$ is continuous) there exists a relatively compact open set $B\ni x_0$  where $u< v-\delta$  for some positive $\delta$.   By the Hartogs lemma for large $j$ we get that $u_{n_j}< v-\delta$ on $B$. This contradicts the fact that
$u_n(x)\cv v(x)$ on a dense subset.
\end{proof}

\begin{proof}[Proof of  Theorem \ref{thm:abstrait}]
Pick   a dense sequence $(\la_p) $ in $\La$. There exists a set   $\om_0\subset (G^\nn)^k$ of full measure such that if $(\mathbf{g}_1, \ldots, \mathbf{g}_k) \in\om_0$, then for every $p$ 
$$\unsur{n}  F\big(\rho_{\la_p}(l_n(\mathbf{g}_1)), \ldots, \rho_{\la_p} (l_n(\mathbf{g}_k)  )\big)\underset{n\cv\infty} \longrightarrow \chi(\la_p).$$
Applying  Proposition \ref{prop:kingman}, let $\om_1^k\subset G^\nn$ be a set of full measure such that for any 
  $(\mathbf{g}_1, \ldots, \mathbf{g}_k) \in\om_1$, 
$$\frac{a_1}{n} \log\norm{\rho_\la(l_n(\mathbf{g}_1))} + \cdots +
\frac{a_k}{n} \log\norm{\rho_\la(l_n(\mathbf{g}_k))} \longrightarrow \chi(\la) \text{ in } L^1_{\rm loc}.$$ 
 
Now for every $\la\in\La$ we have that   $$ \unsur{n} F\big(\rho_{\la }(l_n(\mathbf{g}_1)), \ldots, \rho_{\la } (l_n(\mathbf{g}_k)  )\big)\leq  
\frac{a_1}{n} \log\norm{\rho_\la(l_n(\mathbf{g}_1))} + \cdots +
\frac{a_k}{n} \log\norm{\rho_\la(l_n(\mathbf{g}_k))}+ O\lrpar{\unsur{n}}.$$
 From Lemma \ref{lem:hartogs},  we thus conclude that if 
$(\mathbf{g}_1, \ldots, \mathbf{g}_k) \in\om_0\cap \om_1^k$, 
$$ \unsur{n} F\big(\rho_{\la }(l_n(\mathbf{g}_1)), \ldots, \rho_{\la } (l_n(\mathbf{g}_k)  )\big)\longrightarrow \chi(\la) \text{ in } L^1_{\rm loc}, $$ which finishes the proof.
\end{proof}

As a sample application of Theorem \ref{thm:abstrait},  let us prove the following variant of  Theorem \ref{thm:geom_interpt}.  

\begin{thm}\label{thm:geom_interpt2}
Let $(G,\mu,\rho)$ be an  admissible family of representations of $G$ into $\PSL$. Fix $z_0\in \pu$. Then for $\mu^\nn$-a.e. $\mathbf{g}\in G^\nn$, the sequence of currents $\unsur{n}[\widehat{l_n(\mathbf g)}\cdot z_0]$ in $\La\times\pu$
converges to $\pi_1^*(T_{\rm bif})$.
\end{thm}

By following step by step the proof of Theorem \ref{thm:geom_interpt} (and keeping notation as in that proof),
 we first see that every cluster value $\widehat S$ of the sequence $\unsur{n}[\widehat{l_n(\mathbf g)}\cdot z_0]$ must satisfy $\langle \widehat S, \pi_1^*\omega^k\rangle = 0$. To show that this sequence of currents is a.s. of bounded mass and converges to $\pi_1^*(T_{\rm bif})$, it is enough to show that for every $(k-1,k-1)$ test form $\phi$ on $\La$ and a.e. $\mathbf{g}$
$$\bra{ (\pi_1)_*\left(\unsur{n}[\widehat{l_n(\mathbf g)}\cdot z_0]\wedge \pi^*_2\omega_\pu\right), \phi} 
= \int_\La \phi\wedge dd^c_\la \left(\unsur{n} \log\frac{\norm{\rho_\la(l_n(\mathbf{g}))Z_0}}{\norm{Z_0}}\right)
 \underset{n\cv\infty}\longrightarrow \bra{\tbif, \phi},$$  where the equality on the left hand side is obtained as in \eqref{eq:pi etoile}. Thus we conclude
 that to obtain Theorem \ref{thm:geom_interpt2} it is enough to establish the following:
 
\begin{prop}
Let $(G,\mu,\rho)$ be an  admissible family of representations of $G$ into $\PSL$. Let $z_0\in \pu$ and let $Z_0\in\cd$ be any lift of $z_0$.
Then for $\mu^\nn$ a.e. $\mathbf{g}$ the sequence of functions $\la\mapsto \unsur{n} \log\norm{\rho_\la(l_n(\mathbf{g}))Z_0}$ converges to $\chi(\la)$ in $L^1_{\rm loc}(\La)$.
\end{prop}

\begin{proof} It is enough to check that the assumptions of Theorem \ref{thm:abstrait} hold for the psh function  $F:\gamma  \mapsto \log\norm{\gamma Z_0}$ on $\PSL$. The inequality in {\em i.} is obvious, while {\em ii.} follows from Theorem \ref{t: furstenberg}. The result follows.
\end{proof}
%
%\begin{proof}[Proof of the proposition]
%Pick   a dense sequence $(\la_p)_p$ in $\La$. There exists a set   $\om_0\subset G^\nn$ of full measure such that if $\mathbf{g}\in\om_0$, then for every $p\in\nn$ $\unsur{n} \log\norm{\rho_{\la_p}(l_n(\mathbf g))Z_0}$ converges to $\chi(\la_p)$.
%Let $\om_1$ be the set of $\mathbf{g}$ satisfying  Proposition \ref{prop:kingman}, and let $\mathbf{g}\in \om_0\cap \om_1$, so that $\unsur{n} \log\norm{\rho_\la(l_n(\mathbf{g}))}$ converges in $L^1_{\rm loc}(\La)$ to $\chi(\la)$.
%Now for every $\la\in\La$ we have that   $$0\leq \unsur{n} \log\norm{\rho_\la(l_n(\mathbf{g}))Z_0} \leq  \unsur{n} \log\norm{\rho_\la(l_n(\mathbf{g}))} + O\lrpar{\unsur{n}},$$
%%  so we may extract a converging subsequence
%% $n_j$ of
%% $\unsur{n} \log\norm{\rho_\la(l_n(\mathbf{g}))Z_0}$ in $L^1_{\rm loc}$. Let $\theta$ be its limit: by the previous displayed equation, $\theta\leq \chi$ (integrate on small balls)\note{fait partie du lemme de Hartogs....}.
%% 
%% To conclude the proof it is enough to show that $\theta = \chi$. Assume by contradiction that there exists $\lambda_0$ such that $\theta(\la_0)<\chi(\lambda_0)$. Then, since $\chi$ is continuous, by the Hartogs lemma, there exists an open set $U$ and a $\delta>0$ such that  if $\la\in U$, then  for large $j$, $\unsur{n_j} \log\norm{\rho_{\la }(l_{n_j}(\mathbf g))Z_0}\leq \chi(\la)-\delta$. This is a contradiction since   $\mathbf{g}\in \om_0$, and we are done.
%so the result follows from the previous lemma.
%\end{proof}

The next result shows that under an additional assumption we can integrate with respect to $\mathbf{g}$ in Theorem \ref{thm:abstrait}. It is slightly more convenient to state it in terms of currents rather than potentials. We use the notation  $\m_\om(T)$ for the mass of the current $T$ in $\om$. 

\begin{prop}\label{prop:integree}
 Let $(G,\mu,\rho)$ be an admissible family of representations, and $F$ a function on $\PSL^k$ satisfying the assumptions of Theorem \ref{thm:abstrait}. For $(g_1, \ldots, g_k)\in G^k$ let $T(g_1, \ldots, g_k)$ be the current on $\La$ defined by 
$T(g_1, \ldots, g_k) = dd^c_\la\lrpar{F(\rho_\la(g_1), \ldots , \rho_\la(g_k))}$. 

Assume that for every $(g_1, \ldots, g_k)\in G^k$ and every $\La'\Subset\La$, there exists a constant $C(\La')$ such that 
$\m_{\La'}(T(g_1, \ldots, g_k) )\leq C\sum_{i=1}^k \length(g_i)$. 

Then $\unsur{n}   T(l_n(\mathbf{g}_1), \ldots , l_n(\mathbf{g}_k))$ converges to $\tbif$ in $L^1(\mu^\nn\otimes\cdots \otimes\mu^\nn)$. In particular
\begin{equation}\label{eq:integree}
\unsur{n}\int T(g_1, \ldots, g_k)d\mu^n(g_1)\cdots d\mu^n(g_k)\underset{n\cv\infty} \longrightarrow \tbif.
\end{equation}
\end{prop}

\begin{proof} Note first that \eqref{eq:integree} means that for any  $(k-1, k-1)$ test form $\varphi$ on $\La$ ($k= \dim(\La)$) , 
\begin{equation}\label{eq:integreephi}
\unsur{n}\int \bra{T(g_1, \ldots, g_k), \varphi} d\mu^n(g_1)\cdots d\mu^n(g_k)\underset{n\cv\infty} \longrightarrow \bra{\tbif, \varphi}.
\end{equation}
The mass estimate in the statement of the proposition implies that 
$\abs{\langle T(g_1, \ldots, g_k), \varphi\rangle} \leq  C(\varphi) \sum_{i=1}^k \length(g_i)$.
Since an admissible family of representations satisfies 
\eqref{eq:moment condition finitely generated}, this  guarantees the existence of the integrals in \eqref{eq:integreephi}.
Next, by Theorem \ref{thm:abstrait}, for a.e. $(\mathbf{g}_1, \ldots , \mathbf{g}_k)$, 
$\unsur{n} \bra{ T(l_n(\mathbf{g}_1), \ldots , l_n(\mathbf{g}_k)), \varphi}$ converges to  $\bra{\tbif, \varphi}$. To get the desired result we need to show that this convergence takes place in $L^1(\mu^\nn\otimes\cdots \otimes\mu^\nn)$. Now the domination argument after Theorem \ref{t: furstenberg} implies  that $\unsur{n} \length(l_n(\mathbf{g}))$ converges in $L^1(\mu^\nn)$ to a constant, so the result simply follows from the Dominated Convergence Theorem.
\end{proof}

\subsection{Equidistribution of parameters with a given trace}
Let $(G,\mu,\rho)$ be an admissible family of representations, and fix $t\in \cc$.
Let $\mathcal{P}_t\subset G$ be the set of elements $g$ such that
the function $\la\mapsto\tr^2(g_\la)$ is constant and equal to $t$ (typically, a persistently parabolic element). By %Theorem \ref{thm:guivarch},
 Corollary~\ref{cor:large deviations}, $\mu^n(\mathcal{P}_t)$ converges to zero. %--actually exponentially fast by Theorem \ref{thm:large deviations}.

Our purpose here is to study the distribution of parameters $\la$ such that there exists
$g\in G\setminus \mathcal{P}_t$ with $\tr^2(g_\la)= t$.
This is mostly interesting when $t = 4\cos^2\left(2\pi\frac{p}{q}\right)$, since the representations for these parameters exhibit ``accidental" new relations (also accidental parabolics when $t=4$). Recall from Corollary \ref{cor:accidental} that such parameters are dense in the bifurcation locus.

For $g\in G\setminus \mathcal{P}_t$ we let $Z(g,t)$ be the codimension 1 subvariety of parameter space defined as
  $Z(g,t) =\set{\lambda, \ \tr^2(g_\lambda) - t=0}$ (with the corresponding multiplicity, if any). Recall that with our conventions, if $g\in\mathcal{P}_t$,   $[Z(g,t)] = 0$.

The next  result  belongs to the general scheme presented in the previous paragraph. 

\begin{thm}\label{thm:equidist ae}
Let $(G,\mu,\rho)$ be an admissible family of representations satisfying the exponential moment condition \eqref{eq: exponential moment condition in the group}, and fix $t\in \cc$.

Then for $\mu^{\nn}$-a.e. $\mathbf{g}\in G^\nn$, the sequence of integration currents
$ \unsur{2n}\left[Z(l_n(\mathbf{g}),t)\right]$ converges to $T_{\rm bif}$.
\end{thm}

%
%\begin{proof} We work with potentials so for $g\in G\setminus \mathcal{P}_t$, let
% $u (\lambda, g,  t )  =  \log\abs{\tr^2(g_\lambda) - t}$ be a psh potential of $Z(g,t)$.
%Since $t$ is fixed from now on we drop it.
%Our purpose is to show that for $\mu^{\nn}$-a.e. $\mathbf{g}$, $l_n (\mathbf{g})\notin \mathcal{P}_t$ for large enough $n$, and
%$\unsur{2n} u (\cdot, l_n(\mathbf g) )$ converges to $\chi$ in $L^1_{\rm loc}(\La)$. The argument  is actually completely analogous to that of Theorem \ref{thm:geom_interpt2}.

%Indeed, pick a dense sequence $(\la_p)$ and a set $\om\subset G^\nn$ of full measure such that:
%\begin{enumerate}[{(i)}]
%\item  if $\mathbf g\in \om$, then for every $p$, $\unsur{2n}\log\abs{\tr^2(\rho_{\la_p}(l_n(\mathbf g)) - t}$
%converges to $\chi(\la_p)$ as $n\cv\infty$ (this is possible due to Theorem \ref{thm:guivarch});
%\item  if $\mathbf g\in \om$, then $\unsur{n} \log\norm{\rho_\la(l_n(\mathbf{g}))}$ converges in $L^1_{\rm loc}(\La)$ to $\chi(\la)$.
%\end{enumerate}
%Now for every $\la\in \La$,  we have that
% $$\unsur{2n}  u (\lambda,l_n(\mathbf{g})) \leq\unsur{n} \log\norm{\rho_\la(l_n(\mathbf{g}))} + O\lrpar{\unsur{n}}, $$ so %if $ \mathbf g\in \om$, then either $\unsur{2n}  u (\lambda,l_n(\mathbf{g})$ diverges uniformly to $-\infty$ or there is a subsequence which converges in $L^1_{\rm loc}$ to a limit $\theta\leq \chi$ (possibly $-\infty$). As before, we use    item (i.) above to both exclude divergence to $-\infty$ and force  $\theta = \chi$ and we are done.
% as before the result follows from Lemma \ref{lem:hartogs}.
%\end{proof}
\begin{proof} 
We work with potentials so for $g\in G\setminus \mathcal{P}_t$, let
$u (\lambda, g,  t )  =  \log\abs{\tr^2(g_\lambda) - t}$ be a psh potential of $Z(g,t)$. Notice that 
if $g\in \mathcal{P}_t$, $u (\lambda, g,  t )  \equiv -\infty$, 
 nevertheless this won't affect the argument. 
Since $dd^c\left(\unsur{2n} u(\cdot, l_n(\mathbf{g}), t) \right)= \unsur{2n} [Z(l_n(\mathbf{g}),t)]$, 
to get the desired convergence it suffices
to show that for $\mu^{\nn}$-a.e. $\mathbf{g}$,  
$\unsur{2n} u (\cdot, l_n(\mathbf g),t )$ converges to $\chi$ in $L^1_{\rm loc}(\La)$ (in particular $l_n (\mathbf{g})\notin \mathcal{P}_t$ for large   $n$). 

For this,  it suffices to apply Theorem \ref{thm:abstrait} to  the psh function 
defined on $\PSL$ by $F(\gamma) = \unsur{2} \log\abs{\tr^2(\gamma)-t}$. Assumption {\em i.} in that theorem clearly holds, and Theorem \ref{thm:guivarch} gives  {\em ii.} 
\end{proof}

It is natural to wonder whether the convergence  in the previous theorem can be made more precise. 
We already observed --see the discussion after Corollary \ref{cor:guivarch}-- that it is not true in general that for a given $\la$, $\unsur{2n}\log\abs{\tr^2(\rho_{\la }(l_n(\mathbf g)) - t}$ converges to $\chi(\la)$ in $L^1(\mu^\nn)$. %As a warm-up for the proof of Theorem \ref{theo:equidist speed}, let us
Here we  show  that under some global assumptions on $\La$ we can  indeed  integrate with respect to $\mathbf{g}$ in  Theorem \ref{thm:equidist ae}.

\begin{thm}\label{thm:equidist integre}
Let $(G,\mu,\rho)$ be an admissible family of representations satisfying the exponential moment condition \eqref{eq: exponential moment condition in the group}, and fix $t\in \cc$. Suppose in addition that one of the following two conditions is satisfied:
\begin{enumerate}
\item[i.] the family of representations $(\rho_\la)_{\la\in \La}$ is algebraic;
\item[ii.] or there exists at least one geometrically finite representation  in $\La$.
\end{enumerate}
Then
\begin{equation}\label{eq:equidist integree}
\unsur{2n} \int_{G^\nn}\left[Z(l_n(\mathbf{g}),t) \right]  d\mu^\nn(\mathbf g)=
\unsur{2n} \int_{G} \left[Z(g,t)\right] d\mu^n(g)  \underset{n\cv\infty}{\longrightarrow} T_{\rm bif}
\end{equation}(recall that  if $g\in \mathcal{P}_t$, by definition $[Z(g,t)]=0$).
\end{thm}

 The notion of an algebraic family of representations was  introduced in \S\ref{subs:families}. Observe in particular that condition {\em i.} is satisfied when $\La$ is an open subset of the family of {all} representations of $G$ into $\PSL$ (resp. modulo conjugacy). 
 It will also be clear from the proof that {\em ii.} can be relaxed to only requiring that $\La$ can be continued to a family containing a geometrically finite representation.  

\begin{proof} In view of Proposition \ref{prop:integree} it is enough to show that for 
every $g\in G$, and every $\La'\Subset \La$, $\m_{\La'}\lrpar{[Z(g,t)]}\leq C(\La') \length(g)$. 
Observe that if $g\in  \mathcal{P}_t$ this is true by definition.

 This is easiest under assumption {\em i.}, so let us assume that $(\rho_\la)$  is an  algebraic family.
Recall that $\PSL$ is isomorphic to $\mathrm{SO}(3,\cc)$, so that, reducing $\La'$ if necessary,  we view
$\La'$ as an open subset of an affine subvariety in $\cc^{9k}$ ($k$ is the number of generators). For $1\leq i\leq k$, let $(a_{i, j})_{1\leq i\leq 9}$ be the coefficients corresponding to the generator $g_i$.
If $g\in G$ is any element, it  is easy to see that $\tr^2(g)-t$ is a polynomial in the $a_{i,j}$ of degree $O(\length(g))$, so the desired estimate simply follows from Bézout's Theorem and the fact that the volume of an algebraic subvariety is controlled by its degree.

\medskip

Let us now suppose that {\em ii.} holds. If $g\in G\setminus \mathcal{P}_t$, let
 $u (\lambda,g )  =  \log\abs{\tr^2(g_\lambda) - t}, $ be a psh potential of $[Z(g,t)]$. If $\La''$ is an open set with  $\La'\Subset\La''\Subset\La$, there exists a constant $C( \La',\La'')$ such that
 $\m_{\Lambda'}\lrpar{[Z(g,t)]} \leq C \norm{u(\cdot, g)}_{L^1(\La'')}$ (see \cite[Remark 3.4]{demailly}), so our task is  to control this $L^1$ norm. Notice further that it is enough to consider the case where $\la\mapsto\tr^2(g_\lambda) $ is not constant (for otherwise $[Z(g,t)] = 0$). The following lemma then completes the proof of the theorem.
\end{proof}

\begin{lem}\label{lem:traceminor ii}
If there exists a geometrically finite representation in $\La$, then for every
relatively compact open subset $\La'\subset \La$ there exists a constant $C$ such
that for every $g\in G$, if $\la\mapsto\tr^2(g_{\lambda})$ is not a constant function, then
$\norm{u(\cdot, g)}_{L^1(\La')}\leq C\length(g)$.
\end{lem}

As the proof will show, it is easy to obtain such an estimate for a family consisting entirely of geometrically finite representations. To handle the general case, we use some classical properties of psh functions, which we remind first.

\begin{lem}\label{lem:three circles}
Let $u$ be a psh function on a connected complex manifold $\om$, and $M>0$ with $\sup_\om u \leq  M$. Fix two  relatively compact  open subsets $\om'$ and  $\om''$ of $\om$,  and let $x_0\in \om''$.  Then there exists a constant $A(\om', \om'')$
such that
 the following properties hold:
 \begin{enumerate}[i.]
 \item $\displaystyle \norm{u}_{L^1(\om')} \leq A \max(\abs{u(x_0)},M)$;
 \item $\displaystyle \sup_{\om'}u \geq -A \max(\abs{u(x_0)},M)$.%\note{est ce que ii. nous sert vraiment par la suite? oui par exemple pour invoquer le resultat de hormander}
% \item $\displaystyle\sup_{x\in \om'} \nu(u,x)\leq A \max(\abs{u(x_0)},M)$, where 
% $\nu(u,x)$ denotes the Lelong number of $u$ at $x$.
 \end{enumerate}
 \end{lem}

 \begin{proof}
This follows from a standard compactness argument: consider the family of psh functions of the form $ v = \frac{u}{\max(\abs{u(x_0)},M)}$.  This family is compact in $L^1_{\rm loc}(\om)$,
 since $\sup_\om v \leq  1$ and $v(x_0)\geq -1$, whence {\em  i.} and  {\em ii.} follows.
%
% To get {\em iii.} it is no loss of generality to assume that   $\om'$ is a ball, say the unit ball $B(0,1)\subset \cc^{\dim\om}$, and that $B(0,2)\Subset \om$. We have that
% $\sup_{B(0,2)}v \leq 1$ and we  know from {\em ii.} that there exists  a constant $A$ such that
%$\sup_{B(0,2)\setminus B(0, 3/2)} v \geq -A$. We conclude that
% there exists a universal constant $B$ such that for every $x
%\in B(0,1)$($ = \om'$), the Lelong number
%$\nu( v, x)$ is bounded by $B$. Indeed recall  that
%$\max_{B(x,r)} v$ is a convex function of $\log r$, and that
%$\nu(v,x ) = \lim_{r\cv 0} \frac{\max_{B(x,r)} v}{\log r}$. The result  easily
%follows (see e.g. the proof of  \cite[Cor. 3.3]{zeriahi}).
\end{proof}

\begin{proof}[Proof of Lemma \ref{lem:traceminor ii}]
First, since $\tr^2(g_\la)\leq C\norm{g_\la}^2$,  there exists a constant $M(\La' )>0$ such that for every $g\in G$,
\begin{equation}\label{eq:suptrace}
\sup_{\la\in\Lambda'  } \abs{\tr^2(g_\la) -t} \leq M^{\length(g)}.
\end{equation}

 Let now $\lambda_0$ be a parameter such that $\rho_{\lambda_0}$ is geometrically finite. Recall that this means that $\rho_{\lambda_0}$ is discrete, faithful, and that there is a finite sided fundamental domain for the $\rho_{\lambda_0}$-action of $G$ on hyperbolic $3$-space. In this case it is known (see \cite[Theorem 12.7.8]{ratcliffe}) that given any constant $\ell>0$ the quotient hyperbolic manifold admits only finitely many closed  geodesics of length bounded by $\ell$.
Let now $\La''$ be a small ball containing $\la_0$.
 Observe that if $\gamma $ is the closed geodesic in $M$ corresponding to some element $g_{\lambda_0}$, then the length of $\gamma$ is given by $2 \log |\lambda_{\max}(g_{\lambda_0})|$, where as before $\lambda_{\rm max}$ denotes an eigenvalue of $g_{\lambda}$ of maximal modulus. %The preceding discussion shows 
Therefore there is only a finite number of conjugacy classes of elements $g\in G$ such that $| \lambda_{\rm max} (g_{\lambda_0}) |\leq 4 + |t|$ (observe that in a geometrically finite representation there is only a finite number of conjugacy classes of parabolic elements).
Hence there exists a positive number $C$ such that for every element in this finite number of conjugacy classes, either $\lambda\mapsto \tr ^2 g_{\lambda}$ is   constant, or
there is a parameter $\lambda_1\in \Lambda''$ such that $\abs{\tr ^2 g_{\lambda_1}- t}\geq C$.
On the other hand,  for the elements $g\in G$ satisfying $|\lambda_{\rm max} (g_{\lambda_0} )| >|t|+4$, we have that $|\tr^2 g_{\lambda_0} -t| > 1$.
From this discussion and \eqref{eq:suptrace}, we obtain the desired bound on the $L^1$ norm
 by passing to logarithms and applying  Lemma \ref{lem:three circles} with $x_0 =\la_0$, $\om = \La$, $\om' = \La'$ and $\om''=\La''$.
\end{proof}

\subsection{Collisions between fixed points}
Here we examine the distribution of another  natural codimension 1 phenomenon in the bifurcation locus.
For a pair of elements  $(g,h)$ in $ G$,  consider the subvariety in $\La$ defined by
$$F(g,h) =\set{\la, \ \fix(g_\la)\cap \fix(h_\la)\neq\emptyset} .% = \set{\la,\ \tr[g_\la, h_\la]=2},
$$ %(the second equality is Lemma \ref{lem:commutator}) endowed  the multiplicity induced by its equation $\tr[g_\la, g'_\la]=2$. 
As before, if $F(g,g')=\La$ we declare that $[F(g,g')]=0$.

The associated equidistribution statement is the following.

\begin{thm}\label{thm:collision}
Let $(G,\mu,\rho)$ be  an admissible family of representations satisfying the exponential moment condition \eqref{eq: exponential moment condition in the group}.

Then for $(\mu^{\nn}\otimes  \mu ^\nn)$-a.e. $(\mathbf{g},\mathbf{h})\in (G^\nn)^2$,  we have 
$$ \unsur{4n}\left[  F(l_n(\mathbf{g}), l_n(\mathbf{h})) \right] \underset{n\cv\infty}\longrightarrow
T_{\rm bif}.$$
If furthermore one of the conditions {i.},  ii. of Theorem \ref{thm:equidist integre} holds, then the convergence takes place in $L^1(\mu^\nn\otimes \mu^\nn)$.
\end{thm}

 It is certainly possible to give    estimates  for the speed of convergence in the spirit of Theorem \ref{thm:equidist speed}, but we  omit this. Also, we may choose $\mathbf{h}$ to be generic with respect to some other measure  $\mu'$  on $G$, in which case,  $ \unsur{2n}\left[  F(l_n(\mathbf{g}), l_n(\mathbf{h})) \right] \underset{n\cv\infty}\longrightarrow
T_{\rm bif} + T_{\rm bif}',$ where $\tbif'$ is the bifurcation current associated to $(G,\mu',\rho)$

%\begin{proof}
%We consider the  potentials
%$\la\mapsto v(\la, g,g') =  \log\abs{\tr[g_\la,g_\la']-2}$ and need to show that for $(\mu^{\nn}\otimes (\mu')^\nn)$-a.e. $(\mathbf{g},\mathbf{g'})$,
%$ \unsur{2n} v(\cdot,l_n(\mathbf{g}), l_n(\mathbf{g'})) \cv \chi+\chi'$ in $L^1_{\rm loc}(\La)$.
%Again, for this we use Lemma \ref{lem:hartogs}. So we first observe that for every $(\la,\mathbf g,\mathbf{g'})$,
%$$ \unsur{2n} v(\la,l_n(\mathbf{g}), l_n(\mathbf{g'}))\leq\unsur{n} \log\norm{\rho_\la(l_n(\mathbf{g}))} +\unsur{n} \log\norm{\rho_\la(l_n(\mathbf{g'}))} + O\lrpar{\unsur{n}},$$  while the existence of a dense sequence of parameters $(\la_q)$  at which for a.e. $(\mathbf g,\mathbf{g'})$, $ \unsur{2n} v(\la_q,l_n(\mathbf{g}), l_n(\mathbf{g'}))$ converges to $\chi+\chi'$  readily follows from Corollary \ref{cor:schottky}.

%The argument needed to establish the  convergence in $L^1(\mu^\nn\otimes (\mu')^\nn)$ under  the additional assumption   {(i.)} or  {(ii.)} parallels that of Theorem \ref{thm:equidist integre} and will be left to the reader. 
% \end{proof}

\begin{proof}
Recall from Lemma \ref{lem:commutator} that $F(g,h) = \set{\la,\ \tr[g_\la, h_\la]=2}$. Notice that this allows us to properly define the multiplicity of $F(g,h)$.  Passing to potentials, what we need to prove is that for $(\mu^{\nn}\otimes  \mu ^\nn)$-a.e. $(\mathbf{g},\mathbf{h})$, 
$$\unsur{4n} \log  \abs{\tr[l_n(\mathbf g), l_n(\mathbf h)]-2} \underset{n\cv\infty}\longrightarrow \chi(\la) \text{ in } L^1_{\rm loc}.$$ Again for this we use Theorem \ref{thm:abstrait}, for $F(\gamma_1,\gamma_2) = \log\abs{\tr[\gamma_1, \gamma_2]-2}$. The plurisubharmonicity and {\em i.} are obvious, while {\em ii.} follows from Corollary \ref{cor:schottky}.

The proof of  Theorem \ref{thm:equidist integre} shows that under one of the additional assumption  {\em i.} or {\em ii.} of that theorem, for every $\La'\Subset \La$, $\m_{\La'}([F(g,h)])\leq C (\length(g)+\length(h))$. Therefore the second assertion of   Theorem \ref{thm:collision} follows from Proposition \ref{prop:integree}. 
\end{proof}

\subsection{Speed of convergence}
We now prove Theorem \ref{theo:equidist speed}. 

\begin{thm}\label{thm:equidist speed}
Let $(G,\mu,\rho)$ be an admissible family of representations satisfying the exponential moment condition \eqref{eq: exponential moment condition in the group}, and fix $t\in \cc$. 

 Suppose in addition that one of the following conditions holds:
   \begin{enumerate}
 \item[i.] either $\La$ is an algebraic family of representations, defined over $\overline{\mathbb{Q}}$;
 \item[ii.] or there is at least one geometrically finite representation in $\La$.
 \end{enumerate}
  
Then there exists a constant  $C$ such that  for every test form $\phi$
\begin{equation}\label{eq:equidist speed}
\bra{\unsur{2n} \int  \left[Z(g,t)\right] d\mu^n(g) -T_{\rm bif}, \phi}\leq
 C \frac{\log n }{n} \norm{\phi}_{C^2}.
\end{equation}
\end{thm}

A few words about the proof: the machinery of Theorem \ref{thm:abstrait}, based on a compactness argument, does not allow for such an estimate, so the idea is to reprove Theorem \ref{thm:equidist integre} from scratch by using the quantitative results of Appendix \ref{app:distance}. The necessity to integrate with respect to ${\mathbf g}$ is due to the fact that    the estimate on $\delta(\rho(l_n(\mathbf g)))$ given in Theorem \ref{thm:distance} is too sensitive to bifurcations to  be made uniform in $\la$. As already said, a basic source of difficulty is that in general one cannot expect that for a given $\la$, $\unsur{2n}\log\abs{\tr^2(\rho_{\la }(l_n(\mathbf g)) - t}$ converges to $\chi(\la)$ in $L^1(\mu^\nn)$. To control the size of the set of   ``exceptional'' parameters where this convergence does not hold, we use volume estimates for sublevel sets of psh functions. As usual the notation $C$ stands for a ``constant" which may change from line to line, but does not depend on $n$.

\begin{proof}
 Fix  a finite set $\set{g_1,\ldots,g_k}$ of generators and let
 $B_G(\mathrm{id}, R)\subset G$ be the set of elements of length at most $  R$.
 
 Assume first for simplicity  that $\mu$ has finite support.
To prove the desired estimate we work with potentials, so as before  let
 $u (\lambda,g  )  =  \log\abs{\tr^2(g_\lambda) - t}.$
 Let
 $\mathcal{P}'_t = \bigcup_{\set{s,\ \abs{s-t}\leq 1}}\mathcal{P}_s$ be the set of $g$ such that $\tr^2(g_\lambda)$ is a constant close to $t$. By the Large Deviations Theorem for the traces (Corollary \ref{cor:large deviations}) $\mu^n({\mathcal{P}'_t})$ decreases to zero exponentially fast. We  define $u_n(\lambda)$ by the formula  $$u_n(\lambda ) = \unsur{2n} \int_{G\setminus \mathcal{P}'_t} u(\lambda,g )  d\mu^n(g).$$
This is a psh potential of  $\unsur{2n} \int  \left[Z(g,t)\right] d\mu^n(g)$. Now to prove \eqref{eq:equidist speed} it is enough to show that if $\La'\subset\La$ is a relatively compact  open subset, 
\begin{equation}\label{eq:lognsurn}
\norm{u_n-\chi}_{L^1(\La')} = O\lrpar{\frac{\log n}{n}} .
\end{equation} 

Assume that the condition {\it ii.} of Theorem~\ref{thm:equidist speed} holds. 
 We know from Lemma \ref{lem:traceminor ii} that $\norm{u(\cdot, g)}_{L^1(\La')}\leq C\length(g)$. %Using a theorem due to  Zeriahi \cite{zeriahi}, we   now give an estimate of 
Using standard estimates for the volume of sublevel sets of psh functions, we can control the volume of the set of representations possessing
an element with trace too close to $t$. 

\begin{lem}\label{lem:volume} Assume that the condition {\it ii.} of Theorem~\ref{thm:equidist speed} is satisfied.
Fix a relatively compact open subset $\La'\subset \La$, and a positive constant $A$.
Then if $B>0$ is large enough the volume of the open set
$$V_n = \set{\la\in \La' \text{ s.t. }  \text{ there exists } g\notin \mathcal{P}'_t  \text{ of length } \leq  An  \text{ with } \abs{\tr^2(g_\la) -t}< e^{-Bn^2}}$$ is exponentially small in $n$.
\end{lem}

\begin{proof}
It is no loss of generality to assume that $\La'\Subset\Omega\Subset\La$, where $\om$ is biholomorphic to (and viewed as) the  ball $B(0,1)$
 in $\cc^{\dim(\La)}$. Fix $g\in G$ and let 
$$\widetilde u(\lambda, g) = \unsur{\length(g)}u(\la, g) = \unsur{\length(g)}\log \abs{\tr ^2 g_{\lambda }- t}.$$
%The family of psh functions  $\widetilde u(\cdot, g)$ is relatively compact in $L^1(\La')$, and by Lemmas  \ref{lem:traceminor ii} and \ref{lem:three circles}, 
%there exists a universal constant $\Theta$ such that for every $\la_0\in  
%\La'$, the Lelong number
%$\nu(\widetilde  u(\cdot, g), \lambda_0)$ is bounded by $\Theta$.

%By \cite[Thm 4.2]{zeriahi}, for every $\alpha< 2/\Theta$, there exists constants $C_{1}$, $C_2$
%such that for every $B>0$
%$$\vol \lrpar{\set{ \la \in \La',\ \widetilde  u(\la, g)<-M}}\leq C_1 \exp\lrpar{
%C_2\norm{\widetilde  u(\cdot, g)}_{L^1(B(0,2))}} \exp(-\alpha M).$$
%In this equation,    we put  $M = \frac{Bn^2}{\length(g) }$, $\alpha = 1/\Theta$,
%and   the fact that
%$\length(g)\leq A n$ to infer that there exists a constant $C$ such that 
%$$\vol \lrpar{\set{ \la \in \La',\ \abs{\tr^2(g_\la) -t}< e^{n^2}}} \leq
%C  \exp\lrpar{ -\frac{ Bn}{\Theta A }}. $$  
%To finish the proof, it is enough to sum this estimate over all words of length $\leq An$. Since the number of such elements is exponential in $n$ (bounded by $Ce^{C(A)n}$), we conclude that if $B$ is large enough (i.e. $B> \Theta A C(A)$), $\vol(V_n)$ is exponentially small.
By Lemma \ref{lem:traceminor ii}, the family of psh functions  $\widetilde u(\cdot, g)$ is relatively compact in $L^1(\La')$, so  by Lemma   \ref{lem:three circles},  there exist constants $M$ and $A$ independent of $g$ such that for every $g$, $\sup_{B(0,1)} \widetilde u(\cdot, g)\geq M$ and 
 there exists a point $x_0 \in B(0, 1/2)$ such that $\widetilde u(x_0, g)\geq -A$. 

By \cite[Theorem 4.4.5]{hormanderCV} (combined with Lemma \ref{lem:three circles})\note{en fait on peut faire ref à hormander proposition 4.2.9} there exist constants $c_1$ and $c_2$ such that $\int \exp(-c_1 \widetilde u(\cdot, g))\leq c_2$ uniformly in $g$. Thus by the Markov inequality, there exists a constant $a$ such that for every $s>0$,
$$\vol \lrpar{\set{ \la \in \La',\ \widetilde  u(\la, g)<-s}}\leq C e^{-as}.$$
In this equation,   we put  $s = \frac{Bn^2}{\length(g) }$,
and infer that there exists a constant $C$ such that if
$\length(g)\leq A n$,
$$\vol \lrpar{\set{ \la \in \La',\ \abs{\tr^2(g_\la) -t}< e^{-Bn^2}}} \leq
C  \exp\lrpar{ -\frac{ aBn}{ A }}. $$  
To finish the proof, it is enough to sum this estimate over all words of length $\leq An$. Since the number of such elements is at most exponential in $n$ (bounded by $Ce^{C(A)n}$), we conclude that if $B$ is large enough (i.e. $B>  A C(A)/a$), $\vol(V_n)$ is exponentially small.
\end{proof} 

 \begin{rmk}
A similar estimate was proven  in $\mathrm{SU}(2)$ by Kaloshin and Rodnianski in \cite{kr}, by a different method. It appears that the use of  pluripotential theoretic tools leads to a short proof of their result (see Remark \ref{rmk:moderate} below).   
 \end{rmk}
 
 Let us now assume that {\it i.} holds. 
To estimate the $L^1$ norm of $u(\cdot,g)$, we use 
the results of Appendix \ref{app:number}.

\begin{lem}\label{lem:traceminor i}
Under    the  assumption   {\it i.}  of Theorem~\ref{thm:equidist speed}, for every
relatively compact open subset $\La'\subset \La$ there exists a constant $\delta>0$ such that for every $g\in G$, if $\la\mapsto\tr^2(g_{\lambda})$ is not  constant, then
$$\sup_{\la\in\Lambda'} \abs{\tr^2(g_\la) -t} \geq \delta^{\length(g)\log(\length(g))}.$$
\end{lem}

\begin{proof} Recall that  $\La$ can be viewed as an open subset of an affine subvariety in $\cc^{9k}$ ($k$ is the number of generators). %As before, For $1\leq i\leq k$, let $(a_{i, j})_{1\leq i\leq 9}$ be the coefficients corresponding to the generator $g_i$.
If $g\in G$ is any element, then as before $\tr^2(g)$ is the restriction to $\La$ of a polynomial in the $a_{i,j}$ (the matrix coefficients corresponding to the generator $g_i$) 
of degree $O(\length(g))$, with integer coefficients. Furthermore, there exists a constant $D$ such that each of these coefficients is bounded by $D^{\length(g)}$. The estimate that we seek is now  a direct consequence of Corollary \ref{cor:philippon}. 
\end{proof}

We have a version of Lemma \ref{lem:volume} in this context. 

 \begin{lem}\label{lem:volume i} Assume that condition {\it i.} of Theorem~\ref{thm:equidist speed} holds.   
 Fix a relatively compact open subset $\La'\subset \La$, and a positive constant $A$.
Then if $B>0$ is large enough the volume of the open set
$$V_n = \set{\la\in \La' \text{ s.t. }  \text{ there exists } g\notin \mathcal{P}'_t  \text{ of length } \leq  An  \text{ with } \abs{\tr^2(g_\la) -t}< e^{-Bn^2\log n}}$$ is exponentially small in $n$.
\end{lem}

\begin{proof}
It follows from the previous 
  lemma  and \eqref{eq:suptrace} that
there exist a constant  $C$ such that for every $g\in G$,
\begin{equation}\label{eq:controle L1 i}
 \norm{u(\cdot, g)}_{L^1(\La')}\leq C  \ \length(g)\log(\length(g)),
\end{equation} so we simply put $\widetilde u(\lambda, g) = \unsur{\length(g)\log(\length(g))}u(\la, g)$ and argue as in Lemma \ref{lem:volume}.
 \end{proof}
 
We now resume the proof of the theorem. The less favorable situation is when {\it i.} holds, so let us put ourselves in this case.
 Fix a constant $A$  in Lemma \ref{lem:volume i} such that   $\supp(\mu^n)\subset B_G(\mathrm{id}, An)$, and a corresponding constant $B$. Then, with notation as in the lemma, $V_n$ has exponentially small volume. To prove that 
 $\norm{u_n-\chi}_{L^1(\La')} = O\lrpar{\frac{\log n }{n}}$, we will proceed in two steps: first show that this estimate holds pointwise outside $ V_n$, by using Theorem \ref{thm:distance}, and then use general facts on psh functions to get a global $L^1$ bound.
  
So first fix a parameter $\lambda\notin V_n $. As in Theorem \ref{thm:distance}, let $\delta(g_{\la})$ be the distance between the fixed points of  
 $g_{\la}$. Let $\e_n = n^{-\alpha}$, where $\alpha>0$ is a   constant to be fixed later, and let 
 $$E_n^1 = \set{g\in G,\ \delta(\rho_{\la}(g))\geq \e_n}.$$
 By Theorem \ref{thm:distance}, $\mu^n(E_n^1)\geq 1-C\e_n^K$ for some $K$. 
   
  Let now $\chi_n (\la)= \unsur{n} \int_G \log\norm{g_{\la}} d\mu^n(g)$. We claim that  
 $\norm{\chi_n-\chi}_{L^\infty} = O\lrpar{\unsur{n}}$, in which case it will be enough to prove that  $\abs{u_n(\la) - \chi_n(\la)} = O\lrpar{\frac{\log n }{n}}$.
 Indeed we know from Proposition \ref{prop:Ounsurn} that if $z_0\in \pu$ is fixed, then   $\norm{\chi_{n, z_0}-\chi}_{L^\infty} = O\lrpar{\unsur{n}}$, where  $ \chi_{n, z_0}(\la) = \unsur{n} \int \log{\norm{g_\la(z_0)}}  d\mu^n(g) =  \unsur{n} \int \log\frac{\norm{g_\la(Z_0)}}{\norm{Z_0}} d\mu^n(g)$. Moreover it is clear from the proof of that lemma that the $O(\cdot)$ is uniform with respect to $z_0$.
Thus to get our claim  it is enough to  integrate with respect to $z_0$ and apply  Lemma \ref{lem:norme integree}. 
 
 Let $E_n^2$ be the set of those $g\in E_n^1$ such that moreover  $g\notin \mathcal{P}'_t$ and $\abs{\unsur{n} \log\norm{g_{\la}} -\chi(\la)}\leq \chi(\la)/2$. Then $\mu^n(E_n^1\setminus E_n^2)$ is exponentially small and accordingly, 
$\mu^n(E_n^2)\geq 1-C\e_n^K$.  
 
 We split  $u_n-\chi_n$ as 
 \begin{align*}
 u_n(\la)  - \chi_n(\la) = \unsur{n}&
 \int_{E^2_n}\left( \unsur{2}\log \abs{\tr^2(g_{\la})-t} - \log\norm{g_{\la}}\right) d\mu^n(g) + \\ &+  \unsur{n}
 \int_{(E^2_n)^c\setminus \mathcal{P}_t}  \unsur{2}\log \abs{\tr^2(g_{\la})-t}d\mu^n(g)  -  \unsur{n} \int_{(E^2_n)^c} \log\norm{g_{\la}} d\mu^n(g)  \end{align*}

 To estimate the first integral, we use Lemma \ref{l:norm, trace and distance between fixed points}. Indeed, since $\delta(g_\la)\geq \e_n=n^{-\alpha}$ we see that 
 if $g\in E_n^2$, $\tr^2(g_\la)$ is of order of magnitude $e^{n\chi(\la)}$ so 
 $\unsur{n} \log\abs{\tr^2(g_{\la})-t}  = \unsur{n} \log\abs{\tr^2(g_{\la})-4} + o\lrpar{\unsur{n}}$, and  
 $\unsur{2}\log \abs{\tr^2(g_{\la})-t} - \log\norm{g_{\la}} \sim (\log \delta(g_\la)) = O(\log n)$. We deduce that this first integral is $O\lrpar{\frac{\log n}{n}}$. 
 
 The third integral is bounded by $C\mu^n((E_n^2)^c) = C\e_n^K=Cn^{-\alpha K}$ which is $O(n^{-2})$ if $\alpha$ is large enough (recall that $K$ does not depend on $\alpha$). 
   
 Finally, to estimate the second integral, we use the fact that $\la\notin V_n$.  From this we infer  that $\abs{\log\abs{\tr^2(g_\la) - t}}\leq  Bn^2\log n$, thus the integral is bounded by 
 $\unsur{n} \mu^n{(E_n^2)^c} Bn^2\log n  = O(n^{1-\alpha K}\log n)$, which again is $O(n^{-2})$ for large   $\alpha$. 
 
 \medskip
 
 It is clear that all the $O(\cdot)$  appearing in the above  reasoning are  uniform for  $\la\in \La'\setminus V_n$. Thus at this point we know that $u_n$ is a psh function, bounded from above (by \eqref{eq:suptrace}), with $\norm{u_n}_{L^1(\La')}\leq C \log n$ (by \eqref{eq:controle L1 i}; notice that the $\log n$ is superfluous under  {\it ii.}) and $\norm{u_n-\chi}_{L^\infty(\La'\setminus V_n)} = O\lrpar{\frac{\log n}{n}}$.%, where $\vol(V_n)$ is exponentially small. 

To complete the proof of the desired estimate \eqref{eq:lognsurn} in case $\mu$ has finite support, it remains to show
that  $\norm{u_n-\chi}_{L^1(V_n)} =O(\frac{\log n}{n})$, which is done in the following easy lemma. 
 
\begin{lem}
 Under the above assumptions and notation, $\norm{u_n-\chi}_{L^1(V_n)}$ is exponentially small in $n$. 
\end{lem}

\begin{proof}
Recall from Lemmas \ref{lem:volume} and \ref{lem:volume i} that 
  $\vol(V_n)$ is exponentially small. By boundedness of $\chi$ it follows that 
    $\norm{\chi}_{L^1(V_n)}$ is exponentially small.

To control $\norm{u_n}_{L^1(V_n)}$, put $\widetilde{u}_n =\unsur{\log n} u_n$.
Since   $(\widetilde u_n)$ is bounded in $L^1$, it is no loss of generality to assume that $\widetilde u_n\leq 0$, then as before there exists a   constant $a$ such that $\vol(\set{\widetilde u_n<-M})\leq e^{-aM}$. Thus we simply write
 $$\int_{V_n} \abs{\widetilde u_n} = \int_{V_n\cap\set{\widetilde u_n<-M}} \abs{\widetilde u_n} +  \int_{V_n\cap\set{\widetilde u_n\geq-M}} \abs{\widetilde u_n} \leq \int_{\set{\widetilde u_n<-M}} \abs{\widetilde u_n} + M \vol(V_n) ,$$ and use the coarea formula $$\int_{\set{\widetilde u_n<-M}} \abs{\widetilde u_n} = \int_0^M \vol(\widetilde u_n<-M)dt + \int_M^{+\infty} \vol(\widetilde u_n<-t)dt = O(Me^{-aM})$$ to deduce that
 $\int_{V_n} \abs{\widetilde u_n} = O(Me^{-aM} + M\vol(V_n))$. To conclude that $\int_{V_n} \abs{\widetilde u_n}$, whence $\int_{V_n} \abs{ u_n}$, is exponentially small, it  suffices to pick $M=n$. 
\end{proof}

\medskip

It remains to treat the case where  the support of   $\mu$ is infinite. Recall that we assume that $\mu$ satisfies 
   \eqref{eq: exponential moment condition in the group}.
  We adapt the proof  by  using the exponential moment condition to show that $\mu$ almost behaves like a measure with finite support, and obtain exponentially decaying estimates for the resulting errors.
Again we work under the less favorable assumption {\it i}.

It is an easy consequence of
 the moment condition that that for a sufficiently large constant $A$, $\mu^n(B_G(\mathrm{id}, An)^c)$ and more generally $\int_{B_G(\mathrm{id}, An)^c} \length(g)d\mu^n(g)$ tend to zero exponentially fast. Indeed, let
 $I = \int_G \exp(\tau \length(g))d\mu(g)$ which is finite by assumption. Subadditivity of the length implies that $\int_G \exp(\tau \length(g))d\mu^n(g) \leq I^n$, therefore,  by the Markov inequality, for every $s>0$, $\mu^n(\set{g, \ \length(g)\geq s})\leq \exp(-\tau s)I^n$.

We then infer that
\begin{align*}
\int_{B_G(\mathrm{id}, An)^c}\!\!\length(g)\log(\length(g)) d\mu^n(g)  &= %\int_O^{An}\mu^n(B_G(\mathrm{id}, An)^c)ds +
%\int_{An}^\infty \mu^n(B_G(\mathrm{id},s)^c) ds = O(nI^n e^{-\tau An}),
\sum_{k=An+1}^\infty k(\log k) \mu^n\lrpar{\set{g,\ \length(g) = k}} \\&\leq  I^n \sum_{k=An+1}^\infty k(\log k) e^{-\tau k}
\end{align*}
 which decreases exponentially if $A$ is sufficiently large.

Now, recall from Lemma \ref{lem:traceminor ii} that $\norm{u(\cdot, g)}_{L^1(\La')} = O(\length(g))$. A first consequence is that the sequence $(u_n)$ is bounded in $L^1(\La')$. With notation as before, it is enough to show that $\norm{u_n-\chi_n}_{L^1(\La')}  = O\lrpar{\frac{\log n}{n}}$.
For a constant  $A$ as just above, decompose $u_n$ as
$$u_n = \unsur{2n}\int_{(G\setminus \mathcal{P}'_t)\cap B(\mathrm{id}, An)}\! u(\cdot, g) d\mu^n(g) +
\unsur{2n}\int_{(G\setminus \mathcal{P}'_t)\cap B(\mathrm{id}, An)^c}\!  u(\cdot, g) d\mu^n(g) =: u_n^1+ u_n^2,$$ and similarly for $\chi_n$, and write $u_n-\chi_n = (u_n^1-\chi_n^1)  + (u_n^2-\chi_n^2)$.
The first part of the proof shows that $\norm{u_n^1-\chi_n^1}_{L^1(\La')}= O\lrpar{\frac{\log n}{n}}$
while the above considerations imply that $\norm{u_n^2-\chi_n^2}_{L^1(\La')}$ decreases to zero exponentially fast. The proof is complete.
\end{proof}

\begin{rmk}\label{rmk:moderate}
The proof actually says more. A measure $m$  on $\La$ is said to be {\em moderate} \cite{ds} if for any $\La''\Subset \La'\Subset \La$,  there exist constants $C,\alpha>0$ such that if $u$ is a  psh function   with
$\norm{u}_{L^1(\La')}\leq 1$, then for every $s>0$, $$m(\set{\la\in\La'',\ u(\la)<-s})\leq Ce^{-\alpha s}.$$
An obvious adaptation of the proof shows that $\norm{u_n-\chi}_{L^1_{\rm loc}(m)} = O\lrpar{\frac{\log n}{n}}$ for any moderate measure $m$.

This observation has several interesting consequences. Dinh and Sibony \cite{ds} showed that if $T$ is a (1,1) current
 with H\"older continuous potentials, then its trace measure $\sigma_T$, and more generally that of its successive exterior powers $T\wedge\cdots\wedge T$ are moderate. As a consequence, if we let $Z_n = \unsur{2n} \int  \left[Z(g,t)\right] d\mu^n(g)$, then for every $q\leq \dim(\La)-1$, if $\phi$ is a test form of %bidimension $(\dim(\La)-k-1)$, 
 the right dimension, 
 $$\bra{Z_n\wedge T_{\rm bif}^q - T_{\rm bif}^{q+1}, \phi}\leq C {\frac{\log n}{n}} \norm{\phi}_{C^2}.$$ Such an estimate might prove useful (as   Proposition \ref{prop:Ounsurn} was  for $q=1$) 
 when trying to characterize $\supp(T_{\rm bif}^q)$ (see \S \ref{subsub:wedge}).
 
It is a classical fact  that the area measure on  a totally real submanifold of maximal dimension is moderate --this may also easily be deduced from \cite{ds}.  This applies in particular to  the area measure on $\mathrm{SU}(2)\subset \SL$. Thus, arguing exactly as in Lemma \ref{lem:volume}, we can recover the original
Kaloshin-Rodnianski estimate \cite{kr}: there exists a constant $B>0$ such that  the volume of the set of $(u,v)\in\mathrm{SU}(2)$ with the property  that  there exists a word $w$ of length $n$ in the free group $\mathbb{F}_2$ such that   $\mathrm{dist}(w(u,v), \mathrm{id})<\exp (-Bn^2)$, is exponentially small (recall that if $w\in \mathrm{SU}(2)$, 
$\norm{w-\mathrm{id}}^2 \asymp  \abs{\tr(w)-2}$). For this, we use the fact that {\it ii.} holds in this context, that is, the free group $\mathbb{F}_2$  admits geometrically finite representations  into $\SL$, e.g. Schottky subgroups.
  Notice that the same applies to $\mathrm{U}(n)\subset \mathrm{GL}(n,\cc)$ (resp. $\mathrm{SU}(n)\subset \mathrm{SL}(n,\cc)$), and to free groups with arbitrary many generators.
\end{rmk}

\subsection{Motion of fixed points in $\La\times\pu$}
 Keeping notation as in \S \ref{subs:geom_interpt}, for $g$ in $G\setminus\mathrm{id}$,
we let  $\fix(\widehat g)$ be the hypersurface in $\La\times \pu$ defined by  the equation
$\set{(\la,z), \ g_\la(z) = z}$ (counted with its multiplicity). Notice that if $\lambda_0$ is such that $g_{\lambda_0} = \mathrm{id}$, then $\set{\lambda_0}\times \pu \subset \fix(\widehat g)$.

\begin{thm}\label{thm:motion of fixed points}
Let $(G,\mu,\rho)$ be an  admissible family of representations of $G$ into $\PSL$
 satisfying the exponential moment condition~\eqref{eq: exponential moment condition in the group}.

Then for $\mu^\nn$-a.e.  $\mathbf{g}\in G^\nn$,  the sequence of   currents of bidegree $(1,1)$
$\unsur{ n}  \left[ \fix\left(\widehat {l_n(\mathbf{g}}) \right)\right]$
on $\La\times\pu$ converges to $\pi_1^*(T_{\rm bif})$.

If furthermore, one of the additional assumptions {i.} and  {ii.} of Theorem \ref{thm:equidist integre} is satisfied, then the convergence takes place in $L^1(\mu^\nn)$.
\end{thm}

Again we may interpret this by saying that $\la_0\in \supp(T_{\rm bif})$ iff for every neighborhood $U$ of $\lambda_0$ the average volume of $\fix\left(\widehat {l_n(\mathbf{g}})\right)\cap \pi_1^{-1}(U)$ grows linearly with $n$.
%Since the fixed points move holomorphically over the stability locus, we obtain yet another proof of the fact that $\supp(T_{\rm bif})\subset \mathrm{Bif}$.

\begin{proof}
The proof is similar to that of Theorem \ref{thm:geom_interpt} so we shall be brief. Again, the result is local on $\La$ so we may assume it is a ball, and we decompose the K\"ahler form in $\La\times\pu$ as $\widehat{\omega} =  \pi_1^*\omega+ \pi_2^*\omega_\pu$.

For every $g\in G$, $\pi_1: \fix(\widehat g) \cv \Lambda$  is a dominant mapping of degree at most 2, with possibly some exceptional fibers corresponding to parameters where $\lambda_0 = \mathrm{id}$.  In any case we infer that  $\langle{\widehat{T}_n, \pi_1^*\omega^k}\rangle\cv 0$.  Thus, again,  what we need to analyze is
pairings of the form $\langle \widehat{T}_n,  \pi_2^*\omega_{\pu}\wedge \pi_1^*\phi\rangle = \bra{(\pi_1)_*( \widehat{T}_n\wedge   \pi_2^*\omega_{\pu}), \phi}$,
where $\phi$ is a $(k-1, k-1)$ test form on $\La$.

Let $I(g)$ (resp. $P(g)$ be the subvariety of $\La$ defined by $I(g) = \set{\la, \ g_\la =  \mathrm{id}}$ (resp. $P(g)=\set{\la, \ g_\la \text{ is parabolic}}$).
Consider an open subset  $\om\subset\La$  disjoint from $I(g) $, so that $\pi_1\rest{ \fix(\widehat g)\cap \pi_1^{-1}(\om)}$ is a branched cover of degree 1 (in the case of a persistently parabolic element) or 2 (in the other cases).

Suppose first that $g$ is not persistently parabolic, and pick a ball $U$ where $g_\la$ is never parabolic.
 In this case, $\fix(\widehat g)\cap \pi_1^{-1}(U)$ consists of two graphs $\fix_1(\widehat g)$ and $\fix_2(\widehat g)$ over $U$ corresponding to the two fixed points of $g_\la$. Let us denote these by $f_i(\la)$, $i=1,2$. As in \eqref{eq:chapeau} we obtain that
 $$ \big\langle[\fix (\widehat g)], \pi_2^*\omega_{\pu}\wedge \pi_1^*\phi \big\rangle
 = \big\langle[\fix_1(\widehat g)] + [\fix_2(\widehat g)], \pi_2^*\omega_{\pu}\wedge \pi_1^*\phi \big\rangle =\int_\La\lrpar{ (  f_1)^*\omega_\pu +
(  f_2 )^*\omega_\pu}\wedge \phi$$
%\begin{align*}
%\langle \widehat{T}_n,  \pi_2^*\omega_{\pu}\wedge \pi_1^*\phi\rangle
% &= \unsur{n} \int \big\langle[\fix_1(\widehat g)] + [\fix_2(\widehat g)], \pi_2^*\omega_{\pu}\wedge \pi_1^*\phi \big\rangle d\mu^n(g) \\ &=
% \unsur{ n}
%\int \lrpar{\int_\La (  f_1)^*\omega_\pu\wedge \phi +
%(  f_2 )^*\omega_\pu\wedge \phi}d\mu^n(g).
%\end{align*}
Let $g = \left(\begin{smallmatrix} a & b \\ c& d \end{smallmatrix}\right)$; the coefficients are defined only up to sign, but this does not affect the foregoing formulas. For the matter of computation we may assume that $c_\la$ never  vanishes in $U$, so that $f_1$ and $f_2$ take their values in  a fixed  affine chart $\cc\subset \pu$. We then obtain  that
\begin{align*}
(f_1)^*\omega_\pu + (f_2)^*\omega_\pu &= dd^c  \log\left(1+ \abs{f_1}^2\right)^{\unsur{2}} +
 dd^c  \log\left(1+ \abs{f_2}^2\right)^{\unsur{2}} \\
 &= dd^c \log \left(\abs{b_\la}^2 + \abs{c_\la}^2 + \frac{\abs{d_\la-a_\la}^2 + \abs{\tr ^2(g_\la)-4}}{2}\right)^{\frac12}.
 \end{align*}
The last equality is in turn also valid   when $c_\lambda$ vanishes. Let $v(\la, g)$ be the argument of the $dd^c$ in the last line.

Assume that $P(g)\cap \om$ is not empty (recall that by assumption $I(g)\cap \om = \emptyset$). The function $v(\la, g)$  is locally bounded near $P(g)$ so $dd^cv(\cdot, g)$ gives no mass to $P(g)$. Likewise,   $[\mathrm{Fix}(\widehat g )]$ gives no mass to $\pi^{-1}(P(g)\cap \om)$. Therefore we conclude that $v(\cdot, g)$ is a potential of
$(\pi_1)_*\left([\fix(\widehat g)]\wedge \pi_2^*\omega_\pu\right)$ throughout $\om$.

It is straightforward to check that the same holds when $g_\la$ is persistently parabolic.
%In the case where $g$ is  persistently parabolic the same computation gives
%$2 (f_1)^*\omega = dd^c \log (\abs{b_\la}^2 + \abs{c_\la}^2 +\abs{d_\la-a_\la}^2/2)$. Since $a_\la+ d_\la =\pm 2$, $d_\la-a_\la$ is of the order of magnitude of $d_\la$ and $a_\la$, thus we see that $\log (\abs{b_\la}^2 + \abs{c_\la}^2 +\abs{d_\la-a_\la}^2/2)$ is equivalent  to $\log \norm{g}$.

\medskip

At this point we know that  $v(\cdot, g)$ is a potential of $(\pi_1)_*\left([\fix(\widehat g)]\wedge \pi_2^*\omega_\pu\right)$ outside $I(g)$. We claim that this is actually true
    everywhere on $\La$. Notice first that $v(\cdot, g)$ is a well-defined psh function, with poles on $I(g)$.  Let $\Sigma$ be an irreducible component of
$I(g)$. The extension of $v(\cdot, g)$ as a potential of $(\pi_1)_*\left([\fix(\widehat g)]\wedge \pi_2^*\omega_\pu\right)$ over $\Sigma$ is immediate when $\codim(\Sigma)>1$, since 
 since neither $dd^cv(\cdot, g)$ nor $(\pi_1)_*\left([\fix(\widehat g)]\wedge \pi_2^*\omega_\pu\right)$ would carry any mass on $\Sigma$ in this case. So we can suppose that  $\codim(\Sigma)=1$, and, slicing by 1-dimensional submanifolds, we may further assume that
$\dim(\La)=1$ and $\Sigma = \set{\lambda_0}$. In this case,  the measure  $(\pi_1)_*\left([\fix(\widehat g)]\wedge \pi_2^*\omega_\pu\right)$ has an atom of multiplicity $m$ at $\la_0$, where $m$ is the generic multiplicity of $\fix(\widehat g)$ along $\pi_1^{-1}(\Sigma)$. Let us compute $m$: we are looking  at
the multiplicity of the root $\la_0$ of the equation $c_\la z^2+ (d_\la-a_\la)z+ b_\la = 0$ for generic $z$, at a parameter $\lambda_0$ where
$b_{\la_0} = c_{\la_0} = d_{\la_0}-a_{\la_0} = 0$. We infer that
$$m=\min \left(\mathrm{mult}_{\la_0}(b_\la), \mathrm{mult}_{\la_0}(c_\la),\mathrm{mult}_{\la_0}(d_\la-a_\la)\right).$$

To prove that $dd^cv(\cdot, g) =
(\pi_1)_*\left([\fix(\widehat g)]\wedge \pi_2^*\omega_\pu\right)$, it is enough to show that $dd^cv(\cdot, g)$ also admits an atom of multiplicity $m$ at $\lambda_0$. Equivalently, we need to show that $v(\cdot, g)$ has a logarithmic pole of order $m$ at $\lambda_0$,
which is clear from the  formula defining $v$ and the observation that $\mathrm{mult}_{\la_0}(\tr^2(g_\la)-4)\geq 2m$.

\medskip

% Form a potential of $(\pi_1)_*( \widehat{T}_n\wedge   \pi_2^*\omega_{\pu})$ by putting $v_n = \unsur{n}\int v(\la, g) d\mu^n(g)$.
% By definition of $v(\cdot, g)$,
% it is clear that there exists a constant $C$ independent  of $g$ such that
% $$\min\left (\frac{1}{2}\log\abs{\tr^2(g_\la) - 4} ,  \log \norm{g_\la} \right)-C\leq v(\la, g)\leq  \log \norm{g_\la} + C,$$
% hence, from (the proof of) Theorem \ref{thm:equidist} we see  that $v_n$ is a  well-defined sequence of psh functions converging  to $\chi$ in $L^1_{\rm loc}(\La)$.   Arguing exactly as in Theorem \ref{thm:geom_interpt}, we thereby conclude  that $\widehat{T}_n$ converges to $\pi_1^*(T_{\rm bif})$.

By definition of $v(\cdot, g)$,
 it is clear that there exists a constant $C$ independent  of $g$ such that
 \begin{equation}\label{eq:encadrement}
  \min\left (\frac{1}{2}\log\abs{\tr^2(g_\la) - 4} ,  \log \norm{g_\la} \right)-C\leq v(\la, g)\leq  \log \norm{g_\la} + C. 
 \end{equation} 
 %Let now $\mathbf{g}\in G^\nn$ be chosen as in the proof of Theorem \ref{thm:equidist ae}.
%By the same argument as in that proof, we infer that the sequence of psh functions
%$\unsur{n} v(\cdot, l_n(\mathbf g))$ converges  to $\chi$ in $L^1_{\rm loc}(\La)$. Proceeding exactly as in Theorem \ref{thm:geom_interpt}, 
As usual, we   conclude from Theorem \ref{thm:abstrait} that  $\unsur{ n}  \big[ \fix(\widehat {l_n(\mathbf{g}})) \big]$ converges to $\pi_1^*(T_{\rm bif})$.

%We leave the reader check that under each of the  assumptions (i.) or (ii.) of Theorem \ref{thm:equidist integre}, the  second assertion of the theorem follows by dominated convergence as before. 
Finally,  we leave the reader check that under each of the  assumptions {\em i.} or {\em ii.} of Theorem \ref{thm:equidist integre}, \eqref{eq:encadrement} shows that the mass of $\unsur{ n}  
\big[ \fix(\widehat {l_n(\mathbf{g}})) \big]$ is locally controlled by $\length(g)$. So the second assertion of the theorem follows from   Proposition \ref{prop:integree}. 
\end{proof}

\section{Further comments}\label{sec:further}

\subsection{Canonical bifurcation currents}\label{subs:bers slice}
Let $S$ be a Riemann surface of genus $g\geq2$ and $\mathrm{Hom}(\pi_1(S), \PSL)$ the set of representations of $ \pi_1(S)$ into $\PSL$. Take a holomorphic family $\La\subset \mathrm{Hom}(\pi_1(S), \PSL)$ made of non-elementary representations. Our purpose in this paragraph is to outline the construction of  a canonical bifurcation current on $\La$,   depending only on the Riemann surface structure. The details will appear in a subsequent paper \cite{dd2}.

The idea consists in replacing the \textit{discrete} random walk on $\pi_1(S)$ by a \textit{continuous} Markov process on $S$: the Brownian motion with respect to a conformal metric. This defines a Lyapunov exponent, very much in the spirit of~\cite{dk}. The induced function on $\La$ is well-defined up to a multiplicative constant, because of the conformal invariance of the Brownian motion. %A good normalization is as usual the choice of the Poincar\'e metric.

To be more precise, denote by $\widetilde S$ the universal cover of $S$. Given a representation $\rho\in \La$, consider the flat $\mathbb P^1$-bundle over $S$ with monodromy $\rho$, that we denote by $X$. Recall that it is obtained by taking the quotient of  $\widetilde S\times\pu$ by the diagonal action of $\pi_1(S)$   (defined by $\gamma(x,z) = (\gamma x , \rho(\gamma) z)$). If $x\in S$, we denote the fiber of the bundle over $x$ by $X_x$. Observe that to any oriented continuous path $\gamma$ with endpoints $x$ and $y$  corresponds a holonomy map $h_{\gamma}$ from $X_x$ to $X_y$, obtained by lifting $\gamma$ as a family of continuous paths in the flat sections. 
A spherical metric $\norm{\cdot}$ being given on the $\mathbb P^1$-fibers, for every path $\gamma: [0,\infty) \rightarrow S$   we  may consider the limit
\begin{equation}\label{eq: lyapunov} \chi (\gamma) = \lim _{t\rightarrow \infty} \frac{\log \norm{ h_{\gamma_{|[0,t]}} }}{t}. \end{equation}
If $\gamma$ is a generic Brownian path, the limit in~\eqref{eq: lyapunov} indeed exists and only depends on the conformal metric and the representation (but not on $\gamma$). As   already said, two different conformal metrics give rise to Lyapunov exponent functions on $\La$  that differ only by a multiplicative constant.  To specify this constant it is enough to fix the metric as being
 the Poincar\'e metric of constant curvature $-1$.

%Consider the space $\Gamma$ of all continuous paths  $\gamma:[0, \infty)\cv S$, with the compact open topology. If $\gamma\in\Gamma $ is any such path,  for every $t>0$ we can consider the holonomy map $h_{\gamma\rest{[0,t]}}:X_{\gamma(0)}\cv X_{\gamma(t)}$, obtained by lifting $\gamma$ in the flat sections. The shift semi group acts on $\Gamma$ by the formula $\sigma_s (\gamma) (t) = \gamma (t+s)$ for $t,s\geq 0$, and the  family of functionals $H_s ( \gamma ) = \log\norm{ h_{\gamma \rest{[0,t]} } }$ defines  a sub-additive cocycle relative  the time shift $\sigma$. Therefore if we are able to find a probability measure on $\Gamma$ that is   invariant  under the shift and ergodic, then by the sub-additive ergodic theorem, the limit $$\chi(\gamma) =  \lim_{t\cv\infty} \unsur{t} \log\norm{h_{\gamma\rest{[0,t]}}}$$ will exist for a.e. $\gamma$, and will be independent of $\gamma$. 

%Now, once its complex structure   is fixed, $S$ admits a canonical metric of constant curvature $-1$ \note{en fait on peut prendre n'importe quelle métrique : les exposants de lyapunov brownien de toutes ces métriques diffèrent seulement par multiplication par une constante!}(the hyperbolic metric), producing two such (shift invariant and ergodic) measures on $\Gamma$. One is the Liouville measure, supported on the set of unit speed semi-infinite geodesics, and the second is the Wiener measure, produced by the Brownian motion. These measures gives raise to Lyapunov exponents $\chi _{\rm geodesic}$ and $\chi_{\rm Brownian}$ respectively. 

It is not difficult to  convince oneself that the function $\chi$ on $\La$ is psh. We can thus define a bifurcation current on $\La$, 
depending only on the complex structure on $S$, by the formula $T_{\rm bif} = dd^c\chi$.  What is less obvious is that there actually exists a measure on $\pi_1(S)$, possessing exponential moments, and such that $T_{\rm bif}$ is the associated bifurcation current on $\La$. In particular we have that: 

%A first possibility is to  consider the measure concentrated on the set of unit speed semi-infinite 
%geodesics,     which is induced by the Liouville measure on the unit tangent bundle. 
%The associated Lyapunov exponent is denoted by $\chi_{\rm geodesic} $.

%Another natural choice of measure on the space of paths is the Wiener measure $W$. More precisely, for $x\in S$, we let $W_x$ be the Wiener measure associated to Brownian paths starting at $x$, and $W = \int W_x d{\rm vol}(x)$ (of course ${\rm vol}$ is the canonical volume element). Let $\chi_{\rm Brownian} $ be the associated Lyapunov exponent.  

%In this way we obtain two natural functions on $\La$. It is not difficult to  convince oneself that these functions are psh and differ by a multiplicative constant. We can thus define a bifurcation current on $\La$, 
%depending only on a  choice of   complex structure on $S_0$, by the formula $T_{\rm bif} = dd^c\chi_{\rm Brownian} = dd^c\chi_{\rm geodesic}$.  \note{à modifier: il y a une constante multiplicative}

\begin{thm}
The support of $T_{\rm bif}$ is the bifurcation locus.
\end{thm}

It is also possible to state equidistribution theorems involving summations over the set of closed geodesics on $S$. 

\medskip

Here is a situation where these ideas naturally apply:  consider the set of complex projective structures over a Riemann surface $S$, compatible with its 
complex structure. This is an affine space of dimension $3g-3$,  
 admitting a distinguished point, namely the projective structure obtained by viewing $S$ as a quotient of the unit disk (see~\cite{dumas} for an introductory text on this).  The so-called (and much studied) {\em Bers slice} of Teichm\"uller space is the connected component of this point in the stability locus.
A projective structure induces  a {\em monodromy representation} (which is always non-elementary and defined only up to conjugacy) so the above discussion applies and we conclude that {\em the space of projective structures on $S$ admits a canonical bifurcation current.} We also show in \cite{dd2} that
the Lyapunov exponent function is {\em constant} on the Bers slice.  
Through the Sullivan dictionnary (as extended in \cite{mcm}), this corresponds to the theorem that the Lyapunov exponent of  a monic  polynomial of degree $d$ with connected Julia set is equal to $\log d$.

%\subsection{Other examples}

\subsection{Open questions}

\subsubsection{} Arguably the most important question left open in the paper is: how does $\tbif$ depend on $\mu$? For instance, are the bifurcation currents mutually singular/absolutely continous when $\mu$ varies?

In this context it may be interesting to note that if $\supp(\mu)$ is finite, then $\chi(\mu)$ is a real analytic function of the transition probabilities for a fixed representation~\cite{peres}.

Here is a related question: assume that $\La$ is the character variety of representations of $G$ into $\PSL$. How does the outer automorphism group $\mathrm{Out(G)}$ act on the bifurcation currents? Is it possible to find a measure $\mu$ so that $\mathrm{Out(G)}$ preserves the measure class induced by $T_{\rm bif}$?

\subsubsection{}\label{subsub:wedge} For spaces of rational maps,    
    the description of the exterior powers of $T_{\rm bif}$ is an important theme,  with again  some  emphasis   on  the  characterization of  their  supports and equidistribution theorems \cite{bas-ber, bas-ber1, bas-ber2, df, cubic, buff-epstein, gauthier} . The underlying ideology is that $\supp(T_\mathrm{bif}^k)$, for $1\leq k\leq \dim(\La)$ should define a dynamically meaningful filtration of the bifurcation locus.  

It is also  natural to investigate  this question in our context.  However, it seems that 
the supports of $\tbif^k$ do not give any new information here. To be precise, assume that  $\dim(\La)\geq 2$ and that
different representations in $\La$ are never conjugate (i.e. $\La$ is a subset of the character variety). Then {\em we conjecture that for every $k\leq \dim(\La)$, $\supp(T_{\rm bif}^k )=\mathrm{Bif}$}. 

Here is some evidence for this: let $\theta \in \re\setminus\pi\mathbb{Q}$, $t=4\cos^2(\theta)$ and look at the varieties $Z(g,t)$. Since for $\la\in Z(g,t)$, $\rho_\la$ is not discrete, the bifurcation locus of $\set{\rho_\la, \la \in Z(g,t)}$ is equal to $Z(g,t)$. Hence $\supp(\tbif\wedge [Z(g,t)]) = Z(g,t)$, which by Theorem \ref{thm:equidist ae} makes the equality $\supp(\tbif^2) = \supp(\tbif)$ plausible (see also Remark \ref{rmk:moderate}).

Notice  that the currents constructed by Cantat in \cite{cantat} as natural invariant currents under  holomorphic automorphisms of the character variety,  have zero self-intersection. 

\subsubsection{} We know that the normalized currents of integration over $Z(l_n(\mathbf{g}), 4)$ are equidistributed towards $\tbif$. Now, the parameters in $Z(l_n(\mathbf{g}), 4)$ can be of two types: parabolic or identity. Is there a dominant one?
One might guess that parabolic parameters prevail. %In any case, it even seems to be an open question whether accidental parabolics are dense in the bifurcation locus\footnote{E. Breuillard told us that if $G = \mathbb{F}_n$, $n\geq 3$ and $\La= \mathrm{Hom}(G, \PSL)$, the result easily follows from the techniques of \cite{gelander}.}.

\subsubsection{}  Do our results lead to efficient algorithms for producing computer pictures of stability/bifurcation loci? It is often a delicate issue in this type of problems to find numerical   criteria deciding whether a representation is discrete. Here, given a group endowed with a measure and a one-dimensional  family of representations, one    may simply try to  plot   numerically the Lyapunov exponent function   and look for regions where it is harmonic.

Another approach would  be  to  use Theorem \ref{thm:equidist ae} to obtain an approximation of the bifurcation locus by plotting the solutions to $\tr^2(\rho_\la(g)
)=4$ for (a small number of) large random elements $g = g_n\cdots g_1$.

\subsubsection{} Can $T_{\rm bif}$ be described more precisely in some   particular families (say, one-dimensional, so that $\tbif$ simply becomes a measure)? Are there measures $\mu$ for which $T_{\rm bif}$ is absolutely continuous with respect to Lebesgue measure? Are there measures for which $\tbif$ gives some mass to representations with values in $\text{PSL}(2,\mathbb R)$  ? % $SU(2,\cc)$?
Is the boundary of the bifurcation locus always of zero measure?

\appendix

\section{Estimates on the distance between fixed points and applications}\label{app:distance}

In this section we give some  refinements   of the classical results on random matrix products that we presented in \S\ref{subs:random}. Similar results have been recently proven in a much more general setting by  Aoun \cite{aoun}. In our context the proofs simplify greatly, so we include them  for convenience --besides, we need estimates slightly different from his.

If $\gamma\in \PSL$ we denote by $\delta(\gamma)$ the distance between its fixed points (on $\pu$).
Our first purpose is to show that on a set of large $\mu^\nn$ measure, $\rho(l_n(\mathbf{g}))$ is a loxodromic element with well-separated fixed points (recall that for   ${\bf g}\in G^{\mathbb N}$,   $l_n({\bf g})=g_n \cdots g_1$).

\begin{thm}\label{thm:distance}
Let $(G, \mu,\rho)$ be a  non-elementary representation
satisfying  the exponential moment condition~\eqref{eq:exponential moment}.

Then there exists a constant $K$ such that if
  $\e_n$ is one of the sequences $cn^{-\alpha}$ (for $c $, $\alpha>0$), %$\exp(-c n^\beta)$ (for $c>0$ and $0<\beta<1$ )
   or $\exp(-\gamma n)$ (for $0<\gamma<\gamma_0$ where $\gamma_0$ is an explicit constant), then
$n$ for large enough,
$$ \mathbb{P}\lrpar{\delta(\rho(l_n(\mathbf g)))<\e_n} \leq \e_n^K.$$
\end{thm}

As a first consequence, we obtain the following large deviation estimate in Theorem
\ref{thm:guivarch}

\begin{cor}\label{cor:large deviations}
Let $(G, \mu,\rho)$ be a  non-elementary representation
satisfying  the exponential moment condition~\eqref{eq:exponential moment}.

 Then for every positive  real number $\varepsilon$, the probability
\[ \mathbb P \left( \left| \frac{1}{n} \log \abs{\tr \rho (l_n({\bf g}) )} -\chi \right| > \varepsilon  \right) \]
decreases to zero exponentially fast when $n$ tends to infinity.
\end{cor}

Of course this implies Theorem \ref{thm:guivarch},   by applying the Borel-Cantelli lemma.

\begin{proof}[Proof of the corollary]
Under these  assumptions, the Large Deviation Theorem   holds for the distribution of the values of
$\unsur{n}\log \norm{l_n( \mathbf{g} ) }$ \cite[\S V.6]{bougerol-lacroix}. Thus if $\e '$ is small,
the probability $\mathbb P \left( \left| \frac{1}{n} \log \norm{l_n({\bf g}) } -\chi \right| > \varepsilon '  \right) $ is exponentially small in $n$.

From Lemma \ref{l:norm, trace and distance between fixed points} we know that when $\norm{\gamma}$ is large enough, $$\unsur{2}\log {\abs{\tr^2\gamma -4}} = \log  \norm{\gamma} + \log \delta(\gamma) + O(1).$$ So the result follows by taking $\e_n = \exp({-\e '' n})$ in Theorem \ref{thm:distance} where  $\e ', \e '' >0$ are such that $\e'+\e''<\e$.
\end{proof}

\begin{proof}[Proof of Theorem \ref{thm:distance}] We start with  a lemma.

\begin{lem}\label{l:exponential convergence}
Let  $(\e_p)$ be  as in the statement of   Theorem \ref{thm:distance}, then 
there exist constants $ C,K >0$ such that for any pair $(x_0,y_0)$ of points of $\pu$, and any integer $p$
\[ \mathbb P \left( \rho (l_{p} ({\bf g})) (x_0)\in B(y_0, \e_p) \right) \leq C \e_p^K .  \]
(here $B(y, r)$ is the ball of center $y$ and radius $r$ with respect to the spherical distance.)
\end{lem}

\begin{proof}
Recall that there are constants $C,\alpha, \beta>0$ such that for every integer $p$ we have
\begin{equation}\label{eq:exponential decay}  \norm{P^p - N }_{\alpha} \leq C e^{-\beta p} .\end{equation}
In addition, it is  known that the stationary measure has H\"older regularity, in the sense that  there are constants $C,\eta >0$ such that $\nu (B(x,r)) \leq C r^{\eta}$ for every  $x\in \pu$    and every radius $r>0$ (see \cite[p.161]{bougerol-lacroix})

Introduce a function $f: \mathbb P^1 \rightarrow [0,1]$ such that
\begin{itemize}
\itm $f$ is identically $1$ on the ball $B(y_0, \e_p)$,
\itm $f$ vanishes outside $B(y_0, 2\e_p)$,
\itm $f$ is Lipschitz with Lipschitz constant bounded by $\unsur{\e_p}$.
\end{itemize}
The existence of such a function is straightforward.

Observe that $f$ is $\alpha$-H\"older with $\norm{f}_{C^{\alpha}}$ being bounded above by the Lipschitz norm, namely $\norm{f}_{C^{\alpha}} \leq \e_p$. Moreover, we have
\[ \int f d\nu \leq \nu ( B(y_0, 2\e_p) )\leq C \e_p^\eta .\]

Applying  (\ref{eq:exponential decay}) to this function, we get that
\[ \mathbb P \left( \rho ( l_{p}({\bf g})) (x_0 ) \in B (y_0, \e_p ) \right) \leq P^{p} f (x_0) \leq \int f d\nu + C e^{-\beta p} \norm{f}_{C^{\alpha}} , \]
and we conclude that
\[  \mathbb P ( \rho ( l_{p}({\bf g})) (x_0 ) \in B (y_0, \e_p ) ) \leq  C \left(\e_p^\eta +  \frac{ e^{- \beta  p}}{\e_p}\right) .\] If  $\e_n$ decreases sub-exponentially fast, this is smaller than
$C\e_p^K$ for $K=\min( \eta,1)$. If
$ \e_n = \exp (-\gamma n)$  the same holds as soon as $\gamma<\beta$.
\end{proof}

 Let $\e>0$ be  a   real number  which is small with respect to $\chi$ ($\e = \unsur{100}\min(\chi, 1)$ should be enough)  and introduce the integer $m = \lfloor(1 -  {\varepsilon}) n\rfloor$. We divide the composition $l_n ({\bf g})$  into two parts of respective  lengths $m$ and $n-m$: $l_n({\bf g}) = (g_n \ldots g_{n-m+1} ) ( g_m \ldots g_1)$. The first part will have the effect of making $\rho(l_n({\bf g}))$ of large norm (approximately $e^{(\chi + O(\varepsilon ))n}$) with high probability, while the second one will be used to  separate its fixed points by a distance of
the order of magnitude of $\e_n$. %Here and in what follows, the notation $O(\varepsilon)$ stands for  a quantity which is bounded by  $C\varepsilon$, with $C$  independent of $n$.

The quantity
$$ \mathbb P \Big( \Big| \frac{1}{m} \log \norm{ \rho (l_m({\bf g}) )} -\chi \Big| > \varepsilon  \Big) +
\mathbb P \Big( \log \norm{\rho(g_n \ldots g_{m+1}) } > 2 \chi (n-m) \Big) $$ is exponentially small in $n$
by the large deviation estimates for the norm, and the fact that $m \sim (1-\e) n$. Thus, to obtain the
desired  estimate for
$\mathbb{P}\lrpar{\delta(\rho(l_n(\mathbf g)))<\e_n}$  it is enough to estimate the conditional probability
$$\mathbb P \left( \delta(\rho(l_n(\mathbf g)))<\e_n
 \Bigg|\
  \abs{ \frac{1}{m} \log \norm{ \rho (l_m({\bf g}) )} -\chi } \leq  \varepsilon ,\ \mathrm{and} \ \log \norm{\rho(g_n \ldots g_{m+1}) } \leq 2 \chi (n-m) \right).$$

For this,  we let ${\bf h}=(h_1,\ldots , h_m)\in G^{m}$ be such that
\begin{equation} \label{eq:condition at step m} \chi - \varepsilon \leq \frac{1}{m} \log \norm{\rho (l_m({\bf h}))} \leq \chi + \varepsilon,  \end{equation}
and we will prove that the conditional probability
\begin{equation} \label{eq:conditional probability} \mathbb P \left( \delta(\rho(l_n(\mathbf g)))<\e_n \Big|\ g_i = h_i  \ \mathrm{for}\ i= 1,\ldots , m  ,\ \text{and}\ \log \norm{ g_n\ldots g_{m+1} } \leq 2 \chi (n-m) \right)\end{equation}
is bounded by $\e_n^K$ for some $K$, uniformly in ${\bf h}$ satisfying \eqref{eq:condition at step m}. This will give the desired result.

For every ${\bf h}$ satisfying \eqref{eq:condition at step m}, there exist two balls $A_m({\bf h})$ and $R_m({\bf h})$ in the Riemann sphere such that $\rho( l_m ({\bf h} ))  ( R_m({\bf h}) ^c ) = A_m({\bf h})$ and whose diameter are $\sim \frac{1}{\norm {l_m ({\bf h})}}$. Indeed, by the KAK decomposition 
there exist      $R, R' \in \mathrm{SU}(2)$ such that  $\rho(l_m ({\bf h} )) = R\left(\begin{smallmatrix}  {\sigma_m}  & 0 \\ 0 & {\sigma_m^{-1}}  \end{smallmatrix}\right)R'$, where 
%$\sigma_m$ is the spectral radius of $\sqrt{\rho( l_m ({\bf h} ))^*\rho( l_m ({\bf h} ))} $, whence 
$ {\sigma_m}= \norm {l_m ({\bf h})}$. Now $R$ and $R'$ act as Euclidean rotations on the Riemann sphere, therefore we can simply put $R_m({\bf h}) = (R')^{-1}\left(B\left(0,\sigma_m^{-1}  \right)\right)$ and  $A_m({\bf h}) = R\left(B\left(\infty,\sigma_m^{-1}   \right)\right)$.
% it suffices to establish this for the diagonal matrix $\left(\begin{smallmatrix} \norm {l_m ({\bf h})} & 0 \\ 0 & \norm{l_m ({\bf h})}^{-1} \end{smallmatrix}\right)$, and of course in this case it is enough to consider balls of radius $\sim \frac{1}{\norm {l_m ({\bf h})}}$ centered at $0$ and $\infty$, whence the result.
In particular, we see that the radii of  $A_m({\bf h})$ and $R_m({\bf h})$ are bounded by 
$ e^{(-\chi + \varepsilon ) m } = e^{(-\chi + O(\varepsilon) ) n}$.

Let ${\bf g}\in G^{\mathbb N}$ be such that $g_i = h_i$ for $i=1,\ldots ,m$, and $\norm{\rho(g_n \ldots g_{m+1})} \leq 2 \chi (n-m) \sim 2\chi n \e$. Then the ball $ \rho(g_n \ldots g_{m+1})  A_m ({\bf h})$ has diameter bounded by $ e^{(-\chi + O (\varepsilon) ) n}$. Thus,  slightly abusing  notation,  if we set
\[   A_n ({\bf g}): = \rho (g_n \ldots g_{n-m+1}) A_m ({\bf h}),\ \ R_n ({\bf g}) := R_m ({\bf h}),\]
then we have  that
\[ \rho(l_n ({\bf g})) (R_n ({\bf g})^c ) = A_n ({\bf g})\ \ \mathrm{and}\ \ \mathrm{diam}(A_n({\bf g})), \ \mathrm{diam}(R_n({\bf g}) )\leq e^{(-\chi + O(\varepsilon)) n}.\]

We now  claim that if  the sets $A_n({\bf g}) $ and $R_n({\bf g})$ are separated by a distance  $\geq \e_n$, then  $\delta(\rho(l_n(\mathbf g)))>\e_n$. Indeed, in this case the two balls are disjoint and the map $\rho(l_n ( {\bf g}))$ is loxodromic with one fixed point in each ball $A_n ({\bf g})$ and $R_n({\bf g})$.
%
% Thus we can conjugate it to a diagonal matrix $D$ by a matrix of norm bounded by $e^{O(\varepsilon) n}$. This implies $\tr \rho(l_n({\bf g})) = \tr D$, and because $D$ is diagonal,
%\[|\tr \rho (l_n({\bf g})) | = |\tr D| \sim \norm {D} \geq e^{-O(\varepsilon)n}\norm{\rho(l_n({\bf g})} \geq e^{(\chi-O(\varepsilon)) n}.\]
%Also we obviously  have $|\tr \rho(l_n ({\bf g}) ) | \leq \norm{\rho(l_n({\bf g}))} \leq e^{(\chi +O(\varepsilon))n}$, and the claim is proved.

Fix a point $x_0 \in A_m({\bf h})$ and a point $y_0 \in R_m({\bf h})$. If the sets $A_n({\bf g})$ and $R_n({\bf g})$ are not separated by a distance $\e_n$, then (if $n$ is sufficiently large independently of ${\bf h}$) because their diameter is bounded by $e^{-(\chi + O(\varepsilon)) n}$, the point $x_0$ is mapped under $\rho(g_n \ldots g_{m+1})$ into the ball $B(y_0, 2\e_n)$. But by the Markovian property, and Lemma \ref{l:exponential convergence} applied to $p= n-m$, we see that this happens only with probability less that $C \e_p^K$. Since $p= n-m \sim \e n$, from the choice of possible sequences $(\e_n)$ we get that
$\e_p^K\leq C\e_n^{K(\e)}$.
Hence for $n$ large enough we have shown that the probability in \eqref{eq:conditional probability}
is bounded by $C \e_n^{K(\e)}$, for every ${\bf h}$ satisfying \eqref{eq:condition at step m}. The proof is complete.
\end{proof}

We now study fixed points of pairs of words. We could give more precise estimates in the spirit of Theorem \ref{thm:distance} but those will not be needed.

\begin{thm}\label{thm:schottky}
Let $(G,\mu,\rho)$, $(G,\mu',\rho)$ be  two admissible families of representations satisfying the exponential moment condition \eqref{eq: exponential moment condition in the group}.

Fix  $\gamma>0$. Then for
$\mu^\nn\otimes (\mu')^\nn$ a.e. $(\mathbf{g},\mathbf{g'})$, for large enough $n$, $\rho(l_n(\mathbf{g}))$ and $\rho(l_n(\mathbf{g'}))$ are loxodromic transformations, and the mutual distance between  any two of the four associated fixed points is at least $\exp{(-\gamma n)}$.
\end{thm}

\begin{cor}\label{cor:schottky}
Let $(G,\mu,\rho)$, $(G,\mu',\rho)$ be  two admissible families of representations satisfying the exponential moment condition \eqref{eq: exponential moment condition in the group}.

Then for $(\mu^{\nn}\otimes (\mu')^\nn)$-a.e. $(\mathbf{g},\mathbf{g'})\in (G^\nn)^2$, we have that
$$ \unsur{2n}  \log\abs{\tr[\rho(l_n(\mathbf{g})),\rho( l_n(\mathbf{g'}))]-2}  \underset{n\cv\infty}\longrightarrow \chi(\rho,\mu)+\chi(\rho, \mu').$$
\end{cor}

\begin{proof} Fix  a small   $\gamma>0$  (in particular small with respect to the Lyapunov exponents) and take $(\mathbf{g},\mathbf{g'})$ satisfying the conclusion of Theorem \ref{thm:schottky}, and such that moreover $\unsur{n}\log\norm{l_n(\mathbf{g})}$ (resp. $\unsur{n}\log\norm{l_n(\mathbf{g'})}$ is close to $\chi(\rho,\mu)$ (resp. $\chi(\rho, \mu')$).
With notation as in the proof of Theorem \ref{thm:distance} we see that 
$$[\rho(l_n(\mathbf{g})),\rho( l_n(\mathbf{g'}))](A_n(\mathbf{g'})^c)\subset A_n(\mathbf{g}), \text{ while } 
\mathrm{dist}(A_n(\mathbf{g'}), A_n(\mathbf{g'}))\gtrsim e^{-n\gamma}$$  
(recall that the diameter of these balls is of the order of magnitude of $\exp(-n\chi)$). So we infer that $[\rho(l_n(\mathbf{g})),\rho( l_n(\mathbf{g'}))]$ is a loxodromic element with attracting fixed point in $A_n(\mathbf{g})$ and repelling fixed point in 
$A_n(\mathbf{g'})$. Inspecting the contraction of a ball of macroscopic size, disjoint from   $A_n(\mathbf{g'})$ under  $[\rho(l_n(\mathbf{g})),\rho( l_n(\mathbf{g'}))]$ reveals that 
$$\norm{[\rho(l_n(\mathbf{g})),\rho( l_n(\mathbf{g'}))]}\asymp \norm{\rho(l_n(\mathbf{g}))}^2 \norm{\rho(l_n(\mathbf{g}))}^2, $$ and since the fixed points of this commutator are distant from at least $\exp(-\gamma n)$, we conclude from Lemma \ref{l:norm, trace and distance between fixed points} that for large $n$,
$$\abs{\unsur{2n} \log\abs{\tr[\rho(l_n(\mathbf{g})),\rho( l_n(\mathbf{g'}))]-2}- \unsur{2n}\log  \norm{[\rho(l_n(\mathbf{g})),\rho( l_n(\mathbf{g'}))]}}\leq \gamma,$$
which finishes the proof. 
\end{proof}

\begin{proof}[Proof of the theorem]  
Observe first that it suffices to prove the theorem for small $\gamma$. % We will prove that under this condition the probability that $\rho(l_n(\mathbf{g}))$ and $\rho(l_n(\mathbf{g'}))$ are loxodromic transformations, and the mutual distance between  any two of the four associated fixed points is at least $\exp{(-\gamma n)}$, decreases to $0$ exponentially with $n$. 
We will  show that the probability that any two of the  four sets $A_n ({\bf g})$, $R_n({\bf g})$, $A_n({\bf g'})$ and $R_n({\bf g'})$ are closer  in distance  than $\exp(-\gamma n)$ is exponentially small in $n$. Then applying the Borel-Cantelli lemma gives the desired result. 

We first need to prove that for a.e. ${\bf g}\in \mu ^{\mathbb N}$, the sets $R_n ({\bf g})$ tend exponentially fast in distribution to the stationary measure $\check{\nu}$ associated to the inverse random walk. More precisely: 

\begin{lem}\label{lem:repulsive} Almost surely the ball $R_n({\bf g})$ converges to a point $R_{\infty}({\bf g})$ whose distribution is $\check{\nu}$. Moreover,  the probability that the distance between $R_n$ and $R_{\infty}$ is larger than   $  \exp (-\chi n /4)$ is exponentially small in $n$. \end{lem}

\begin{proof} For an element $l$ of $\PSL$, and a constant $D>0$, let us introduce the set $$R^D(\ell) := \set{ [x,y]\in \mathbb P^1\ \big| \ \norm{ \ell (x,y)} \leq D \norm{(x,y)} }.$$ Observe that if $1 \leq D \leq \frac{\norm{\ell}}{2}$, then $R^D (\ell) $ is contained in a $C^{st} \frac{D}{\norm{l}}$-neighborhood of $R^1(\ell)$, and that we can choose $R_n ({\bf g})$ to be $R^1 (\rho (l_n ({\bf g})) )$. These considerations are left to the reader. 

Because the measure $\mu$ has an exponential moment, there is a constant such that 
$$\mu \left( \set{ \norm{\rho (g) } \geq \exp ( \chi n/4)  } \right) \leq C^{st} \exp ( - \chi \tau n /4) $$  ($\tau$ is the constant appearing in~\eqref{eq:exponential moment}). Therefore, if $E_n\subset G^{\mathbb N}$ is defined by 
$$E_n = \set {\mathbf g \in G^\nn,\ \forall m\geq n, \  \norm {\rho(g_m)} \leq \exp (\chi m/4) \text{ and }
\norm{\rho ( l_m ( {\bf g} ))} \geq \exp(\chi m /2)},$$
% the set of elements ${\bf g}$ such that \[ \norm {\rho(g_m)} \leq \exp (\chi m/4)\ \ \ \mathrm{and} \ \ \ \norm{\rho ( l_m ( {\bf g} ))} \geq \exp(\chi m /2) \] for every $m\geq n$,
then the measure of $E_n$ is exponentially close to $1$ (we use the large deviations estimate for the norm).

Now pick  ${\bf g}\in E_n$ 	and let $m\geq n$. Then, on $R_m ({\bf g}) = R^1 ( \rho(l_m({\bf g})))$, we have $\norm { \rho(l_{m+1} ({\bf g})) } \leq C^{st} \norm{\rho(g_{m+1})} \norm { \rho(l_{m} ({\bf g})) }\leq \exp (\chi m/4)$ so that $R_{m+1}({\bf g}) $ is contained in the $C^{st} \exp (- \chi m /4)$-neighborhood of $R_m ({\bf g})$. We deduce that $R_m ({\bf g})$ converges to a point $R_{\infty}({\bf g})$ for every ${\bf g} \in E_n$, and hence a.s. by Borel-Cantelli, and that the distance between $R_n({\bf g})$ and $R_{\infty}({\bf g})$ is bounded by $A\exp( - \chi n /4 )$ for ${\bf g} \in E_n$ for some constant $A$. Of course we can adjust $A=1$ by introducing appropriate constants in the above reasoning.

To finish the proof, it suffices to verify that the distribution of $R_{\infty}$ is given by the stationary measure $\check{\nu}$. For this, we argue as in Step 2 of the proof of Theorem \ref{thm:support}, by observing  that  the point  $R_{\infty} = \lim R ( l_n ({\bf g}) )$ only depends on the tail of the sequence $l = ( l_n( {\bf g} ) ) $ -- or equivalently of $r = ( r_n (\check{\bf g}) )$, with $\check{\bf g} = (g_n^{-1})$ -- and satisfies the equivariance property~\eqref{eq: equivariance}, with respect to the inverse process. Thus, $R_{\infty}$ defines a map from the Poisson boundary $P(G,\check{\mu})$ to $\mathbb P^1$ which is $\rho$-equivariant. The distribution of the point $R_{\infty}$ is hence given by the stationary measure $\check{\nu}$. 
%(in fact, with notation as in \eqref{eq: equivariance} it is not difficult to prove that $R_{\infty}({\bf g})$ coincides with $\check{\theta} (\check{\bf g})$).
\end{proof}

With this at hand, let us conclude the proof of the theorem.  Fix $\gamma<\chi/2$. Recall from Theorem \ref{thm:distance}, that with probability exponentially close to 1, 
$\norm {l_n({\bf g})} \geq \exp ((\chi -\varepsilon) n)$ and 
  $A_n({\bf g})$ and $R_n({\bf g})$ are $\exp(-\gamma n)$-separated. Fix a pair of such elements $\mathbf g, \mathbf g'$, and let us show that 
 the probability that the balls $A_n ({\bf g})$ and $R_n({\bf g})$ are not $\exp(-\gamma n )$-separated from both $A_n ({\bf g'})$ and $R_n ({\bf g'})$ is exponentially small.  
 
  Let us first estimate the probability that $R_n ({\bf g})$ intersects the $\exp(-\gamma n)$-neighborhood of $A_n({\bf g'})\cup R_n({\bf g'})$. This implies that the point $R_{\infty} ({\bf g})$ is $C^{st} \exp (-\gamma n ) $-close to $A_n ({\bf g'})\cup R_n({\bf g'})$, hence belongs to a union of two fixed balls of radii $C^{st} \exp (-\gamma n)$ (recall that  $\gamma < \chi /2$). By Lemma \ref{lem:repulsive} and 
  the H\"older regularity property of the  stationary measure, we conclude that this event happens with probability bounded by  $C^{st} \exp(-\gamma \eta n)$  for some  $\eta>0$. 
  
 To bound the probability that $A_n({\bf g})$ intersects the $\exp (-\gamma n)$-neighborhood of $A_n ({\bf g'})\cup R_n ({\bf g'})$, we simply reverse the  random walk on $G$, and use the fact that the distribution of $A_n({\bf \check{g}})$ is the same as that  of $ R_n({\bf g})$ .
 \end{proof}

 \section{A number-theoretic estimate}\label{app:number}

 Our purpose here is to prove the following result. We thank P. Philippon for explaining it to us\footnote{We were informed by S. Boucksom  that in case $V$ is defined over $\mathbb{Q}$,
\cite[Lemma 2.6]{boucksom-chen} gives the same result with   the term $\deg(P)\log\deg(P) $ replaced by $\deg(P)$.}.

 \begin{prop}\label{prop:philippon}
 Let $V$ be an irreducible affine algebraic variety in $\cc^n$, defined over $\overline{\mathbb{Q}}$, and
 $U\Subset V$ be a relatively compact open subset. Then there exists a constant $\delta>0$ (depending on $V$, $U$, $n$) such that if $P\in \zz[X_1, \cdots X_n]$ is any polynomial with integer coefficients, then
 \begin{itemize}
 \itm either $P\rest{V} = 0$
 \itm or $\sup_U\abs{P}\geq \delta^{\deg(P)\log\deg(P) + \log H(P)}$, where $H(P)$ is the maximum  modulus  of the coefficients of $P$.
 \end{itemize}
 \end{prop}
 
 Here is the precise corollary that we need.

 \begin{cor}\label{cor:philippon}
 Let $V$,  $U$, $P$ be as in Proposition \ref{prop:philippon}. Then there exists a constant $\delta>0$ such that the following alternative holds
 \begin{itemize}
 \itm either $P\rest{V}$ is constant
 \itm or $\mathrm{var}_U P \geq \delta^{\deg(P)\log\deg(P) + \log H(P)}$, where $ \mathrm{var}_U {P} =
 \sup_{x,y\in U} \abs{P(x)-P(y)}$.
 \end{itemize}
 \end{cor}

 To obtain the corollary, it is enough to apply the proposition to $\widetilde{P}:(x,y)\mapsto P(x)-P(y)$ restricted to $\widetilde{V} :=V\times V\subset \cc^n_x\times \cc^n_y$.
 
 \medskip

 For the matter of proving the proposition we briefly introduce a few concepts from number theory; the reader is referred to   \cite{waldschmidt} for
 details. If $P\in \zz[X_1, \cdots X_n]$, the {\em usual height} $H(P)$ of $P$ is the maximum modulus of its coefficients. Let now $\alpha\in \overline{\mathbb{Q}}$ be an algebraic number, and  $P\in \zz[X]$ be its   minimal polynomial. By definition,  the degree $\deg(\alpha)$ equals $\deg(P)$ and we let   $H(\alpha) := H(P)$.
 We do not need to define precisely the \textit{height} $h(\alpha)$ of $\alpha$, but only note that is satisfies
 $ \abs{h(\alpha) - \unsur{\deg(\alpha)}\log H(\alpha)}\leq C$, where  $C$ is a universal constant $\leq 2$. If $\alpha = p/q$ is rational, $h\lrpar{\alpha} = \log\max({\abs{p}, \abs{q}})$. Also $h(\cdot)$ behaves well under the operation of taking sums and products of algebraic numbers.

 Another useful property is that if $P_1$ $P_2$ are polynomials in $\zz[X]$ with respective degree $d_1$, $d_2$, then $H(P_1P_2)
 \geq 2^{-d_1d_2}(d_1d_2+1)^{-1/2}  H(P_1)H(P_2)$. From this we infer that if $\alpha\in \overline{\mathbb{Q}}$, and $P\in\zz[X]$ is any polynomial such that $P(\alpha) =0$, then there exists a constant $C=C(\deg(P))$ depending only on $\deg(P)$ such that
 $h(\alpha)\leq   \log H(P)+C$.
 Furthermore, if now $\alpha$ satisfies an algebraic equation of the form $P(\alpha)= 0$, where $P\in \mathbb{Q}(\beta)[X]$, with $\beta\in \overline{\mathbb{Q}}$,  then $h(\alpha)\leq  {C(\deg(P), \beta)}(h(P)+1)$, where $h(P)$ denotes the maximum height of the coefficients of $P$. To see this, just observe that the product of the Galois conjugates  of $P$   is an annihilator  of $\alpha$ belonging to $\mathbb{Q}[X]$ and estimate its degree and coefficients.
 
 \medskip
 
 When $V$ is merely a point, the estimate in Proposition \ref{prop:philippon}
  is  classical and known as the  \textit{Liouville  inequality} (see \cite[Proposition 3.14]{waldschmidt}). We need to state it precisely: if $x_1, \ldots, x_n$ are algebraic numbers, and $P\in \zz[X_1, \ldots X_n]$ is a polynomial not vanishing at $x = (x_1, \ldots, x_n)$, then there exists a constant $c$ depending only on the dimension $n$ such that
 \begin{equation}\label{eq:liouville}
 \abs{P(x_1, \ldots , x_n )}\geq e^{-cD (\deg(P)\max_i h(x_i) + \log H(P))}, \text{ where } D = [\mathbb{Q}(x_1, \ldots, x_n):\mathbb{Q}].
 \end{equation}
 
 \begin{proof}[Proof of the proposition] Throughout the proof the notation $a\lesssim b$   means $a\leq Cb$ where $C$ is a constant independent of $P$ (and similarly for  $a\gtrsim b$).
 The main step is to  prove that if $P\rest{V} \neq 0$, then there exists an algebraic point $x = (x_1,\ldots ,x_n)\in V$, such that $P$ does not vanish at $x$ and furthermore
 $\max_i\deg(x_i)\lesssim 1$ and $\max_i h(x_i)\lesssim \log\deg(P)$. Then, applying the Liouville inequality
 \eqref{eq:liouville} to $\abs{P(x)}$ gives the result.
 
 To show the existence of such a point $x$, we use a projection argument.
 We fix a  linear projection $\pi: \cc^n\cv \cc^{\dim V}$, defined over $\mathbb{Q}$,  in general position with respect to $V$.
 Let $\Sigma = \set{P=0}\cap V$, which is a proper subvariety of $V$.
 The projection $\pi(\Sigma)\subset \cc^{\dim(V)}$ is a hypersurface of degree at most $\deg(V)\deg(P)$.
 
 Let $k = \dim(V)$ and
  let $B\subset \cc^{k}$ be a ball contained in $\pi(U)$.
 We claim that there exists a rational point $y\in \mathbb{Q}^{k}$ such that
 $y\in B \setminus \pi(\Sigma)$ and $h(y)\lesssim\log\deg(P)$ (abusing slightly, here we put
 $h(y) = \max_i h(y_i)$, where the $y_i$ are the coordinates of $y$). If $k = 1$ this is obvious since
 $\pi(\Sigma)$ contains at most  $\deg(V)\deg(P)$ points while
 $\# \set{y\in B\cap \mathbb{Q}, \ h(y)\leq h}\gtrsim e^{h}.$ Now if $k=2$, $\pi(\Sigma)$ contains at most $\deg(V)\deg(P)$ lines, so there is a line over $\mathbb{Q}$ defined by an equation of height $\lesssim \log(\deg(V)\deg(P))$ not contained in $\pi(\Sigma)$, and in this line we are back to the previous case $k=1$.  The general case follows by induction.
 
 \medskip
 
 Finally, to obtain the desired $x$ we simply lift $y$ to $V$. Therefore, $x$ is an intersection point between $V$ and the fiber $\pi^{-1}(y)$ of the projection $\pi$ through $y$,  which is 
 a  $(n-k)$-plane parallel to some fixed rational direction and passing through  $y$. To estimate the degree and height of $x$, we work in a projective space $\pp^n$ compactifying $\cc^n$. Since $\pi$ is a linear projection 
defined over $\mathbb{Q}$, Bézout's theorem implies that $\deg(x)\leq\deg(V)\deg(\pi^{-1} (y)) = \deg(V)\lesssim 1$.  Similarly, there are Bézout-type theorems for the height of intersections of projective varieties (see \cite{bgs} or  
\cite[Thm 3]{philippon}). In our case, it expresses as 
$$h(V\cdot \pi^{-1} (y))\leq h(V) \deg(\pi^{-1} (y)) + h(\pi^{-1} (y))\deg(V) + c \deg(V)\deg(\pi^{-1} (y)),$$ where $c$ is a dimensional constant. Here $V\cdot \pi^{-1} (y)$ is a 0-dimensional cycle containing $x$ so that $h(x)\leq h(V\cdot \pi^{-1} (y))$. The precise definition of the height of an  algebraic subvariety is delicate, and differs slightly among authors. Fortunately, these   definitions differ from at most an additive constant depending on the dimension. To fix the ideas let us say that we define $h$ according to \cite{bgs}. 
Since $V$ is a fixed variety, to obtain the desired estimate on $h(x)$ we just need  to check that 
$h (\pi^{-1} (y))\lesssim \log\deg(P)$. 

To see this, we simply note that the height   of a projective subspace 
  is the height of its image under the Pl\"ucker embedding of the corresponding  Grassmanian (see the remarks about Proposition 4.1.2. in \cite{bgs}). 
Here $\pi^{-1} (y)$, viewed as a projective subspace in $\pp^n$, lifts to a linear subspace of dimension $n-k+1$ in $\cc^{n+1}$. We can choose  a basis  $v_1,\cdots, v_{n-k+1}$ made of vectors of height $\lesssim h(y)$, and the Pl\"ucker image of $\pi^{-1} (y)$ is $v_1\wedge \cdots \wedge v_{n-k+1}$. Since the height is subadditive under multiplication, we see that the  coordinates of this vector are $\lesssim h(y)\lesssim \log\deg(P)$, therefore 
$h(\pi^{-1} (y))\lesssim \log\deg(P)$, which finishes the proof.
  \end{proof}

%  The simplest case is when $\dim(V)=n-1$. Let then $R$ be a polynomial such that $V = \set{R=0}$, belonging to $K[X]$, where $K$ is some number field. In an adapted system of rational coordinates (depending only on $V$), $x$ expresses as $(x_1, \cdots, x_{n-1}, x_n)$, where for $i\leq n-1$, $y_i = x_i$, and $x_n$ is a root of $f(t) = 0$, where $f$ is the polynomial with coefficients in $K$ defined by $f(t)=  R(x_1, \cdots, x_{n-1}, t)$.  We see that $\deg(f)\leq \deg (R) \lesssim 1$. Furthermore, the height $h(a_i)$ of the coefficients of $f$ satisfies $h(a_i)\lesssim h(y)\lesssim  \log(\deg(P))$.  By the observations preceding the proof, $x$ satisfies the desired estimate.
%   
%  In the more delicate situation where $\dim(V) =k$ is arbitrary, $x$ is obtained as an  intersection point      between  $V$  and a linear subspace $H(y)$ of dimension $n-k$.
%  The linear subspace $H(y)$  is simply the fiber $\pi^{-1}(\set{y})$, that is,
%  a  $(n-k)$-plane parallel to some fixed rational direction and passing through  $y$.
%  There exist Bézout-type estimates   for the degree and height of these intersection points, based on elimination theory (see \cite[Thm 3]{philippon}).\note{reprendre} From these,
%  we obtain that  $x$ has degree $\leq \deg(V) \lesssim 1$ and height $h(x)\lesssim h(y) \lesssim \log\deg(P)$. The  proof is complete.

\end{document}